\documentclass[11pt,a4paper]{amsart}

\usepackage{graphicx,epsf}
\usepackage[svgnames]{xcolor}
\usepackage{mst-stylefile}
\usepackage{amssymb,amsmath,harvard}
\usepackage{stmaryrd}
\usepackage{amsfonts}
\usepackage{multirow}
\usepackage[normalem]{ulem}
\usepackage{subfigure}
\usepackage{epstopdf}
\usepackage{hhline}

\usepackage{amsfonts}

\newcommand{\NOR}{\mathcal{N}}

\newcommand{\ccoef}{{\mathsf r}}

\newcommand{\EE}[1]{\mathbb{E}\!\left[#1\right]}

\newcommand{\COV}[2]{\textnormal{Cov}\left[#1,#2\right]}
\newcommand{\VAR}[1]{\text{Var}\left[#1\right]}
\newcommand{\Conv}{\mathop{\scalebox{1.2}{\raisebox{-0.02ex}{$\ast$}}}}
\newcommand{\CONV}{\mathop{\scalebox{2}{\raisebox{-0.3ex}{$\ast\;$}}}}

\newcommand{\Cov}{\textnormal{Cov}}
\newcommand{\vecoper}{\textnormal{vec}}

\newcommand{\bbZ}{{\Bbb Z}}
\newcommand{\bbR}{{\Bbb R}}
\newcommand{\real}{{\Bbb R}}
\newcommand{\bbN}{{\Bbb N}}
\newcommand{\bbC}{{\Bbb C}}

\newcommand{\signif}{{\mathsf{s}}}

\newcommand{\e}{{\mathrm e}}
\def\va{\makebox{Var}}

\def\cov{\makebox{Cov}}

\graphicspath{{fig/}}
\graphicspath{{./figures/}}
\DeclareGraphicsExtensions{.pdf}

\newcommand{\diag}{\textnormal{diag}}

\DeclareMathOperator{\var}{{\mathrm{Var}}}

\usepackage{cite}

\graphicspath{
{./FCFigPerf/}
}

\title{Multivariate Hadamard self-similarity: \\
 testing fractal connectivity}

\usepackage[foot]{amsaddr}

\begin{document}

\author{Herwig~Wendt$^1$, Gustavo~Didier$^2$, S\'{e}bastien~Combrexelle$^1$, Patrice~Abry$^3$%
}
%
\thanks{P.A.\ and H.W. are supported in part by the French ANR grant MultiFracs. G.D.\ was supported in part by the ARO grant W911NF-14-1-0475. G.D.\ gratefully acknowledges the support of ENS de Lyon for his visits, and thanks Alexandre Belloni for the insightful mathematical discussions.}

\address{$^1$ IRIT-ENSEEIHT, CNRS (UMR 5505), Universit\'{e} de Toulouse, France, {\tt herwig.wendt@irit.fr}}

\address{$^2$ Mathematics Department, Tulane University, New Orleans, USA, {\tt gdidier@tulane.edu}}

\address{$^3$ Universit\'{e} de Lyon, ENS de Lyon, Universit\'{e} Claude Bernard, CNRS,  Laboratoire de Physique, F-69342 Lyon, France, {\tt patrice.abry@ens-lyon.fr}}

\date{}

\maketitle

\begin{abstract}
While scale invariance is commonly observed in each component of real world multivariate signals, it is also often the case that the inter-component correlation structure is not fractally connected, i.e., its scaling behavior is not determined by that of the individual components. To model this situation in a versatile manner, we introduce a class of multivariate Gaussian stochastic processes called Hadamard fractional Brownian motion (HfBm). Its theoretical study sheds light on the issues raised by the joint requirement of entry-wise scaling and departures from fractal connectivity. An asymptotically normal wavelet-based estimator for its scaling parameter, called the Hurst matrix, is proposed, as well as asymptotically valid confidence intervals. The latter are accompanied by original finite sample procedures for computing confidence intervals and testing fractal connectivity from one single and finite size observation. Monte Carlo simulation studies are used to assess the estimation performance as a function of the (finite) sample size, and to quantify the impact of omitting wavelet cross-correlation terms. The simulation studies are shown to validate the use of approximate confidence intervals, together with the significance level and power of the fractal connectivity test. The test performance and properties are further studied as functions of the HfBm parameters.
\end{abstract}


%


\section{Introduction}
\label{sec:intro}

\subsection{Scale Invariance}

The relevance of the paradigm of scale invariance is evidenced by its successful use, over the last few decades, in the analysis of the dynamics in data obtained from a rather diverse spectrum of real world applications. The latter range from natural phenomena -- physics (hydrodynamic turbulence \cite{Mandelbrot1974}, out-of-equilibrium physics), geophysics (rainfalls), biology (body rhythms \cite{Lopes2009}, heart rate \cite{Akselrod1987,Kiyono:2006}, neurosciences and genomics \cite{CIUCIU:2012:A,CIUCIU:2014:A,He2010,He2014}) -- to human activity -- Internet traffic \cite{Abry1998,Abry2002}, finance \cite{mandelbrot99}, urban growth and art investigation \cite{ABRY:2013:B,spmag2008,AbrySPM2015}.

In essence, scale invariance -- also called scaling, or scale-free dynamics -- implies that the phenomenical or phenomenological dynamics are driven by a large continuum of equally important time scales, rather than by a small number of characteristic scales. 
Thus, the investigation's focus is on identifying a relation amongst relevant scales rather than picking out characteristic scales.

Historically, self-similarity was one of the first proposed mathematical frameworks for the modeling of scale invariance (e.g., \cite{Samorodnitsky1994}). A random system is called self-similar when dilated copies of a single signal $X$ are statistically indistinguishable, namely,
\begin{equation}
\label{eq:oss}
 \{X(t)\}_{ t \in \real}  \stackrel{\textnormal{fdd}}{=} \{a^H X(t/a)\}_{t \in \real}, \quad \forall a > 0,
 \end{equation}
 where $\stackrel{\textnormal{fdd}}{=} $ stands for the equality of finite-dimensional distributions. An example of a stochastic process that satisfies the property \eqref{eq:oss} is fractional Brownian motion (fBm). Indeed, the latter is the only self-similar, Gaussian, stationary increment process, and it is the most widely used scaling model for real-world signals \cite{taqqu:2003}.

Starting from \eqref{eq:oss}, the key parameter for quantifying scale-free dynamics is the scaling, or Hurst, exponent $0 <H <1$. The estimation of $H$ is the central task in scaling analysis, and it has received considerable effort and attention in the last three decades (see \cite{Bardet2003a} for a review). The present contribution is about wavelet-based estimation \cite{flandrin:1992,veitch:abry:1999}. It relies on the key scaling property
 \begin{equation}
 \label{equ:wavcoef}
\frac{1}{T} \sum_t T_X^2(a,t) \simeq C a^{\alpha}, \quad \alpha := 2H,
 \end{equation}
where $T_X(a,t)$ is the wavelet coefficient of an underlying self-similar stochastic process and $T$ is the number of available coefficients. In other words, the sample wavelet variance of the stochastic process behaves like a power law with respect to the scale $a$.

 \subsection{Multivariate scaling}

In many modern fields of application such as Internet traffic and neurology, data is collected in the form of multivariate time series. Univariate-like analysis in the spirit of \eqref{equ:wavcoef} -- i.e., independently on each component -- does not account for the information stemming from correlations across components. The classical fBm parametric family, for example, provides at best a model for component-wise scaling, and thus cannot be used as the foundation for a multivariate modeling paradigm.

To model self-similarity in a multivariate setting, a natural extension of fBm, called Operator fractional Brownian motion (OfBm), was recently defined and studied (see \cite{amblard:coeurjolly:2011,didier:pipiras:2011,Coeurjolly_J-F_2013_ESAIM_wamfbm}). An OfBm ${\underline X}$ satisfies the $m$-variate self-similarity relation
\begin{equation}
\label{equ:ssmulti}
 \{\underline{X}(t)\}_{ t \in \real}  \stackrel{\textnormal{fdd}}{=} \{a^{\underline{\underline{H}}} \underline{X}(t/a)\}_{t \in \real}, \quad \forall a > 0,
 \end{equation}
 where the scaling exponent is a $m \times m$ matrix $\underline{\underline{H}} $, and $a^{\underline{\underline{H}}}$ stands for the matrix exponential $\sum^{\infty}_{k=0}(H\log a)^k/k!$. Likewise, the wavelet spectrum of each individual component is not a single power law as in \eqref{equ:wavcoef}; instead, it behaves like a mixture of distinct univariate power laws. In its most general form, OfBm remains scarcely used in applications; recent efforts have tackled many difficulties that arise in the identification of its parameters \cite{didier2015demixing,abry:didier:2016}.

\subsection{Entry-wise multivariate scaling}

We call an OfBm entry-wise scaling when the Hurst parameter is simply a diagonal matrix $\underline{\underline{H}} = \textnormal{diag}(H_1, \ldots, H_m)$. This instance of OfBm has been used in many applications (e.g., \cite{achard2008fractal,CIUCIU:2014:A}) and its estimation is thoroughly studied in \cite{amblard:coeurjolly:2011}. Since $\underline{\underline{H}}$ is diagonal, the relation \eqref{equ:ssmulti} takes the form
 \begin{equation}
\label{equ:ssmultis}
 \{X_1(t), \ldots, X_m(t) \}_{ t \in \real}  \stackrel{\textnormal{fdd}}{=} \{a^{H_1}X_1(t/a), \ldots, a^{H_m}X_m(t/a)\}_{t \in \real}, \quad \forall a > 0,
 \end{equation}
 which is reminiscent of the univariate case. This implies that the extension of \eqref{equ:wavcoef} to all auto- and cross-components of $m$-variate data can be written as
 \begin{equation}
 \label{equ:wavcoefm}
\frac{1}{T}\sum_t T_{X_{q_1}}(a,t)T_{X_{q_2}}(a,t)  \simeq C a^{\alpha_{q_1 q_2}},  \quad \alpha_{q_1 q_2} := H_{q_1}+H_{q_2}, \quad q_1,q_2 = 1,\hdots,m,
 \end{equation}
 where $T$ is as in \eqref{equ:wavcoef}.

\subsection{Fractal connectivity}

Yet, entry-wise scaling OfBm is a restrictive model since the cross-scaling exponents $\alpha_{q_1 q_2}$, $q_1 \neq q_2$, are determined by the auto-scaling exponents $\alpha_{q_1 q_1}$ and $\alpha_{q_2 q_2}$, i.e.,
\begin{equation}\label{e:FC}
\alpha_{q_1 q_2} = H_{q_1}+H_{q_2} = (\alpha_{q_1 q_1} + \alpha_{q_2 q_2})/2.
\end{equation}
In this situation, called \emph{fractal connectivity} \cite{achard2008fractal,Wendt2009icassp,CIUCIU:2014:A}, no additional scaling information can be extracted from the analysis of \textit{cross-components}. However, in real world applications, cross-components are expected to contain information on the dynamics underlying the data, e.g., cross-correlation functions \cite{kristoufek2013mixed,kristoufek2015can}. As an example, recent investigation of multivariate brain dynamics in \cite{CIUCIU:2014:A} produced evidence of departures from fractal connectivity, notably for subjects carrying out prescribed tasks.
This means that there is a clear need for more versatile models than entry-wise scaling OfBm (see also Remark \ref{r:increments} below). The covariance structure of the new model should satisfy the following two requirements:
\begin{enumerate}
\item all auto- and cross-components are (approximately) self-similar;
\item departures from fractal connectivity are allowed, i.e., the exponents of the cross-components are not necessarily determined by the exponents of the corresponding auto-components.
\end{enumerate}
\noindent Hereinafter, a departure from fractal connectivity \eqref{e:FC} on a given covariance structure entry $(q_1,q_2)$ will be quantified by means of the parameter
\begin{equation}\label{e:delta_kn}
\delta_{q_1 q_2} = \frac{\alpha_{q_1 q_1} + \alpha_{q_2 q_2}}{2} - \alpha_{q_1 q_2} \geq 0, \quad q_1, q_2 = 1,\hdots,m,
\end{equation}
where nonnegativeness is a consequence of the Cauchy-Schwarz inequality (see \eqref{e:hq1q2} below). It is clear that $\delta_{q_1 q_2} = 0$ when $q_1 = q_2$.

\subsection{Goals, contributions and outline}

Our contribution comprises four main components. First, we propose a new class of multivariate Gaussian stochastic processes, called \emph{Hadamard fractional Brownian motion} (HfBm), that combines scale-free dynamics and potential departures from fractal connectivity. Moreover, we provide a precise discussion of the issues entailed by the presence of these two properties (Section~\ref{sec:model}). Second, we study the multivariate discrete wavelet transform (DWT) of HfBm, define wavelet-based estimators for the scaling exponents $\alpha_{q_1 q_2} $ and the fractal connectivity parameter $ \delta_{q_1 q_2}$, mathematically establish their asymptotic performance (i.e., asymptotic normality and covariance structure), and computationally quantify finite sample size effects (Section~\ref{sec:wavestim}). Third, starting from a single sample path, we construct approximate confidence intervals for the proposed estimators. The procedure is backed up by the mathematical identification of the approximation orders as a function of the sample path size and the limiting coarse scale. This is further investigated by means of Monte Carlo simulations, as well as by means of a study of the ubiquitous issue of the impact of (partially) neglecting the correlation amongst wavelet coefficients (Section~\ref{sec:MC}). Beyond being of interest in multivariate modeling, the study sheds more light on the same issue for the univariate case. Fourth, we devise an efficient test for the presence of fractal connectivity from a single sample path. In addition, we assess the finite sample performance of the test in terms of targeted risk by means of Monte Carlo simulations (Section~\ref{sec:test}). Finally, routines for the synthesis of HfBm, as well as for estimation, computation of confidence intervals and testing will be made publicly available at time of publication. All proofs can be found in the Appendix.

\section{Hadamard fractional Brownian motion}
\label{sec:model}

For Hadamard fractional Brownian motion, defined next, the fractal connectivity relation \eqref{e:FC} does not necessarily hold.
\begin{definition}\label{def:HfBm}
A Hadamard fractional Brownian motion $B_H = \{B_H(t)\}_{t \in \bbR}$ (HfBm) is a proper, Gaussian (stationary increment) process whose second moments can be written as
\begin{equation}\label{e:Hadamard_harmonizable}
\EE{B_{H}(s)B_{H}(t)^*}
= \int_{\bbR} \Big( \frac{e^{isx}-1}{ix}\Big)\Big( \frac{e^{-itx}-1}{-ix}\Big) f_{H}(x) dx , \quad s,t \in \bbR,
\end{equation}
For $0 < h_{\min}\leq h_{\max} < 1$, the matrix exponent $H = \Big(h_{q_1 q_2}\Big)_{q_1,q_2 = 1,\hdots,m}$ satisfies the conditions
\begin{equation}\label{e:HfBm_exponent_H}
h_{q_1q_2} \in [h_{\min},h_{\max}], \quad q_1, q_2 = 1,\hdots,m.
\end{equation}
The matrix-valued function $f$ is a spectral density of the form
\begin{equation}\label{e:HfBm_specdens}
f_H(x)_{q_1 q_2} = \Big(\rho_{q_1 q_2}\sigma_{q_1}\sigma_{q_2}|x|^{- 2(h_{q_1 q_2} - 1/2)}\Big)g_{q_1 q_2}(x), \quad q_1, q_2=1,\hdots,m,
\end{equation}
i.e., the Hadamard scaling parameters are given by
\begin{equation}\label{e:alphail}
\alpha_{q_1q_2} = 2 h_{q_1 q_2}, \quad q_1,q_2=1,\hdots,m,
\end{equation}%
where $\rho_{q_1 q_2} \in [-1,1]$, $\sigma_{q_1}, \sigma_{q_2} \in\bbR^+$. The real-valued functions $g_{q_1 q_2} \in C^2(\bbR)$ satisfy 
\begin{equation}\label{e:g_is_bounded}
\max_{l=0,1,2}\sup_{x \in \bbR} \Big|\frac{d^l}{dx^l}g_{q_1 q_2}(x)\Big| \leq C,
\end{equation}
\begin{equation}\label{e:g_is_Schwartz}
\Big|\frac{d^l}{dx^l}(g_{q_1 q_2}(x) - 1) \Big| \leq C' |x|^{\varpi_0-l}, \quad x \in (-\varepsilon_0, \varepsilon_0), \quad l = 0,1,2,
\end{equation}
for constants $C, C', \varepsilon_0 > 0$, where
\begin{equation}\label{e:hmax<delta0=<2(1+hmin)}
2 h_{\max} < \varpi_0 \leq  2(1 + h_{\min}).
\end{equation}
\end{definition}

\begin{example}
An HfBm with parameters $h_{q_1 q_2} = (h_{q_1 q_1} + h_{q_2 q_2})/2$ (i.e., fractally connected), $g_{q_1 q_2}(x)\equiv 1$, $q_1, q_2=1,\hdots,m$, is an entry-wise scaling OfBm with diagonal Hurst matrix $H =\textnormal{diag}(h_1,\ldots,h_m)$ (see \cite{amblard:coeurjolly:2011,didier:pipiras:2011,Coeurjolly_J-F_2013_ESAIM_wamfbm}).
\end{example}

By the known properties of spectral densities, $f_H(x) = (f_H(x)_{q_1 q_2})_{q_1,q_2=1,\hdots,m} \in {\mathcal S}_{>0}(m,\bbR)$ (symmetric positive semidefinite) a.e.\ and satisfies
\begin{equation}\label{e:CS}
|f_H(x)_{q_1 q_2}| \leq \sqrt{f_H(x)_{q_1 q_1}}\sqrt{f_H(x)_{q_2 q_2}} \quad dx\textnormal{-a.e.}
\end{equation}
(\cite{brockwell:davis:1991}, p.436). The relation \eqref{e:CS} further implies that
\begin{equation}\label{e:hq1q2}
h_{q_1 q_2} \leq \frac{h_{q_1 q_1} + h_{q_2 q_2}}{2}, \quad q_1, q_2= 1,\hdots,m.
\end{equation}
Whenever convenient we also write
$$
h_{q q} = h_q, \quad \alpha_{qq} = \alpha_q, \quad q = 1,\hdots,m.
$$

The name ``Hadamard" comes from Hadamard (entry-wise) matrix products. If one rewrites HfBm componentwise as $B_H(t) = (B_{H,1}(t),\hdots,B_{H,m}(t))^*$, then the conditions \eqref{e:Hadamard_harmonizable}, \eqref{e:HfBm_exponent_H} and \eqref{e:g_is_Schwartz} yield the asymptotic equivalence
$$
\EE{B_{H,q_1}(cs)B_{H,q_2}(ct)} \sim \frac{\varsigma}{2} \{|cs|^{2h_{q_1 q_2}}+|ct|^{2h_{q_1 q_2}} - |c(s-t)|^{2h_{q_1 q_2}}\}, \quad c \rightarrow \infty,
$$
$q_1, q_2 = 1,\hdots,m$, for some $\varsigma \in \bbR$.
In other words, over large scales, the covariance between each pair of entries of an HfBm approaches that of a univariate fBm, up to a change of sign (see also Proposition \ref{p:4th_moments_wavecoef}, ($iii$)). In this sense, for large $c$, an HfBm behaves like its ideal counterpart $B_{H,\textnormal{ideal}} = \{B_{H,\textnormal{ideal}}(t)\}_{t \in \bbR}$, defined as a generally non-existent stochastic process satisfying the also ideal Hadamard (entry-wise) self-similarity relation
\begin{equation}\label{e:Hadamard_ss}
\EE{B_{H,\textnormal{ideal}}(cs)B_{H,\textnormal{ideal}}(ct)^*} = c^{\circ H} \circ \EE{B_{H,\textnormal{ideal}}(s)B_{H,\textnormal{ideal}}(t)^*}, \quad c > 0,
\end{equation}
where $\circ$ denotes the Hadamard (entry-wise) matrix product and $c^{\circ H} := \Big( c^{h_{q_1 q_2}} \Big)_{q_1, q_2=1,\hdots,m}$. The process $B_{H,\textnormal{ideal}}$ can be viewed as a heuristic tool for developing intuition on multivariate self-similarity. Mathematically, though, it can only exist in fractally connected instances, the reason being that distinct (spectral) power laws cross over. Indeed, we must have $\alpha_{q_1 q_2} \leq \frac{\alpha_{q_1 q_1} + \alpha_{q_2 q_2}}{2}$ for $x$ close to 0 and $\alpha_{q_1 q_2} \geq \frac{\alpha_{q_1 q_1} + \alpha_{q_2 q_2}}{2}$ for large $|x|$, whence $\alpha_{q_1 q_2} \equiv \frac{\alpha_{q_1 q_1} + \alpha_{q_2 q_2}}{2}$.
This shows that HfBm is a perturbation of its virtual counterpart, where the regularization functions $g_{q_1 q_2}(x) $ in \eqref{e:HfBm_specdens} introduce high-frequency corrections.

\begin{example}
\label{ex:A}
An illustrative subclass of HfBm is obtained by setting $g_{qq}(x) \equiv 1 = \sigma_{q}$, $q=1,\hdots,m$, and $g_{q_1 q_2}(x) = e^{-x^2}$, $q_1 \neq q_2$. Note that $g_{q_1 q_2}(\cdot)$ satisfies \eqref{e:hmax<delta0=<2(1+hmin)} with $\varpi_0 = 2$. In this case, the expression for the main diagonal spectral entry of an HfBm is identical to that of an ideal-HfBM, and the difference lies on the off-diagonal entries:
\begin{equation}\label{e:regHfBm_condition_on_the_det}
\rho_{q_1q_2}^{2} |x|^{-2\alpha_{q_1q_2}} e^{-x^2} \leq \rho_{q_1q_1}\rho_{q_2q_2}|x|^{-(\alpha_{q_1q_1}+\alpha_{q_2q_2})} \quad dx\textnormal{-a.e.}
\end{equation}
In this case, each individual entry $\{B_H(t)_{q}\}_{t \in \bbR}$, $q=1,\hdots,m$, of HfBm $B_H$ is by itself a fBm with Hurst parameter $0 < h_q < 1$. In particular,
$$
\{B_H(ct)_{q}\}_{t \in \bbR} \stackrel{\textnormal{fdd}}= \{c^{h_q}B_H(t)_{q}\}_{t \in \bbR}, \quad c > 0, \quad q = 1,\ldots,m.
$$
However, it is generally \textit{not} true that
$$
\{B_H(ct)\}_{t \in \bbR} \stackrel{\textnormal{fdd}}= \{\textnormal{diag}(c^{h_1},\ldots,c^{h_q},\ldots,c^{h_m})B_H(t) \}_{t \in \bbR}, \quad c > 0.
$$
Otherwise, $B_H$ would necessarily be fractally connected.
\end{example}

\begin{remark}\label{r:CME_ideal_HfBm}
When simulating HfBm via Circulant Matrix Embedding, one verifies that regularization is rarely necessary for ensuring the positive definiteness of the covariance matrix for finite sample sizes. In other words, ideal-HfBM is also a useful approximation in practice.
\end{remark}

\begin{remark}\label{r:increments}
Let $Y_H(t) = B_{H}(t) - B_H(t-1)$, where $\{B_{H}(t)\}_{t \in \bbR}$ is an HfBm. Then, $\{Y_H(t)\}_{t \in \bbZ}$ is a (discrete time) stationary process with spectral density
$$
f_{Y_H}(x) = 2(1 - \cos(x))\sum^{\infty}_{k = -\infty}\frac{f_H(x + 2 \pi k)}{|x + 2 \pi k|^2}, \quad x \in (-\pi,\pi],
$$
where $f_H$ is the HfBm spectral density \eqref{e:Hadamard_harmonizable}. Under \eqref{e:g_is_bounded}, \eqref{e:g_is_Schwartz} and \eqref{e:hmax<delta0=<2(1+hmin)}, we obtain the entry-wise limiting behavior
$$
f_{Y_H}(x)_{q_1 q_2}(x) \sim \rho_{q_1 q_2 } \sigma_{q_1}\sigma_{q_2}|x|^{-2(h_{q_1 q_2}-1/2)}, \quad x \rightarrow 0^{+}, \quad q_1,q_2 = 1,\hdots,m.
$$
In particular, $\{Y_H(t)\}_{t \in \bbZ}$ does not fall within the scope of the usual definitions of multivariate scaling behavior or long range dependence \cite{lobato:1999,robinson:2008,nielsen:2011,kechagias:pipiras:2015}, which are restricted to fractally connected processes.
\end{remark}

\section{Wavelet-based analysis of HfBm}
\label{sec:wavestim}

In this section, we construct the wavelet analysis and estimation for HfBm. Due to the mathematical convenience of the notion of Hadamard (approximate) self-similarity, most of the properties of wavelet-based constructs resemble their univariate analogues.

\subsection{Multivariate discrete wavelet transform}
\label{sec:wavcoef}

Throughout the rest of the paper, we will make the following assumptions on the underlying wavelet basis. Such assumptions will be omitted in the statements.

\medskip

\noindent {\sc Assumption $(W1)$}: $\psi \in L^1(\bbR)$ is a wavelet function, namely,
\begin{equation}\label{e:N_psi}
\int_{\bbR} \psi^2(t)dt = 1 , \quad \int_{\bbR} t^{q}\psi(t)dt = 0, \quad q = 0,1,\hdots, N_{\psi}-1, \quad N_{\psi} \geq 2.
\end{equation}
\noindent {\sc Assumption ($W2$)}: the functions $\varphi$ (a bounded scaling function) and $\psi$ correspond to
\begin{equation}\label{e:W2}
\textnormal{a MRA of $L^2(\bbR)$, and $\textnormal{supp}(\varphi)$ and $\textnormal{supp}(\psi)$ are compact intervals.}
\end{equation}
\noindent {\sc Assumption $(W3)$}: for some $\beta > 1$,
$$
\widehat{\psi} \in C^2(\bbR)
$$
and
\begin{equation}\label{e:psihat_is_slower_than_a_power_function}
\sup_{x \in \bbR} |\widehat{\psi}(x)| (1 + |x|)^{\beta} < \infty.
\end{equation}

\medskip

\noindent Under \eqref{e:N_psi}--\eqref{e:psihat_is_slower_than_a_power_function}, $\psi$ is continuous, $\widehat{\psi}(x)$ is everywhere differentiable and
\begin{equation}\label{e:psihat_deriv=0}
\widehat{\psi}^{(l)}(0) = 0, \quad l = 0,\hdots,N_{\psi}-1
\end{equation}
(see \cite{Mallat1998}, Theorem 6.1 and the proof of Theorem 7.4).

\begin{definition}
Let $B_H = \{B_H(t)\}_{t \in \bbR} \in \bbR^m$ be an HfBm. For a scale parameter $j \in \bbN$ and a shift parameter $k \in \bbZ$, its ($L^1$-normalized) wavelet transform is defined by
\begin{equation}\label{e:D(j,k)}
D(2^j,k) = 2^{-j/2}\int_{\bbR}2^{-j/2}\psi(2^{-j}t - k) B_H(t) dt
=: \Big(d_q(j,k) \Big)_{q=1,\hdots,m}.
\end{equation}
\end{definition}
Under \eqref{e:N_psi}-\eqref{e:psihat_is_slower_than_a_power_function} and the continuity of the covariance function \eqref{e:Hadamard_harmonizable}, the wavelet transform \eqref{e:D(j,k)} is well-defined in the mean-square sense and ${\Bbb E}D(2^j,k) = 0$, $k \in \bbZ$, $j \in \bbN$ (see \cite{cramer:leadbetter:1967}, p.\ 86).

\subsection{Multivariate wavelet spectrum}
\label{sec:wavsp}
Fix $j_1 \leq j_2$, $j_1,j_2 \in \bbN$. Because of the approximate nature of Hadamard self-similarity, analysis and estimation must be considered in the coarse scale limit, by means of a sequence of dyadic numbers $\{a(n)\}_{n \in \bbN}$ satisfying
\begin{equation}\label{e:assumption_a(n)_n}
1 \leq \frac{n}{a(n)2^{j_2}} \leq \frac{n}{a(n)2^{j_1}} \leq n, \quad \frac{a(n)^{4 (h_{\max} - h_{\min})+1}}{n} + \frac{n}{a(n)^{1 + 2 \varpi_0}}\rightarrow 0, \quad n \rightarrow \infty.
\end{equation}

\begin{example}
An example of a scaling sequence satisfying \eqref{e:assumption_a(n)_n} for large enough $n$ is
$$
a(n) := 2 \lfloor n^{\frac{\eta}{4 (h_{\max} -  h_{\min})+1}} \rfloor, \quad \frac{4(h_{\max} - h_{\min})+1}{1 + 2 \varpi_0} < \eta < 1.
$$
In other words, a wide parameter range $[h_{\min},h_{\max}]$ or a low regularity parameter value $\varpi_0$ implies that $a(n)$ must grow slowly by comparison to $n$.
\end{example}
\begin{remark}
For a fixed octave range $\{j_1,\hdots,j_2\}$ associated with an initial scaling factor value $a(n_0) =1$ and sample size $n_0$, define the scale range $\{2^{j_1(n)},\hdots,2^{j_2(n)}\} = \{a(n)2^{j_1}, \hdots,a(n)2^{j_2}\}$ for a general sample size $n$ (where $a(n)$ is assumed dyadic). Then, under \eqref{e:assumption_a(n)_n}, for every $n$ the range of useful octaves is constant and given by $j_2(n) - j_1(n) = j_2 - j_1$, where the new octaves are $j_l(n) = \log_2 a(n) + j_l \in \bbN$, $l = 1,2$.
\end{remark}

\begin{definition}
Let $B_H = \{B_H(t)\}_{t \in \bbR}$ be an HfBm. Let $\{D(a(n)2^j,k)\}_{k = 1,\hdots,n_j,\hspace{0.5mm}j = j_1 , \hdots, j_2}$ be its wavelet coefficients, where
$$
n_{a,j}= \frac{n}{a(n)2^j}
$$
is the number of wavelet coefficients $D(a(n)2^j,\cdot)$ at scale $a(n)2^j$, of a total of $n$ (wavelet) data points. The wavelet variance and the sample wavelet variance, respectively, at scale $a(n)2^j$ are defined by
\begin{equation}\label{e:Sn(j)}
\EE{S_{n}(a(n)2^j)}, \quad S_{n}(a(n)2^j) = \sum^{n_{a,j}}_{k=1}\frac{D(a(n)2^j,k)D(a(n)2^j,k)^*}{n_{a,j}}, \quad j = j_1, \hdots, j_2.
\end{equation}
Let
\begin{equation}\label{e:rho^(q1q2)}
\varrho^{(q_1q_2)}(a(n)2^j) = \EE{d_{q_1}(a(n)2^j,0)d_{q_2}(a(n)2^j,0)}.
\end{equation}
The standardized counterparts of \eqref{e:Sn(j)} are
\begin{equation}\label{e:Wn(j)}
\EE{W_{n}(a(n)2^j)}, \quad W_{n}(a(n)2^j) =  \varrho(a(n)2^j)^{\circ -1} \circ \sum^{n_{a,j}}_{k=1} \frac{D(a(n)2^j,k)D(a(n)2^j,k)^*}{ n_{a,j} },
\end{equation}
$j = j_1, \hdots, j_2$, where $\varrho(a(n)2^j) = \Big(\varrho^{(q_1q_2)}(a(n)2^j)\Big)_{q_1,q_2=1,\hdots,m}$. Entry-wise, for $1 \leq q_1, q_2 \leq m$,
\begin{equation}\label{e:Sq1q2}
S^{(q_1q_2)}_{n}(a(n)2^j) =\sum_{k=1}^{n_{a,j}}\frac{d_{q_1}(a(n)2^j,k)d_{q_2}(a(n)2^j,k)}{n_{a,j}}, \quad W^{(q_1 q_2)}_{n}(a(n)2^j) = \frac{S^{(q_1q_2)}_{n}(a(n)2^j)}{\varrho^{(q_1q_2)}(a(n)2^j)}.
\end{equation}
\end{definition}

Proposition~\ref{p:4th_moments_wavecoef}, stated next, provides basic results on the moments and asymptotic distribution, in the coarse scale limit $a(n)2^j \rightarrow \infty$, of the wavelet transform \eqref{e:D(j,k)} and variance \eqref{e:Wn(j)}. In particular, in regard to the limits in distribution, the vector of random matrices $\{S_{n}(a(n)2^j)\}_{j=j_1,\hdots,j_2}$ can be intuitively interpreted as an asymptotically unbiased and Gaussian estimator of its population counterpart $\{\EE{S_{n}(a(n)2^j,k)}\}_{j=j_1,\hdots,j_2}$. The celebrated decorrelation property of the wavelet transform (e.g., \cite{bardet:2002}, Proposition II.2) lies behind the Gaussian limits in the proposition, as well as of the fact that the random matrices $S_{n}(a(n)2^j)$ have weak inter-component, intra-scale and inter-scale correlations. Also, each entry $S^{(q_1 q_2)}_{n}(a(n)2^j)$, $q_1,q_2 = 1,\hdots, m$, displays asymptotic power law scaling. In Section~\ref{sec:aci} below, these properties are used to define estimators of the scaling exponents and to analytically establish their asymptotic performance.

In the statement of Proposition~\ref{p:4th_moments_wavecoef}, we make use of the operator
\begin{equation}\label{e:vec_S}
\textnormal{vec}_{{\mathcal S}}S = (s_{11}, s_{12},\hdots, s_{1m}; s_{22},\hdots, s_{2m}; \hdots;s_{m-1,m-1}, s_{m-1,m};s_{m,m})^*.
\end{equation}
In other words, $\textnormal{vec}_{{\mathcal S}}\,\cdot$ vectorizes the upper triangular entries of the symmetric matrix $S$.
\begin{proposition}\label{p:4th_moments_wavecoef}
Let $B_{H} = \{B_H(t)\}_{t \in \bbR}$ be an HfBm. Consider $j, j' \in \bbN$ and $k,k' \in \bbZ$. Then,
\begin{itemize}
\item [($i$)] for fixed $j$, $j'$, $k$, $k'$, we can write
\begin{equation}\label{e:Phij,j'}
\EE{D(a(n)2^j,k)D(a(n)2^{j'},k')^*} = \Big( \Xi^{jj',a}_{q_1 q_2}(a(n)(2^jk-2^{j'}k'))\Big)_{q_1,q_2=1,\hdots,m}
\end{equation}
for some matrix-valued function $\Xi^{jj',a}(\cdot)$ that depends on $a(n)$. In particular, for fixed scales $j$, $j'$, $\{D(a(n)2^j,k)\}_{k \in \bbZ}$ is a stationary stochastic process;
\item [($ii$)] for $q_1, q_2 = 1,\hdots,m$,
\begin{equation}\label{e:varrhoa,j_neq_0}
\varrho^{(q_1 q_2)}(a(n)2^j) = \Xi^{jj,a}_{q_1 q_2}(0) \neq 0, \quad j \in \bbN,
\end{equation}
i.e., \eqref{e:Wn(j)} is well-defined;
\item [($iii$)] for $\Xi^{jj',a}(\cdot)$ as in \eqref{e:Phij,j'}, $q_1, q_2 = 1,\hdots,m$, and $z \in \bbZ$, as $n \rightarrow \infty$,
\begin{equation}\label{e:Phi(j,j')q1q2(z)}
\frac{\Xi^{jj',a}_{q_1 q_2}(a(n)z)}{a(n)^{2h_{q_1 q_2}}} \rightarrow \Phi^{jj'}_{q_1 q_2}(z)= \rho_{q_1q_2} \sigma_{q_1}\sigma_{q_2}\hspace{0.5mm}\int_{\bbR}e^{i z x}|x|^{-(2h_{q_1q_2}+1)} \widehat{\psi}(2^jx) \overline{\widehat{\psi}(2^{j'}x)}dx.
\end{equation}
In particular, the wavelet spectrum ${\Bbb E}S_n(a(n)2^j)$ (see \eqref{e:Sn(j)}) can be approximated by that of an ideal-HfBM (see \eqref{e:Hadamard_ss}) in the sense that
\begin{equation}\label{e:ED(j,k)D(j,k)*=scaling}
a(n)^{\circ \hspace{0.5mm}-2H} \circ \EE{S_n(a(n)2^j)} \rightarrow \EE{ S_{\textnormal{ideal}}(2^j)} = \Big(\Phi^{jj}_{q_1 q_2}(0)\Big)_{q_1,q_2 = 1,\hdots,m};
\end{equation}
\item[($iv$)] 
for $\Xi^{jj'}$ as in \eqref{e:Phij,j'} and $q_1,q_2 = 1,\hdots,m$, as $n \rightarrow \infty$,
$$
\frac{1}{\sqrt{n_{a,j}}}\frac{1}{\sqrt{n_{a,j'}}}\sum^{n_{a,j}}_{k=1}\sum^{n_{a,j'}}_{k'=1}
\frac{\Xi^{jj',a}_{q_1 q_2}(a(n)(2^jk-2^{j'}k')) }{a(n)^{2 h_{q_1 q_2}} }
$$
\begin{equation}\label{e:limiting_kron}
\rightarrow 2^{-\frac{(j+j')}{2}}\textnormal{gcd}(2^j,2^{j'})\sum^{\infty}_{z= - \infty} \Phi^{jj'}_{q_1 q_2}(z \textnormal{gcd}(2^j,2^{j'})) ;
\end{equation}
where $\Phi^{jj'}_{q_1 q_2}(\cdot)$ is given by \eqref{e:Phi(j,j')q1q2(z)};
\item[(v)] 
for $1 \leq q_1 \leq q_2 \leq m$, $1 \leq q_3 \leq q_4 \leq m$, as $n \rightarrow \infty$,
\begin{equation}\label{e:Gjj'}
\frac{\sqrt{n_{a,j}}}{a(n)^{\delta_{q_1 q_2}}}\frac{\sqrt{n_{a,j'}}}{a(n)^{\delta_{q_3 q_4}}}\hspace{1mm} \textnormal{Cov}\Big[W^{(q_1 q_2)}_{n}(a(n)2^j),W^{(q_3 q_4)}_{n}(a(n)2^{j'})\Big] \rightarrow G^{jj'}(q_1, q_2, q_3, q_4),
\end{equation}
for $\delta_{\cdot \cdot}$ as in \eqref{e:delta_kn}, where
$$
\frac{\Phi^{jj}_{q_1 q_2}(0)\Phi^{j'j'}_{q_3 q_4}(0)}{2^{-\frac{(j+j')}{2}}}G^{jj'}(q_1, q_2, q_3, q_4)
$$
$$
= \left\{\begin{array}{cc}
\phi^{jj'}(q_1, q_3, q_2, q_4), & \textnormal{if }\delta_{q_1 q_3} = 0 = \delta_{q_2 q_4} \\
                                & \textnormal{and } (\delta_{q_1 q_4} > 0 \textnormal{ or } \delta_{q_2 q_3} >0 );\\
\phi^{jj'}(q_1, q_4, q_2, q_3), & \textnormal{if }(\delta_{q_1 q_3} > 0 \textnormal{ or } \delta_{q_2 q_4} > 0 ) \\
                                & \textnormal{ and } \delta_{q_1 q_4} = 0 = \delta_{q_2 q_3} ;\\
\phi^{jj'}(q_1, q_3, q_2, q_4)+ \phi^{jj'}(q_1, q_4, q_2, q_3), & \textnormal{if }\delta_{q_1 q_3} = \delta_{q_2 q_4} = \delta_{q_1 q_4} = \delta_{q_2 q_3} = 0;\\
0, & \textnormal{otherwise},\\
\end{array}\right.
$$
and
\begin{equation}\label{e:phi^jj'(q1,q2,q3,q4)}
\phi^{jj'}(q_1, q_2, q_3, q_4) := \textnormal{gcd}(2^j,2^{j'}) \sum^{\infty}_{z = -\infty}\Phi^{jj'}_{q_1 q_2}(z \textnormal{gcd}(2^{j},2^{j'}))
\Phi^{jj'}_{q_3 q_4}(z \textnormal{gcd}(2^{j},2^{j'}));
\end{equation}
\item [(vi)] as $n \rightarrow \infty$,
\begin{equation}\label{e:clt_wavelet_coef}
\Big(\textnormal{vec}_{{\mathcal S}}\Big[\Big(\frac{\sqrt{n_{a,j}}}{a(n)^{\delta_{q_1 q_2}}} \Big)_{q_1,q_2 = 1,\hdots,m} \circ  (W_{n}(a(n)2^j) - {\mathbf 1})\Big]\Big)_{j=j_1,\hdots,j_2}
\stackrel{d}\rightarrow \NOR_{\frac{m(m+1)}{2} \times J}(0,G),
\end{equation}
where
\begin{equation}\label{e:J=j2-j1+1}
J = j_2 - j_1 + 1,
\end{equation}
and ${\mathbf 1}$ is a vector of ones. Each entry of the asymptotic covariance matrix $G$ in \eqref{e:clt_wavelet_coef} is given by the terms $G^{jj'}(q_1,q_2,q_3,q_4)$ in \eqref{e:Gjj'} for appropriate values of $j$, $j'$, $q_1$, $q_2$, $q_3$ and $q_4$.
\end{itemize}
\end{proposition}

\begin{remark}
Note that, up to a change of sign, each entry $\Phi^{jj'}_{q_1 q_2}(z)$ on the right-hand side of \eqref{e:Phi(j,j')q1q2(z)} corresponds to the covariance between the wavelet coefficients at octaves $j$ and $j'$ of a fBm with parameter $h_{q_1q_2}$.
\end{remark}

\begin{remark}
In Proposition \ref{p:4th_moments_wavecoef}, $(v)$, the asymptotic covariance between same-entry wavelet variance terms is always nontrivial, irrespective of the values of fractal connectivity parameters. For example, when $(1,2) =: (q_1,q_2) = (q_3,q_4) $, it is given by
$$
G^{jj'}(1, 2, 1, 2) = \frac{2^{-\frac{(j+j')}{2}}}{\Phi^{jj}_{1 2}(0)\Phi^{j'j'}_{12}(0)}\cdot
\left\{\begin{array}{cc}
\phi^{jj'}(1, 1, 2, 2), & \delta_{1 2} > 0;\\
\phi^{jj'}(1, 1, 2, 2) + \phi^{jj'}(1, 2, 2, 1), & \delta_{1 2} = 0.
\end{array}\right.
$$
This is not true for terms associated with different pairs of indices. For example, if $(q_1,q_2) = (1,1) \neq (q_3,q_4) = (2,2)$, then
$$
G^{jj'}(1, 1, 2, 2) = \frac{2^{-\frac{(j+j')}{2}}}{\Phi^{jj}_{1 1}(0)\Phi^{j'j'}_{22}(0)}\cdot
\left\{\begin{array}{cc}
2\phi^{jj'}(1, 2, 1, 2), & \delta_{1 2} = 0;\\
0, & \delta_{1 2} > 0.\\
\end{array}\right.
$$
In other words, the phenomenon of the asymptotic decorrelation of wavelet variance terms is only observed for instances involving departures from fractal connectivity.
\end{remark}

\subsection{Estimation of the scaling exponents}
\label{sec:aci}

\subsubsection{Definition of the estimators}

As in the univariate case, the fact that the wavelet variance \eqref{e:Sn(j)} satisfies the Hadamard scaling relation in the coarse scale limit (see \eqref{e:ED(j,k)D(j,k)*=scaling}) points to the development of a log-regression regression method based on the sample wavelet variance across scales.

Estimators can be defined in a standard way by means of the log-regression relations
$$
\hat \alpha_{q} =  \sum_{j=j_1}^{j_2} w_j\log_2 S^{(qq)}_{n}(a(n)2^j),\qquad 1 \leq q \leq m,
$$
\begin{equation}\label{eq:alphaW}
\hat \alpha_{q_1 q_2} =  \sum_{j=j_1}^{j_2} w_j\log_2\hspace{0.5mm}|S^{(q_1 q_2)}_{n}(a(n)2^j)|\hspace{0.5mm},\qquad 1 \leq q_1 < q_2 \leq m,
\end{equation}
which are well-defined with probability 1, where $w_j$, $j = j_1,\hdots,j_2$, are linear regression weights satisfying
\begin{equation}\label{e:sum_wj=0,sum_jwj=1}
\sum^{j_2}_{j=j_1}w_j = 0, \quad \sum^{j_2}_{j=j_1}j w_j = 1
\end{equation}
(e.g., \cite{aftv00,stoev:pipiras:taqqu:2002}).
The derived estimator
\begin{equation}\label{equ:defdelta}
\hat\delta_{q_1q_2}=\frac{\hat\alpha_{q_1 q_1}+\hat\alpha_{q_2 q_2}}{2} - \hat\alpha_{q_1q_2},\qquad q_1\neq q_2,
\end{equation}
for the parameter \eqref{e:delta_kn} will be used in Section \ref{sec:test} in the construction of a hypothesis test for fractal connectivity (see also \cite{flandrin:1992,veitch:abry:1999,amblard:coeurjolly:2011}). With the standard Gaussian asymptotics for the sample wavelet variances $\{S_n(a(n)2^j)\}_{j=j_1,\hdots,j_2}$ established in Proposition \ref{p:4th_moments_wavecoef}, the system of equations \eqref{eq:alphaW} and \eqref{equ:defdelta} is expected to yield efficient estimators. In fact, this is proved in the next section.

\subsubsection{Asymptotic distribution}\label{s:asymptotic_performance}

In the next result, Theorem~\ref{t:asymptotic_dist_estimators}, we build upon Proposition \ref{p:4th_moments_wavecoef} to show that the vector $( \widehat{\alpha}_{q_1 q_2})_{q_1,q_2=1,\hdots,m}$ is an asymptotically unbiased and Gaussian estimator of the scaling exponents $( {\alpha}_{q_1 q_2})_{q_1,q_2=1,\hdots,m}$ (whence the analogous statement holds for the estimator \eqref{equ:defdelta}). As in the aforementioned proposition, the decorrelation property of the wavelet transform contributes to the Gaussianity of the asymptotic distribution.
\begin{theorem}\label{t:asymptotic_dist_estimators}
Let $w_{j_1}, \hdots, w_{j_2}$ be the weight terms in \eqref{eq:alphaW} and let $G$ be as in \eqref{e:clt_wavelet_coef}. Suppose
\begin{equation}\label{e:rho_q1q2_neq_0}
\rho_{q_1 q_2} \neq 0, \quad q_1,q_2 = 1,\hdots,m, \quad q_1 \neq q_2
\end{equation}
(see \eqref{e:HfBm_specdens}). Then, as $n \rightarrow \infty$,
\begin{equation}\label{e:asymptotic_dist_estimators}
\vecoper_{{\mathcal S}} \Big[ \Big(\frac{\sqrt{n}}{a(n)^{\delta_{q_1 q_2}+1/2}}\hspace{1mm}( \widehat{\alpha}_{q_1 q_2} - \alpha_{q_1 q_2})\Big)_{q_1,q_2=1,\hdots,m} \Big]
\stackrel{d}\rightarrow \NOR_{\frac{m(m+1)}{2}  }(0,MGM^*).
\end{equation}
\end{theorem}

\begin{remark}
Condition \eqref{e:rho_q1q2_neq_0} is needed to ensure the consistency of the estimator. Furthermore, it is clear that $\rho_{q q} > 0$ for $q = 1,\hdots,m$, since otherwise the process is not proper or its spectral density is not positive semidefinite a.e.
\end{remark}

\begin{remark}
The convergence rate in \eqref{e:asymptotic_dist_estimators} depends on the unknown fractal connectivity parameters $\delta_{\cdot \cdot}$. In practice, the latter can be replaced by their estimates or, in some cases, ignored, since they exponentiate the slow growth term $a(n)$.
\end{remark}
\subsubsection{Numerical simulation setting}\label{s:numerical_sim_setting}
We conducted Monte Carlo experiments over 1000 independent realizations of bivariate HfBm. Though several parameter settings were tested, results are reported only for two representative cases: fractally connected ideal-HfBm $(\alpha_{11},\alpha_{22}, \rho_{12}) = (0.4, 0.8, 0.6)$;
and non-fractally connected HfBm (see Example~\ref{ex:A}) with $(\alpha_{11},\alpha_{22}, \delta_{12},\rho_{12}) = (0.4, 0.8, 0.2, 0.6)$.
HfBm copies were synthesized using the multivariate process synthesis toolbox described in \cite{Helgason_H_2011_j-sp_smsspmdccme,Helgason_H_2011_j-sp_fessmgtsuce} and available at {\tt www.hermir.org}.
For the wavelet analysis, least asymmetric Daubechies wavelet with $N_{\Psi} = 2$ vanishing moments were used \cite{Mallat1998}.
Estimation by means of weighted linear regression was performed on the octave range $(j_1, j_2) = (3, \log_2 n - N_{\Psi})$.
This choice of regression range (with fine scale $j_1$ fixed and $j_2\sim\log_2 n$) is used here in order to, first, obtain results that are consistent and comparable with those reported in the literature for the estimation of univariate scaling exponents, cf., e.g., \cite{aftv00} and references therein, and, second, avoid the issue of fixing $h_{\min}$, $h_{\max}$ and $\varpi_0$ in \eqref{e:assumption_a(n)_n}.

Monte Carlo parameter estimates $(\widehat{\alpha}_{11},\widehat{\alpha}_{22}, \widehat{\alpha}_{12}, \widehat{\delta}_{12})$ are reported in Figures~\ref{fig:EStimPerfa} and \ref{fig:EStimPerfb} in terms of bias, standard deviation, skewness and kurtosis, as functions of (the $ \log_2$ of) sample sizes $n=\{2^{10},2^{12},2^{14},2^{16},2^{18},2^{20}\}$.

\subsubsection{Finite sample performance}
\label{s:finite_sample_performance}

Figures~\ref{fig:EStimPerfa} and \ref{fig:EStimPerfb} show that for all four estimates $(\hat \alpha_{11}, \hat \alpha_{22}, \hat \alpha_{12}, \hat \delta_{12})$, the bias becomes negligible as the sample size grows. It can also be seen that the bias is slightly larger for non-fractally connected data, notably for the parameter $\delta_{12}$.

The figures further show that standard deviations for all estimates decrease as $n^{-1/2}$ (the latter trend being plotted as superimposed dashed lines), with no significant difference between fractally and non-fractally connected instances. For $(\hat \alpha_{11}, \hat \alpha_{22}, \hat \alpha_{12})$, there is no detectable dependence of the outcomes on the actual values of the parameters themselves (in accordance with theoretical calculations reported in Table~\ref{tab:varb}).
Accordingly, one observes that the variances for $\hat \alpha_{11}$ and $\hat \alpha_{22} $ are identical and depend on the sample size, but not on any parameter of the stochastic process. For $\hat \alpha_{12} $, the variance does not depend on $\alpha_{12} $, but results not shown demonstrate that it does vary with $\rho_{12}$, as confirmed by approximate calculations (see \eqref{equ:varaxy}). For $\hat \delta$, further simulations not displayed indicate that Var $\hat \delta$ depends on $\delta$ roughly proportionally to $ 1 + \delta $. Departures from fractal connectivity tend to imply a decrease in the variance of $\hat \delta$; this is a counterintuitive phenomenon that has yet to be fully explained.

In regard to convergence in distribution, Figures~\ref{fig:EStimPerfa} and \ref{fig:EStimPerfb} show that the Monte Carlo skewness and kurtosis estimates for $(\alpha_{11},\alpha_{22}) $ are very close to $0$ even at very small sample sizes, both with and without fractal connectivity. This characterizes a fast convergence to limiting normal distributions, and lies in agreement with the asymptotics presented in Theorem~\ref{t:asymptotic_dist_estimators}. However, the weak convergence is observed to be much slower in practice for the cross-exponents $(\alpha_{12}, \delta_{12})$, yet with no noticeable difference between fractally and non-fractally connected instances, i.e., different values of the parameter $\delta_{12}$.

\begin{figure}[tb]
\centerline{
\includegraphics[width=\linewidth]{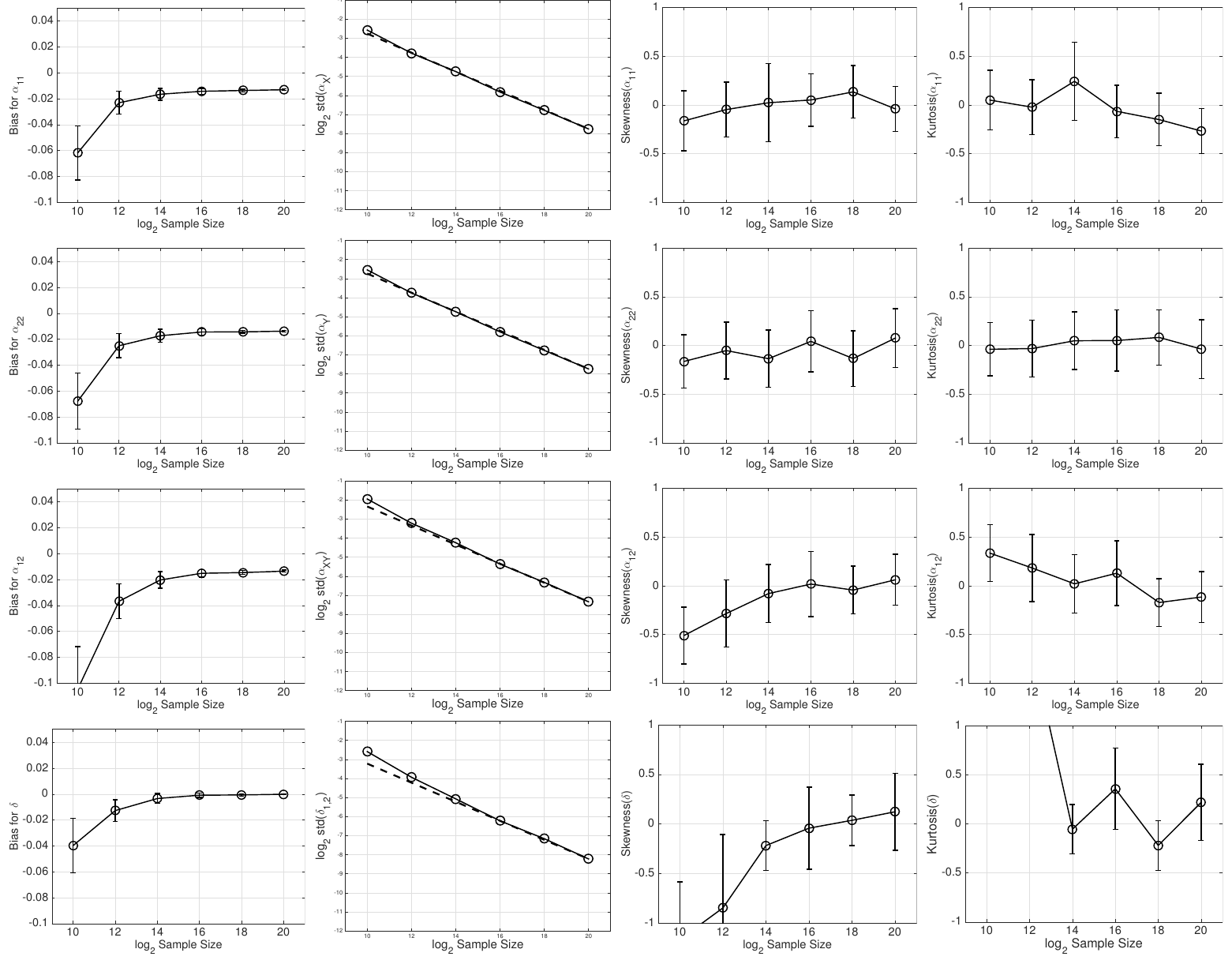}
}
\caption{\label{fig:EStimPerfa}{\bf Estimation performance and asymptotic normality.} Study based on 1000 Monte Carlo realizations of (fractally connected) ideal-HfBm with parameters  $(\alpha_{11},\alpha_{22}, \delta_{12}, \rho_{12}) = (0.4, 0.8, 0, 0.6)$.}
\end{figure}

\begin{figure}[tb]
\centerline{
\includegraphics[width=\linewidth]{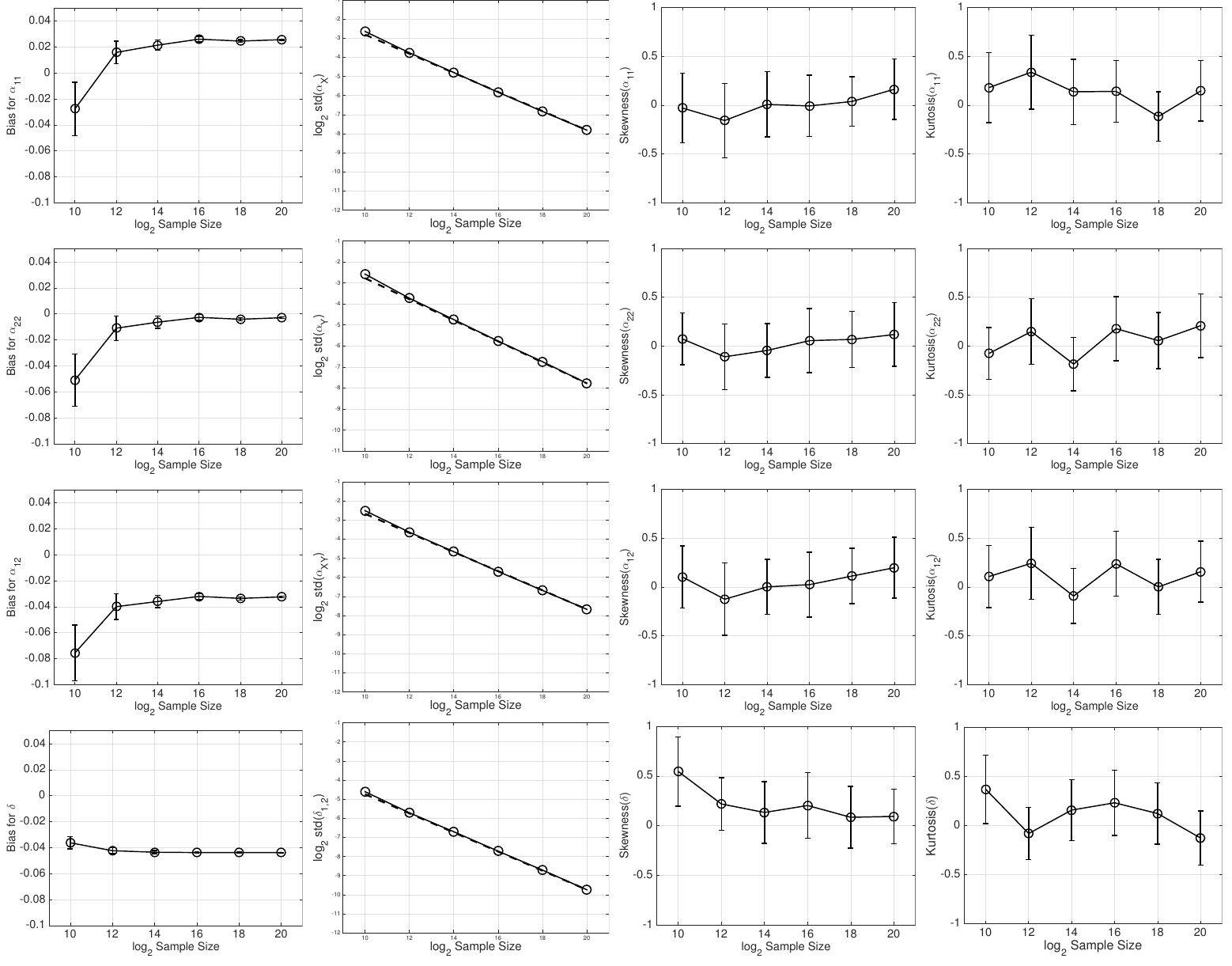}
}
\caption{\label{fig:EStimPerfb}{\bf Estimation performance and asymptotic normality.} Study based on 1000 Monte Carlo realizations of (non-fractally connected) HfBm with parameters ($\alpha_{11},\alpha_{22}, \delta_{12},\rho_{12}) = (0.4, 0.8, 0.2, 0.6)$.}
\end{figure}

\section{Confidence intervals}
\label{sec:MC}

To complement the asymptotic and finite sample results in Sections \ref{s:asymptotic_performance} and \ref{s:finite_sample_performance},
we now construct confidence intervals for the estimators \eqref{eq:alphaW} and \eqref{equ:defdelta}.
In particular, we investigate the effect of omitting the covariance among sample wavelet variance terms,
as well as the dependence of the asymptotic variance in \eqref{e:asymptotic_dist_estimators} and confidence intervals on the unknown parameter values.
This will allow us to efficiently compute approximate confidence intervals in practice based on a single sample path.

\subsection{Asymptotic covariance}
\label{sec:CI}

In Theorem~\ref{t:Cov(log,log)_cross} below, we provide a finite sample approximation for the covariance between the logarithms of sample wavelet variances, coupled with the convergence rate of the covariance approximation to its asymptotic expression in terms of ideal-HfBm covariances \eqref{e:Phi(j,j')q1q2(z)}. As in Theorem \ref{t:asymptotic_dist_estimators}, in part due to the decorrelation property of the wavelet transform, the estimators' second order properties are comparable to those observed in the i.i.d.\ case (e.g., \cite{lehmann:casella:1998}, p.\ 430). In particular, the finite sums provided in Theorem \ref{t:Cov(log,log)_cross} mathematically justify heuristic expressions obtained from Taylor expansions (c.f.\ \eqref{e:Cov(log,log)_cross_order=zero} below), and the theorem also accurately quantifies the approximation error as a function of the sample size and scale.

In the theorem, we use indicator functions to regularize $\log_2  |S^{(q_1q_2)}_{n}(a(n)2^j)|$ around the origin. The truncation converges fast to zero and is of no impact in practice. Indeed, one can redefine the estimator \eqref{eq:alphaW} with the indicators and obtain the same asymptotic distribution \eqref{e:asymptotic_dist_estimators}.

\begin{theorem}\label{t:Cov(log,log)_cross}
Let $q_1,q_2, q_3, q_4 = 1,\hdots,m$ be generic indices, and suppose condition \eqref{e:rho_q1q2_neq_0} holds. Consider $S_{n}(a(n)2^j) = \{S^{(q_1 q_2)}_{n}(a(n)2^j)\}$, $W_{n}(a(n)2^j) = \{W^{(q_1 q_2)}_{n}(a(n)2^j)\}$ as in \eqref{e:Sq1q2} and $\Phi^{jj'}_{q_1q_2}(\cdot)$ as in \eqref{e:Phi(j,j')q1q2(z)}. For notational simplicity, denote
\begin{equation}\label{e:n*}
n_* = \frac{n}{a(n)2^{j+j'}}
\end{equation}
and also recall the definition of $\varpi_0$ in \eqref{e:hmax<delta0=<2(1+hmin)}. Consider a sequence $\{r_n\}_{n \in \bbN} \subseteq (0,1/2)$ satisfying
\begin{equation}\label{e:rn->0}
\frac{1}{r_n} = O\Big( \exp\Big\{C \Big( \frac{n_*}{a(n)^{4 h_{\max} - 4 h_{\min}}}\Big)^{1 - \xi} \Big\} \Big( \frac{a(n)^{4 h_{\max} - 4 h_{\min}}}{n_*}\Big)^{1/2}\Big)
\end{equation}
for any $C > 0$ and any $0 < \xi < 1$. Then,
$$
\COV{ \log_2 |S^{(q_1q_2)}_{n}(a(n)2^j)| 1_{\{|W^{(q_1 q_2)}_{n}(a(n)2^j)| > r_n\}}
}{\log_2 |S^{(q_3q_4)}_{n}(a(n)2^{j'})| 1_{\{|W^{(q_3 q_4)}_{n}(a(n)2^{j'})| > r_n\}}}
$$
$$
= \hspace{0.5mm} \frac{(\log_2 e)^2 \hspace{1mm} 2^{-(j+j')}} {\Phi^{jj}_{q_1q_2}(0) \Phi^{j'j'}_{q_3q_4}(0)\{1 + O(a(n)^{-\varpi_0})\}^2 }
$$
$$
\Big\{\frac{a(n)^{2 (h_{q_1 q_3}+h_{q_2 q_4})- 2 (h_{q_1 q_2}+h_{q_3 q_4})}}{n_*}\Big[ \frac{1}{n_*}\sum^{2^{j'}n_*}_{k = 1}\sum^{2^{j}n_*}_{k' = 1}\Phi^{jj'}_{q_1 q_3}(2^{j}k-2^{j'}k')\Phi^{jj'}_{q_2 q_4}(2^{j}k-2^{j'}k')\Big]
$$
$$
+ \frac{a(n)^{2 (h_{q_1 q_4}+h_{q_2 q_3})-2 (h_{q_1 q_2}+h_{q_3 q_4})}}{n_*} \Big[\frac{1}{n_*}\sum^{2^{j'}n_*}_{k = 1}\sum^{2^{j}n_*}_{k' = 1}\Phi^{jj'}_{q_1 q_4}(2^{j}k-2^{j'}k')\Phi^{jj'}_{q_2 q_3}(2^{j}k-2^{j'}k')\Big]
$$
\begin{equation}\label{e:Cov(log,log)_cross}
+ o\Big( \frac{a^{2\max\{h_{q_1 q_3}+h_{q_2 q_4},h_{q_1 q_4}+h_{q_2 q_3}\} - 2 (h_{q_1 q_2}+h_{q_3 q_4})}}{n_*}\Big) \Big\} + O\Big( \Big(\frac{a(n)^{4 h_{\max} - 4 h_{\min}}}{ n_*}\Big)^2\Big).
\end{equation}
\end{theorem}
\begin{remark}
Note that the finite summations on the right-hand side of \eqref{e:Cov(log,log)_cross} converge as $n \rightarrow \infty$. For example,
$$
 \frac{1}{n_*}\sum^{2^{j'}n_*}_{k = 1}\sum^{2^{j}n_*}_{k' = 1}\Phi^{jj'}_{q_1 q_3}(2^{j}k-2^{j'}k')\Phi^{jj'}_{q_2 q_4}(2^{j}k-2^{j'}k') \rightarrow \phi^{jj'}(q_1,q_2,q_3,q_4),
$$
where $\phi^{jj'}(\cdot)$ is defined by \eqref{e:phi^jj'(q1,q2,q3,q4)}
(this can be shown by the same argument for establishing \eqref{e:summation_deviation_HfBm_idealHfBm} and \eqref{e:(1/n*)sum_sum_Phi(2jk-2j'k')->gcd(2j,2j')sum_Phi(zgcd(2j,2j'))} in the proof of Proposition \ref{p:4th_moments_wavecoef}, $(iv)$).
\end{remark}

\begin{remark}\label{r:O((a^(4hmax-4hmin)/n*)^2)=o(1)}
In \eqref{e:Cov(log,log)_cross}, the main residual term satisfies
\begin{equation}\label{e:O(a^(4hmax-4hmin)/n*))}
O\Big( \Big(\frac{a(n)^{4 h_{\max} - 4 h_{\min}}}{ n_*}\Big)^2\Big) = o(1)
\end{equation}
under the condition \eqref{e:assumption_a(n)_n}. In practice, the modeling of instances involving extreme deviations from fractal connectivity ($\delta_{\cdot \cdot}$ close to 1) may require greater regularity from the functions $g_{\cdot \cdot}$ (see \eqref{e:HfBm_specdens}) or the wavelet basis. Otherwise, for conservative choices of the regularity parameters (e.g., $\varpi_0 = 2$ and $\beta = 1 + \varepsilon$ for small $\varepsilon > 0$), the expression \eqref{e:O(a^(4hmax-4hmin)/n*))} may converge slower to zero than the main terms in \eqref{e:Cov(log,log)_cross}.
\end{remark}

\begin{remark}
In \eqref{e:D(j,k)}, we assume that an HfBm continuous time sample path is available. However, it is well known that, under mild assumptions, the availability of discrete time observations does not generally alter the nature of the asymptotic distribution for wavelet estimators (see, for instance, \cite{bardet:2002}, Section III; \cite{stoev:pipiras:taqqu:2002}, Section 3.2; \cite{abry:didier:2016:supplementary}, Section C). Establishing this for the case of HfBm is a topic for future work.
\end{remark}

\begin{remark}
When $q = q_1= q_2$, for the univariate case the expansion \eqref{e:Cov(log,log)_cross} appears implicitly in \cite{moulines:roueff:taqqu:2007:spectral}, expression (86). Also, when computing the moments of $W^{(qq)}_{n}(a(n)2^j)$, the truncation is unnecessary (see Remark \ref{r:q=q1=q2=>trunc_is_unnecessary} below).
\end{remark}

\subsection{Variances and covariances of $\hat \alpha_{q_1 q_2}$ and $\hat \delta_{q_1 q_2}$}

\subsubsection{Closed-form approximations}

Turning back to expression \eqref{e:Cov(log,log)_cross}, by ignoring the scaling factor ($a(n) \equiv 1$), the truncation and setting all the convergence order terms to zero we obtain the approximation
$$
\COV{ \log_2  |S^{(q_1q_2)}_{n}(2^j)|}{\log_2 |S^{(q_3q_4)}_{n}(2^{j'})| }
$$
$$
\approx \hspace{0.5mm} \frac{(\log_2 e)^2 2^{-(j+j')}}{\Phi^{jj}_{q_1q_2}(0) \Phi^{j'j'}_{q_3q_4}(0) } \frac{1}{n_j}
\Big\{\frac{1}{n_{j'}}\sum^{n_j}_{k = 1}\sum^{n_{j'}}_{k' = 1}\Phi^{jj'}_{q_1 q_3}(2^{j}k-2^{j'}k')\Phi^{jj'}_{q_2 q_4}(2^{j}k-2^{j'}k')
$$
\begin{equation}\label{e:Cov(log,log)_cross_order=zero}
+ \frac{1}{n_{j'}}\sum^{n_j}_{k = 1}\sum^{n_{j'}}_{k' = 1}\Phi^{jj'}_{q_1 q_4}(2^{j}k-2^{j'}k')\Phi^{jj'}_{q_3 q_2}(2^{j'}k'-2^{j}k)\Big\},
\end{equation}
where
$$
n_{j}:= \frac{n}{2^j}.
$$
Note that the approximation \eqref{e:Cov(log,log)_cross_order=zero} is a finite sample one, and ignores the potential asymptotic decorrelation effect stemming from the shifting scaling factor $a(n)$ (see Proposition \ref{p:4th_moments_wavecoef}, $(v)$, and \eqref{e:Cov(log,log)_cross}). Consider the normalization
\begin{equation}\label{e:r12_in_terms_of_Phi}
\ccoef_{q_1 q_2}(j,k;j',k') := \frac{\Phi^{jj'}_{q_1 q_2}(2^{j}k-2^{j'}k')}{\sqrt{\Phi^{jj}_{q_1 q_1}(0)\Phi^{j'j'}_{q_2 q_2}(0)}} \in [-1,1].
\end{equation}
Expressions \eqref{eq:alphaW}, \eqref{e:Cov(log,log)_cross_order=zero} and \eqref{e:r12_in_terms_of_Phi} yield the covariance approximation
\begin{multline}
\label{equ:VarAlphaFinal}
\COV{ \hat \alpha_{q_1q_2}}{\hat \alpha_{q_3q_4}} \approx(\log_2\e)^2\sum_{j,j'=j_1}^{j_2}\frac{w_{j}w_{j'}}{n_{j}n_{j'}}\sum_{k=0}^{n_j-1}\sum_{k'=0}^{n_{j'}-1} \frac{\ccoef_{q_1q_3}(j,k;j',k')\ccoef_{q_2q_4}(j,k;j',k')}{\ccoef_{q_1q_2}(j,k;j,k)\ccoef_{q_3q_4}(j',k';j',k')}\\
+\frac{\ccoef_{q_1q_4}(j,k;j',k')\ccoef_{q_2q_3}(j,k;j',k')}{\ccoef_{q_1q_2}(j,k;j,k)\ccoef_{q_3q_4}(j',k';j',k')}.
\end{multline}
In particular, \eqref{equ:VarAlphaFinal} further allows us to compute
\begin{multline}\label{eq:var}
\VAR{\hat\delta_{q_1q_2}} = \frac{1}{4}\left(\VAR{\hat\alpha_{q_1q_1}}+\VAR{\hat\alpha_{q_2q_2}}\right) + \VAR{\hat\alpha_{q_1q_2}}
\\+ \frac{1}{2}\COV{\hat\alpha_{q_1q_1}}{\hat\alpha_{q_2q_2}} - \COV{\hat\alpha_{q_1q_1}}{\hat\alpha_{q_1q_2}}
- \COV{\hat\alpha_{q_2q_2}}{\hat\alpha_{q_1q_2}}.
\end{multline}
The variances \eqref{eq:var} will be used in Section~\ref{sec:test} in the construction of a test for fractal connectivity, i.e., for the hypothesis $H_0:\;\delta_{q_1q_2}\equiv 0$.
Table \ref{tab:vara} summarizes the closed-form approximations for $\VAR{\hat\alpha_{q_1q_2}}$, $\VAR{\hat\alpha_{q_1q_1}}$, $\COV{ \hat \alpha_{q_1q_1}}{\hat \alpha_{q_1q_2}}$ and $\COV{ \hat \alpha_{q_1q_1}}{\hat \alpha_{q_2q_2}}$ established in \eqref{equ:VarAlphaFinal}.

 \begin{table}[h]
\centering
\begin{tabular}{| c | c |}
\hline
$\frac{\VAR{ \hat \alpha_{qq}}}{(\log_2\e)^2}$ & $2\sum_{j,j'=j_1}^{j_2}\frac{w_{j}w_{j'}}{n_{j}n_{j'}}\sum_{k,k'=1}^{n_j,n_j'}{\ccoef_{qq}^{2}(j,k;j',k')}{}$\\
\hline
$\frac{\VAR{\hat  \alpha_{q_1q_2}}}{(\log_2\e)^2}$ & $\sum_{j,j'=j_1}^{j_2}\frac{w_{j}w_{j'}}{n_{j}n_{j'}}\sum_{k,k'=1}^{n_j,n_j'}\frac{\ccoef_{q_1q_2}(j,k;j',k')\ccoef_{q_1q_2}(j',k';j,k)+\ccoef_{q_1q_1}(j,k;j',k')\ccoef_{q_2q_2}(j,k;j',k')}{\ccoef_{q_1q_2}(j,k;j,k)\ccoef_{q_1q_2}(j',k';j',k')}$\\
\hline
$\frac{\COV{ \hat \alpha_{q_1q_1}}{\hat \alpha_{q_2q_2}}}{(\log_2\e)^2}$ & $2\sum_{j,j'=j_1}^{j_2}\frac{w_{j}w_{j'}}{n_{j}n_{j'}}\sum_{k,k'=1}^{n_j,n_j'}\frac{\ccoef_{q_1q_2}^2(j,k;j',k')}{\ccoef_{q_1q_1}(j,k;j,k)\ccoef_{q_2q_2}(j',k';j',k')}$\\
\hline
$\frac{\COV{ \hat \alpha_{q_1q_1}}{\hat \alpha_{q_1q_2}}}{(\log_2\e)^2} $ & 2$\sum_{j,j'=j_1}^{j_2}\frac{w_{j}w_{j'}}{n_{j}n_{j'}}\sum_{k,k'=1}^{n_j,n_j'}\frac{\ccoef_{q_1q_1}(j,k;j',k')\ccoef_{q_1q_2}(j,k;j',k')}{\ccoef_{q_1q_1}(j,k;j,k)\ccoef_{q_1q_2}(j',k';j',k')}$\\
\hline
\end{tabular}
\caption{\label{tab:vara}\textbf{Closed-form expression for $\COV{ \hat \alpha_{q_1q_1}}{\hat \alpha_{q_1q_2}}$.}
}
\end{table}

\subsubsection{Impact of inter- and intra-scale correlations}

The expression of $\COV{ \hat \alpha_{q_1q_2}}{\hat \alpha_{q_3q_4}} $ can further be split into three terms, namely,
\begin{multline}
\label{equ:varaxy}
\COV{ \hat \alpha_{q_1q_2}}{\hat \alpha_{q_3q_4}} \approx(\log_2\e)^2\sum_{j=j_1}^{j_2}\Bigg[\frac{w_{j}^2}{n_j}\left(1+\frac{1}{\ccoef_{q_1q_2}(j,0;j,0)\ccoef_{q_3q_4}(j,0;j,0)}\right) \\
+\frac{w_{j}^2}{n_{j}^2}\sum_{k}\sum_{k'\ne k} \frac{\ccoef_{q_1q_3}(j,k;j,k')\ccoef_{q_2q_4}(j,k;j,k')+\ccoef_{q_1q_4}(j,k;j,k')\ccoef_{q_2q_3}(j,k;j,k')}{\ccoef_{q_1q_2}(j,k;j,k)\ccoef_{q_3q_4}(j,k';j,k')} \\
+\sum_{j'\ne j}\frac{w_{j}w_{j'}}{n_{j}n_{j'}}\sum_{k}\sum_{k'} \frac{\ccoef_{q_1q_3}(j,k;j',k')\ccoef_{q_2q_4}(j,k;j',k')+\ccoef_{q_1q_4}(j,k;j',k')\ccoef_{q_2q_3}(j,k;j',k')}{\ccoef_{q_1q_2}(j,k;j,k)\ccoef_{q_3q_4}(j',k';j',k')}\Bigg],
\end{multline}
where the first term in the sum over $j$ reflects the variance only of wavelet coefficients, the second term the covariance of wavelet coefficients at a given scale, and the third term, the covariance of wavelet coefficients at different scales.
In other words, if wavelet coefficients were independent, the second and third terms would equal zero.

The relative contributions of the three terms to the final variances are quantified by means of Monte Carlo simulations conducted following the same protocol and settings as those described in Section \ref{s:numerical_sim_setting}. Table \ref{tab:ratiomethoda}, reporting the relative contributions of each of the three terms for various sample sizes under fractal connectivity (i.e., $\delta_{q_1q_2} \equiv 0$), clearly shows that the second and third terms (intra- and inter-scale covariances) cannot be neglected, namely, one cannot use only the first term (variance) in the construction of confidence intervals for fBm, as proposed in \cite{Abry1998}. Identical conclusions, not shown here, are drawn under departures from fractal connectivity  (i.e., $\delta_{q_1q_2} > 0$).

\begin{table}[ht]
\setlength{\tabcolsep}{2.15pt}
\footnotesize
\centering
\begin{tabular}{| c || c | c | c | c || c | c | c | c || c | c | c | c || c | c | c | c |}\hline
$n$&$2^{10}$&$2^{12}$&$2^{14}$&$2^{16}$&$2^{10}$&$2^{12}$&$2^{14}$&$2^{16}$&$2^{10}$&$2^{12}$&$2^{14}$&$2^{16}$&$2^{10}$&$2^{12}$&$2^{14}$&$2^{16}$\\\hline \hline
& \multicolumn{4}{c||}{$\VAR{\widehat{\alpha}_{11}}$}& \multicolumn{4}{c||}{$\VAR{\widehat{\alpha}_{22}}$}& \multicolumn{4}{c||}{$\VAR{\widehat{\alpha}_{12}}$}& \multicolumn{4}{c|}{$\VAR{\widehat{\delta}_{12}}$}\\\hline
var$\times 10^{3}$&65.33         &9.51         &1.87         &0.42        &69.27        &10.05         &1.97         &0.44        &18.76         &2.73         &0.54         &0.12 &12.42         &1.81        &0.35         &0.08 \\\hline
term 1&0.74        &0.68        &0.66        &0.66        &0.69        &0.65        &0.63        &0.62        &0.72        &0.67        &0.65        &0.64 &0.71           &0.66        &0.65        &0.64 \\
term 2&0.15        &0.14        &0.14        &0.14         &0.20        &0.19        &0.18        &0.18        &0.17        &0.16        &0.16        &0.16 &0.18           &0.17        &0.16        &0.16 \\
term 3&0.11        &0.17         &0.20         &0.20        &0.11        &0.17        &0.19        &0.19        &0.11        &0.17        &0.19         &0.20 &0.11           &0.17        &0.19         &0.2 \\\hline
& \multicolumn{4}{c||}{$\COV{\widehat{\alpha}_{11}}{\widehat{\alpha}_{22}}$}& \multicolumn{4}{c||}{$\COV{\widehat{\alpha}_{11}}{\widehat{\alpha}_{12}}$}& \multicolumn{4}{c||}{$\COV{\widehat{\alpha}_{22}}{\widehat{\alpha}_{12}}$}& \multicolumn{4}{c|}{}\\\hline
co $\times 10^{3}$&54.47         &7.92         &1.55         &0.35        &33.12         &4.82         &0.95         &0.21         &34.1         &4.95         &0.97         &0.22 & & & & \\\hline
term 1&0.72        &0.66        &0.65        &0.64        &0.73        &0.67        &0.66        &0.65        &0.71        &0.66        &0.64        &0.63 & & & & \\
term 2&0.18        &0.17        &0.16        &0.16        &0.16        &0.15        &0.15        &0.15        &0.19        &0.18        &0.17&        0.17 & & & & \\
term 3&0.11        &0.17        &0.19         &0.20        &0.11        &0.17        &0.19         &0.20        &0.11        &0.17        &0.19         &0.2 & & & & \\\hline
\end{tabular}
\caption{\label{tab:ratiomethoda}\textbf{\boldmath Relative contributions of the three terms in \eqref{equ:varaxy} to $\VAR{\widehat{\alpha}_{(\cdot)}}$, $\COV{ \hat \alpha_{q_1q_2}}{\hat \alpha_{q_3q_4}}$ and $\VAR{\widehat{\delta}_{(\cdot)}}$} for various sample sizes. ($(\alpha_{11},\alpha_{22}, \delta_{12}, \rho) = (0.2, 0.6, 0, 0.9)$, $j_1=2$ and $j_2=\{5,7,9,11\}$, $n=\{2^{10},2^{12},2^{14},2^{16}\}$).
}
\end{table}

\subsubsection{First order approximations for the variances and covariances of $\hat \alpha_{q_1q_2}$}

It is of interest to further examine the leading order approximations for the variances and covariances of $\hat \alpha_{q_1q_2}$ and $\delta_{q_1q_2}$, corresponding to neglecting all intra- and inter-scale correlations amongst wavelet coefficients.
The first order approximations neglecting all inter- and intra- scale correlations amongst wavelet coefficients of $\COV{ \hat \alpha_{q_1q_2}}{\hat \alpha_{q_3q_4}} $ and of $\VAR{\hat\delta_{q_1q_2}} $ are summarized in Table \ref{tab:varb}.

\begin{table}[h]
\centering
\begin{tabular}{| c | c |}
\hline
$\frac{\VAR{ \hat \alpha_{qq}}}{(\log_2\e)^2}$ & $ 2 \sum_{j=j_1}^{j_2}\frac{w^2_{j}}{n_{j}} $\\
\hline
$\frac{\VAR{\hat  \alpha_{q_1q_2}}}{(\log_2\e)^2}$ & $\sum_{j=j_1}^{j_2}\frac{w^2_{j}}{n_{j}} (1+ \frac{1}{\ccoef_{q_1q_2}^2(j,0;j,0)})$\\
\hline
$\frac{\COV{ \hat \alpha_{q_1q_1}}{\hat \alpha_{q_2q_2}}}{(\log_2\e)^2}$ & $2\sum_{j=j_1}^{j_2}\frac{w^2_{j}}{n_{j}} \ccoef_{q_1q_2}^2(j,0;j,0)$\\
\hline
$\frac{\COV{ \hat \alpha_{q_1q_1}}{\hat \alpha_{q_1q_2}}}{(\log_2\e)^2} $ & $ 2 \sum_{j=j_1}^{j_2}\frac{w^2_{j}}{n_{j}}$\\
\hline
\hline
$ \frac{\VAR{\hat \delta_{q_1q_2}}}{(\log_2\e)^2} $ &  $ \sum_{j=j_1}^{j_2}\frac{w^2_{j}}{n_{j}} \Big( (\ccoef_{q_1q_2}^2(j,0;j,0) + \frac{1}{\ccoef_{q_1q_2}^2(j,0;j,0)}) - 2 \Big) $\\
\hline
\end{tabular}
\caption{\label{tab:varb}\textbf{First-order approximations in \eqref{equ:VarAlphaFinal} for the variances and covariances of $\hat \alpha_{q_1q_2}$ and $\delta_{q_1q_2}$} neglecting all inter- and intra-scale correlations amongst wavelet coefficients.}
\end{table}

 The results show that $\VAR{ \hat \alpha_{q_1q_1}}$ and $\VAR{\hat  \alpha_{q_1q_2}}$, $q_1 \neq q_2$, do not depend on the actual values of the scaling exponents $( \alpha_{q_1q_1},  \alpha_{q_2q_2},  \alpha_{q_1q_2})$, which corroborates the numerical performance reported in Section~\ref{s:finite_sample_performance}.

Note that for an ideal-HfBm with fractal connectivity, it is straightforward to show that $\ccoef_{qq}(j,0;j,0) \equiv 1 $ and $\ccoef_{q_1q_2}(j,0;j,0) \equiv \rho_{q_1q_2}$ when $q_1 \neq q_2$.
While $\VAR{ \hat \alpha_{q_1q_1}}$ does not depend on correlations $\rho_{q_1q_2}$, as expected, $\VAR{\hat  \alpha_{q_1q_2}}, q_1 \neq q_2 $ does vary with $\rho_{q_1q_2}$ according to $1/\rho^2_{q_1q_2}$, showing that $\VAR{\hat  \alpha_{q_1q_2}} \rightarrow +\infty $ when $ \rho_{q_1q_2} \rightarrow 0$.
This can be interpreted as the fact that when $ \rho_{q_1q_2} \rightarrow 0$, the scaling exponent $ \alpha_{q_1q_2}$, $q_1 \neq q_2 $, loses its meaning.
Furthermore, $ \COV{ \hat \alpha_{q_1q_1}}{\hat \alpha_{q_2q_2}}$ depends on $\rho$ as $ \rho^2_{q_1q_2} $, not surprisingly indicating that when $ \rho_{q_1q_2} \rightarrow 0$ (no correlation amongst components), $ \COV{ \hat \alpha_{q_1q_1}}{\hat \alpha_{q_2q_2}}  \rightarrow 0$ (no correlation amongst estimates).

Moreover, the first order approximation of $\VAR{\hat \delta_{q_1q_2}}$ is observed not to depend on the actual value of $\delta_{q_1q_2}$, while $\VAR{\hat \delta_{q_1q_2}}$ clearly depends on $\delta_{q_1q_2}$  based on the numerical simulations reported in Section~\ref{s:finite_sample_performance}.
This can be interpreted as the fact that, for $\hat \delta_{q_1q_2}$, the first order approximation (neglecting all intra- and inter-scale correlations amongst wavelet coefficients) is not sufficient to approximate well $\VAR{\hat \delta_{q_1q_2}}$, as opposed to what is observed for $\VAR{ \hat \alpha_{q_1q_1}}$ and $\VAR{\hat  \alpha_{q_1q_2}}, q_1 \neq q_2 $.

To finish with, $\VAR{\hat \delta_{q_1q_2}}$ varies with $ \rho_{q_1q_2} $ as $ \rho^2_{q_1q_2}  + 1/ \rho^2_{q_1q_2} -2$.
This shows again that when  $ \rho_{q_1q_2} \rightarrow 0$, parameter $\delta_{q_1q_2}$ becomes irrelevant.
Moreover, it also shows that when $ \rho_{q_1q_2} \rightarrow \pm 1$, $\VAR{\delta_{q_1q_2}}  \rightarrow 0$.
This can be understood as the fact that when $ \rho_{q_1q_2} \rightarrow \pm 1$, a departure from fractal connectivity is no longer permitted, as indicated by \eqref{e:regHfBm_condition_on_the_det}. Thus, $ \rho_{q_1q_2} \rightarrow \pm 1 $ implies $ \delta_{q_1q_2} \rightarrow 0$, which is then no longer a random variable.

\subsection{Practical computation of the variances and covariances of $\hat \alpha_{q_1q_2}$ and $\delta_{q_1q_2}$}

\subsubsection{Computation of $\EE{d_{q_1}(j,k) d_{q_2}(j',k')} $  and $\ccoef_{q_1q_2}(j,k;j',k')$}

Evaluation of \eqref{equ:VarAlphaFinal} requires knowledge of the covariance between wavelet coefficients \eqref{e:r12_in_terms_of_Phi}, which will be developed here explicitly for the discrete and the dyadic wavelet transforms.
Let $h(k)$ and $g(k)$, $k=1,\cdots,L$, be the coefficients of the high pass and low pass filters of the discrete wavelet transform, respectively, and let $\uparrow_2[\cdot]$ and $\downarrow_2[\cdot]$ be the dyadic upsampling and decimation operators.
The wavelet transform of the discrete-time process component $X_q(k)$ yields, at each scale $j=1,\dots,J$, sequences of approximation coefficients $ a_q(j,k)$ and detail coefficients $ d_q(j,k)$, $q = 1,\hdots,m$. The corresponding dyadic coefficients are given by  $\tilde a_q(j,k) =  a_q(j,2^jk)$ and $\tilde d_q(j,k) = d_q(j,2^jk)$, respectively. Pick the initialization $ a_q(0,k) = X_q(k)$ (see \cite{af94} for a discussion of the initialization of the discrete wavelet transform).
At scale $j=1$, $ a_q(1,\cdot)= h\Conv a_q(0,\cdot)$ and $ d_q(1,\cdot)=g\Conv a_q(0,\cdot)$, at scale $j=2$,
$ a_q(2,\cdot)= \uparrow_2[h] \Conv a_q(1,\cdot) = \uparrow_2[h]\Conv h\Conv a_q(0,\cdot)$ and $ d_q(2,\cdot)=\uparrow_2[g]\Conv  a_q(1,\cdot) = \uparrow_2[g]\Conv h\Conv a_q(0,\cdot) $. By iteration, we obtain the sequences of detail coefficients at each scale $j=j'$, i.e.,
\begin{equation}\label{eq:Dw}
d_q(j',k) = \left(g_{j'}\Conv a_q(0,\cdot)\right)(k),\quad  d_q(j',k) =  d_q(j',2^{j'}k),
\end{equation}
where
$$
g_{j'} = \uparrow_{2^{j'-1}}[g] \Conv \big(\CONV_{j=0}^{j'-2}  \uparrow_{2^{j'}} [h]\big).
$$
Now let $\gamma_{h_{q_1q_2}}(s,t)$, $1 \leq q_1 \leq q_2 \leq m$, be the covariances of univariate fBm of indices $h_{q_1q_2}=\alpha_{q_1q_2}/2$,
\begin{equation}
\label{covdx:fbm}
\gamma_{q_1q_2}(s,t) = \rho_{q_1q_2}\sigma_{q_1}\sigma_{q_2}\{|s|^{\alpha_{q_1q_2}}+|t|^{\alpha_{q_1q_2}}-|s-t|^{\alpha_{q_1q_2}}\}.
\end{equation}
Then (cf.\ \cite{flandrin:1992}; note that a change in time scale would only result in a multiplicative constant, which we assume to be absorbed in $\rho_{q_1q_2}$ and which cancels out in the final expression for $\ccoef_{q_1q_2}$)
	\begin{align*}
\hskip5mm&
\hskip-5mm \mathbb{E}[ d_{q_1}(j,k)  d_{q_2}(j',k+\tau)]/(\sigma_{q_1}\sigma_{q_2})=\\
	=& \sum_{p}\sum_{q}g_j(p)g_{j'}(q)\mathbb{E}[a_{q_1}(0,k-p)a_{q_2}(0,k+\tau-q)]\\
	=& -\rho_{q_1q_2} \sum_{p}\sum_{q}g_j(p)g_{j'}(q)|-\tau+q-p|^{\alpha_{q_1q_2}}\\
	&+\rho_{q_1q_2} \sum_{p}g_j(p)\sum_{q}g_{j'}(q)|k-p|^{\alpha_{q_1q_2}}+\rho_{q_1q_2} \sum_{q}g_{j'}(q)\sum_{p}g_j(p)|k+\tau-q|^{\alpha_{q_1q_2}}\\
	=&-\rho_{q_1q_2} \sum_{p}\sum_{q}g_j(p)g_{j'}(q)|\tau-q+p|^{\alpha_{q_1q_2}}\qquad\qquad  \qquad\qquad \qquad  \qquad(p'=p-q)\\
	=&-\rho_{q_1q_2} \sum_{p'}\sum_qg_j(p'+q)g_{j'}(q)|\tau-p'|^{\alpha_{q_1q_2}}
	=-\rho_{q_1q_2} \sum_{p'}\sum_qg_j(p'-q)\check g_{j'}(q)|\tau-p'|^{\alpha_{q_1q_2}}\\
	=&-\rho_{q_1q_2}  \sum_{p}(g_{j}\Conv \check  g_{j'})(p)|\tau-p|^{\alpha_{q_1q_2}} =  -\rho_{q_1q_2}((g_{j}\Conv \check g_{j'})\Conv \eta_{q_1q_2})(\tau),\label{eq:computationDwDw}
\end{align*}
where
\begin{eqnarray*}
\check g_{j}(k)&=&g_{j}(L-k),\quad k=1,\dots,L,\\
\eta_{q_1q_2}(\tau) &=& |\tau|^{\alpha_{q_1q_2}}.
\end{eqnarray*}
Consequently,
\begin{equation}
\label{equ:corrdxfinal}
\ccoef_{q_1q_2}(j,k;j',k')=\rho_{q_1q_2}\frac{((g_{j}\Conv \check g_{j'})\Conv \eta_{q_1q_2})(k'-k)}{\sqrt{((g_{j}\Conv \check g_{j})\Conv \eta_{ii})(0)\;((g_{j'}\Conv \check g_{j'})\Conv \eta_{ll})(0)}}
\end{equation}
and, for the dyadic wavelet transform,
\begin{equation}
\label{equ:corrdxfinaldyad}
\tilde \ccoef_{q_1q_2}(j,k;j',k')=\ccoef_{q_1q_2}(j,2^jk;j',2^{j'}k').
\end{equation}
For given values of $\alpha_{q_1q_2}$ and $\rho_{q_1q_2}$, these expressions can be easily evaluated numerically.

\subsubsection{Practical estimation of (co)variances of $\hat \alpha_{q_1q_2}$ and of ${ \hat \delta_{q_1q_2}}$}

Evaluating \eqref{equ:VarAlphaFinal} and \eqref{eq:var} for HfBm in practice requires the unknown parameter values $\alpha_{q_1q_2}$ and $\rho_{q_1q_2}$ in \eqref{covdx:fbm}, and hence we replace them by their estimates $\widehat{\alpha}_{q_1 q_2}$ and $\widehat{\rho}_{q_1 q_2}$. The former are defined in \eqref{eq:alphaW}, and estimates $\hat\rho_{q_1q_2}$ for $\rho_{q_1q_2}$ for $q_1\neq q_2$ can be readily obtained as the cross-correlation coefficients of the difference  processes (fGn) of the process components $X_{q_1}$ and $X_{q_2}$ 
However, note that the expressions for the (co)variances of $\hat\alpha_{q_1q_2}$ in the previous sections are derived assuming knowledge of the true parameter values and can only be expected to be approximations when these are replaced by estimates.
This will be studied numerically in the next section.

\subsubsection{Assessment of the estimated (co)variances of $\hat \alpha_{q_1q_2}$ and of ${ \hat \delta_{q_1q_2}}$ by means of Monte Carlo experiments}

Monte Carlo studies were conducted following the protocol and settings described in Section~\ref{s:numerical_sim_setting}, aiming to evaluate the quality of the estimated approximations \eqref{equ:VarAlphaFinal} and \eqref{eq:var} for the (co)variances of $\hat \alpha_{q_1q_2}$ and of ${ \hat \gamma_{q_1q_2}}$. The simulations involved $1000$ independent realizations of each of two general instances of HfBm with $m=2$ components, one using the true values of the parameters $\alpha_{q_1q_2}$ and $\rho_{q_1q_2}$, and the other, their estimates $\hat\alpha_{q_1q_2}$ and $\hat\rho_{q_1q_2}$.
Four different sample sizes, $n=\{2^{10},2^{12},2^{14},2^{16}\}$ and three different values $\rho_{12}=\{0.3,0.6,0.9\}$ are investigated for the set of exponents $[\alpha_{11},\alpha_{22},\alpha_{12}]=[0.2,0.6,0.4]$.

Table \ref{tab:estimationalpha} summarizes the square roots of the ratios of the averages over realizations of (co)variance estimates and of the Monte Carlo (co)variances. The first four columns, labeled ``theo/MC", report results obtained when using theoretical parameter values and yield the following conclusions.
First, even for small sample size $n=2^{10}$ and weak correlation $\rho_{12}=0.3$, the quality of the approximations \eqref{equ:VarAlphaFinal} and \eqref{eq:var} is very good for the variances of exponents $\alpha_{q}$, $q = 1,2$, and satisfactory for the ``cross'' exponent $\alpha_{12}$, the covariance parameters and the connectivity parameter $\delta_{12}$. Second, when the sample size $n$ and correlation level $\rho_{12}$ increase, the approximation of variances and covariances becomes excellent, with maximum errors of the order of $5\%$ for $n=2^{16}$ and strong correlation $\rho_{12}=0.9$.
Finally, the last four columns of Table \ref{tab:estimationalpha}, labeled ``est/MC", report results obtained when using estimates $\hat\alpha_{q_1q_2}$ and $\hat\rho_{12}$. They indicate that replacing the true parameter values  $\alpha_{q_1q_2}$ and $\rho_{12}$ with estimates has very little impact on the quality of approximations \eqref{equ:VarAlphaFinal} and \eqref{eq:var}. Indeed, the average values of the (co)variance estimates are essentially equal to those obtained when using true parameter values.

\setlength{\tabcolsep}{2.5mm}
\begin{table}[h]
\footnotesize
\centering
\begin{tabular}{| c | c || c | c | c | c || c | c | c | c |}\hline
\multicolumn{2}{| c ||}{$n$}&$2^{10}$&$2^{12}$&$2^{14}$&$2^{16}$&$2^{10}$&$2^{12}$&$2^{14}$&$2^{16}$\\ \hline
$\rho_{12}$ & {Ratio of $\sqrt{-/-}$} & \multicolumn{4}{c ||}{{theo}/MC} & \multicolumn{4}{ c |}{{est}/MC}\\ \hline
\multirow{4}{*}{$0.3$}
& \va$[\hat{\alpha}_{11}]$ 			& $  0.95$ & $  0.97$ & $  0.95$ & $ 0.99 $ & $  0.94$ & $  0.97$ & $  0.94$ & $ 0.99 $ \\
& \va$[\hat{\alpha}_{12}]$				& $  0.82$ & $  0.85$ & $  0.90$ & $ 0.93 $ & $  0.83$ & $  0.86$ & $  0.90$ & $ 0.93 $ \\
&\cov$[\hat{\alpha}_{11},\widehat{\alpha}_{22}]$	& $  1.12$ & $  1.07$ & $  0.97$ & $ 1.32 $ & $  1.09$ & $  1.06$ & $  0.97$ & $ 1.31 $ \\
& \va$[\hat{\alpha}_{12}]$ 				& $  0.78$ & $  0.82$ & $  0.88$ & $ 0.91 $ & $  0.80$ & $  0.82$ & $  0.88$ & $ 0.91 $\\
\hline
\multirow{4}{*}{$0.6$}
& \va$[\hat{\alpha}_{11}]$ 			& $  0.98$ & $  0.98$ & $  0.96$ & $ 0.98 $ & $  0.98$ & $  0.97$ & $  0.96$ & $ 0.98 $ \\
& \va$[\hat{\alpha}_{12}]$				& $  0.94$ & $  0.92$ & $  1.00$ & $ 1.00 $ & $  0.93$ & $  0.92$ & $  1.00$ & $ 1.00 $ \\
&\cov$[\hat{\alpha}_{11},\widehat{\alpha}_{22}]$	& $  1.00$ & $  0.93$ & $  0.97$ & $ 1.05 $ & $  0.99$ & $  0.92$ & $  0.97$ & $ 1.05 $ \\
& \va$[\hat{\delta}_{12}]$ 				& $  0.88$ & $  0.87$ & $  0.98$ & $ 0.97 $ & $  0.88$ & $  0.87$ & $  0.98$ & $ 0.97 $\\
\hline
\multirow{4}{*}{$0.9$}
& \va$[\hat{\alpha}_{11}]$ 			& $  0.97$ & $  0.95$ & $  0.98$ & $ 0.98 $ & $  0.96$ & $  0.95$ & $  0.98$ & $ 0.98 $ \\
& \va$[\hat{\alpha}_{12}]$				& $  1.00$ & $  0.98$ & $  1.00$ & $ 1.01 $ & $  0.99$ & $  0.98$ & $  1.00$ & $ 1.01 $ \\
&\cov$[\hat{\alpha}_{11},\widehat{\alpha}_{22}]$	& $  1.01$ & $  0.99$ & $  1.01$ & $ 1.02 $ & $  1.00$ & $  0.98$ & $  1.01$ & $ 1.01 $ \\
& \va$[\hat{\delta}_{12}]$ 				& $  0.93$ & $  0.92$ & $  0.95$ & $ 0.97 $ & $  0.95$ & $  0.93$ & $  0.96$ & $ 0.99 $\\ \hline
\end{tabular}
\vspace{1mm}
\caption{\label{tab:estimationalpha}\textbf{\boldmath Estimation of $\VAR{\alpha_{(\cdot)}}$.} Square roots of ratios of mean of (co)variances computed using \eqref{equ:VarAlphaFinal} and of Monte Carlo (co)variances:
\eqref{equ:VarAlphaFinal} evaluated using theoretical values $\alpha_{11}$, $\alpha_{22}$, $\alpha_{12}$,
$\rho_{12}$
(left columns, labeled ``theo/MC'') and estimates  $\hat\alpha_{11}$, $\hat\alpha_{22}$, $\hat\alpha_{12}$, $\hat \rho_{12}$ (right columns, labeled ``est/MC''). ( $(\alpha_{11},\alpha_{22}, \delta_{12}, \rho_{12}) = (0.2, 0.6, 0, \rho_{12})$, $j_1=2$ and $j_2=\{5,7,9,11\}$, $n=\{2^{10},2^{12},2^{14},2^{16}\}$).
}
\end{table}

\section{Statistical test for fractal connectivity}
\label{sec:test}

\subsection{Procedure}
\label{sec:testproc}
The mathematical and computational results in Sections \ref{sec:wavestim} and \ref{sec:MC} enable us to construct component-wise fractal connectivity tests, i.e., for the hypotheses
$$
H_0:\;\delta_{q_1q_2}=\frac{\alpha_{q_1q_1}+\alpha_{q_2q_2}}{2} - \alpha_{q_1q_2}=0,\qquad  q_1\neq q_2.
$$
We assume $\rho_{q_1q_2} \neq 0$ for $q_1\neq q_2$. As a consequence of Theorem \ref{t:asymptotic_dist_estimators}, the distribution of $\hat\delta_{q_1q_2}$ under $H_0$ can be approximated over finite samples by
$$
\hat\delta_{q_1q_2}\sim \mathcal{N}\left(0,\VAR{\hat\delta_{q_1q_2}}\right)\textnormal{ under }H_0,
$$
where, in turn, $\VAR{\hat\delta_{q_1q_2}}$ can be approximated by \eqref{eq:var}. Therefore, a simple two-sided test with significance level
 $\signif=P(\textnormal{reject }H_0|H_0\textnormal{ true})$ can be defined as
\begin{equation}
\label{equ:test}
d_\signif=
\begin{cases}
1,&\textnormal{ if }\;|\hat\delta_{q_1q_2}|> \sqrt{\VAR{\hat\delta_{q_1q_2}}} \phi^{-1}(1-\signif/2);\\
0,&\textnormal{ otherwise},
\end{cases}
\end{equation}
where $\phi^{-1}(\cdot)$ is the inverse cumulative distribution function of the standard Normal distribution.
In addition, the $p$-value of the test statistic (i.e., the probability of observing an absolute value at least as large as $|\hat\delta_{q_1q_2}|$ for the test statistic under $H_0$) is given by
\begin{equation}
\label{equ:pval}
p(|\hat\delta_{q_1q_2}|):=2\phi\Big(- |\hat\delta_{q_1q_2}|\Big/\sqrt{\VAR{\hat\delta_{q_1q_2}}}\Big).
\end{equation}
This test can be performed by evaluating \eqref{equ:test} with an estimate for $\VAR{\hat\delta_{q_1q_2}}$ obtained from the procedure detailed in Section \ref{sec:CI}.

\subsection{Monte Carlo assessment of the test performance}

We assess the performance of the test by applying it to $1000$ independent realizations of HfBm with exponent values $[\alpha_{11},\alpha_{22}]=[0.2,0.6]$ and exponent values $\alpha_{12}$ detailed below for sample sizes $n=\{2^{10},2^{12},2^{14},2^{16}\}$ and correlation levels $\rho_{12}=\{0.5,0.7,0.9\}$.
For simplicity of illustration and without loss of generality, we consider again HfBm with $m=2$ components.
For each realization, the test decision \eqref{equ:test} and the $p$-value \eqref{equ:pval} are evaluated using \eqref{eq:var} with approximations \eqref{equ:VarAlphaFinal} to obtain an estimate of the $\VAR{\hat\delta_{12}}$. Estimates of the expected values of the test decisions and $p$-values, denoted by $\hat d_\signif$ and $\hat p$, are then obtained as the averages over realizations of test decisions and $p$-values \eqref{equ:test} and \eqref{equ:pval}.

We now compare the performance of the proposed test, denoted hereinafter HFBM (not to be confused with the stochastic process HfBm), to that of the test put forward in \cite{Wendt2009icassp}. The latter relies on the intuition that the wavelet coherence function of two components of a multivariate Gaussian scale invariant random process approximately behaves as
$$
\Gamma_{q_1q_2}(j)=S_{n_j}^{(q_1q_2)}(j)/\sqrt{S_{n_j}^{(q_1q_1)}(j)S_{n_j}^{(q_2q_2)}(j)}\simeq \rho_{q_1q_2} 2^{j(\alpha_{q_1q_2}-\alpha_{q_1q_1}-\alpha_{q_2q_2})}.
$$
The test itself, denoted WCF (for wavelet coherence function), is formulated without the rigorous statistical framework developed above. Rather, it is built on the observation that $\Gamma_{q_1q_2}(j)$ is the Pearson product-moment correlation coefficient of the time series $d_{q_1}(j,\cdot)$ and $d_{q_2}(j,\cdot)$, and hence that the Fisher's $z$ statistics of $\Gamma_{q_1q_2}(j)$, $j=j_1,\ldots,j_2$, are approximately Gaussian, with known variances and, in the case of fractal connectivity, with equal means across scales. The test for fractal connectivity is then formulated as the UMPI test for the equality of means of Gaussian random variables, cf.\ \cite{Wendt2009icassp} for details.

\subsubsection{Performance under $H_0$}
We first consider the case that $H_0$ is true, i.e., $(\alpha_{11}+\alpha_{22})/2=\alpha_{12}=0.4$. The significance level is set to $\signif=0.1$, results are reported in Table \ref{tab:test:alpha} for the proposed test (top) and for the test in \cite{Wendt2009icassp} (bottom). Note that, under $H_0$, averages of test decisions $\hat d_\signif$ should equal the preset significance level $\signif$, and averages of $p$-values $\hat p$ should equal $\frac{1}{2}$.
HFBM rejects $H_0$ with slightly larger probability than the prescribed value $\signif=0.1$, yet the differences between empirical significance levels $\hat d_\signif$ and $\signif$ never exceed $5\%$; similarly, average $p$-values are slightly below $0.5$. For large sample size and large $\rho_{12}$, average test decisions and $p$-values are very close to the theoretical values $\signif=0.1$ and $p=\frac{1}{2}$. These remarks are consistent with the results reported in Table \ref{tab:estimationalpha}, where a small but systematic underestimation of $\VAR{\hat\delta_{12}}$ for small sample sizes and $\rho_{12}$ is observed.

In contrast, the empirical significances $\hat d_\signif$ of WCF strongly differ from the preset value $\signif$ by values of up to $16\%$, and this difference is especially pronounced for large sample sizes for which one would expect the test to perform better. One reason for this poor performance may  lie in the fact that the test in \cite{Wendt2009icassp} was designed for fGn, rather than fBm. Note that the asymptotic calculations developed above can be adapted to the easier case of fGn (and, in principle, any other Gaussian process with stationary increments) without difficulty by simply changing the covariance function $\gamma_{q_1q_2}(s,t)$ in the calculations leading to the expressions \eqref{equ:corrdxfinal} and \eqref{equ:corrdxfinaldyad}.

\subsubsection{Test power}
We assess the power of the test under the alternative hypotheses $H_1:\;\delta_{12}=(\alpha_{11}+\alpha_{22})/2-\alpha_{12}\neq 0$ with $\delta_{12}=\{0.05,\,0.1,\,0.15,\,0.2\}$.
Yet, a direct comparison of the power of HFBM and WCF is only meaningful for identical rejection probabilities under $H_0$.
Indeed, a test for which $\hat d_\signif>\signif$ under $H_0$ is expected to display an artificially large power.
Since the performance of HFBM and WCF differ, as discussed above (cf.\ Table \ref{tab:test:alpha}), we therefore adjust, for each sample size and correlation level, the prescribed significance to the value $\tilde\signif$ for which the average rejection rate under $H_0$ equals $\hat d_{\tilde\signif}=\signif=0.1$. Using this adjusted level of significance $\tilde\signif$, the power of the test is then estimated as the average of the test decisions $d_{\tilde\signif}$ when $H_1$ is true.
Results are reported in Table \ref{tab:test:beta} and yield the following conclusions. First, the power of each test systematically increases with the magnitudes of the deviation from $\delta_{12}=0$, of the correlation level $\rho_{12}$ and of the sample size $n$, as expected. Second, HFBM is systematically and significantly more powerful. Indeed, it enables us to detect a non-zero value for $\delta_{12}$  up to two times as often as WCF. For instance, for the small sample size of $n= 2^{10}$ and the low correlation level of $\rho_{12}=0.5$, it permits the detection of a deviation of $0.2$ from the null value $\delta_{12}=0$ with probability $0.5$, as compared to a probability of $0.22$ for the test in \cite{Wendt2009icassp}.

Overall, these results confirm that the proposed developments can be relevantly applied in the assessment of scaling and fractal connectivity in multivariate time series.

\begin{table}[h]	
\centering
\footnotesize{%
\begin{tabular}{| c || r | r || r| r || r| r |}
\hline \multicolumn{7}{|c|}{HFBM -- $H_0:\;\delta_{12}\equiv 0$, $\signif=0.1$}\\
\hline
& \multicolumn{2}{c||}{$ \rho_{12}=0.5$} & \multicolumn{2}{c||}{$\rho_{12}=0.7 $}& \multicolumn{2}{c|}{$\rho_{12}=0.9 $}\\
\hline
 & $100\hat d_\signif$ & $\hat p$ & $100\hat d_\signif$ & $\hat p$ & $100\hat d_\signif$ & $\hat p$\\\hline
$n=2^{10}$& $  13.2$& $0.45$& $  14.0$& $0.48$& $  14.5$& $0.45$\\
\hline
$n=2^{12}$& $  13.8$& $0.45$& $  10.9$& $0.47$& $  11.1$& $0.45$\\
\hline
$n=2^{14}$& $  13.8$& $0.45$& $  11.2$& $0.46$& $  11.0$& $0.46$\\
\hline
$n=2^{16}$& $  12.3$& $0.46$& $  10.9$& $0.48$& $  11.0$& $0.47$\\
\hline
\hline \multicolumn{7}{|c|}{WCF -- $H_0:\;\delta_{12}\equiv 0$, $\signif=0.1$}\\
\hline
& \multicolumn{2}{c||}{$ \rho_{12}=0.5$} & \multicolumn{2}{c||}{$\rho_{12}=0.7 $}& \multicolumn{2}{c|}{$\rho_{12}=0.9 $}\\
\hline
 & $100\hat d_\signif$ & $\hat p$ & $100\hat d_\signif$ & $\hat p$ & $100\hat d_\signif$ & $\hat p$\\\hline
$n=2^{10}$& $  14.7$& $0.44$& $  15.3$& $0.46$& $  16.0$& $0.44$\\
\hline
$n=2^{12}$& $  20.7$& $0.42$& $  16.0$& $0.43$& $  17.2$& $0.42$\\
\hline
$n=2^{14}$& $  20.9$& $0.40$& $  18.6$& $0.41$& $  22.0$& $0.39$\\
\hline
$n=2^{16}$& $  26.1$& $0.36$& $  22.5$& $0.36$& $  22.1$& $0.38$\\
\hline
\end{tabular}
}
\caption{\label{tab:test:alpha}\textbf{Test significance.}
Mean test decisions and $p$-values for different values of $n$ and $\rho_{12}$ under $H_0:\;\delta_{12}=(\alpha_{11}+\alpha_{22})/2-\alpha_{12}=0$ ($[\alpha_{11},\alpha_{22}]=[0.2,0.6]$, $j_1=2$ and $j_2=\{5,7,9,11\}$, $n=\{2^{10},2^{12},2^{14},2^{16}\}$) for the proposed test (top) and for the test in \cite{Wendt2009icassp} (bottom).}
\end{table}

\setlength{\tabcolsep}{1.75mm}
\begin{table}[h]	
\centering
\footnotesize{%
\begin{tabular}{| c|| r| r | r| r || r| r | r| r || r| r | r| r |}
\hline \multicolumn{13}{|c|}{HFBM -- $H_1:\;\delta_{12}\neq0$, $\hat d_\signif=0.1$}\\
\hline
 & \multicolumn{4}{c||}{ $\rho_{12}=0.5$} & \multicolumn{4}{c||}{$\rho_{12}=0.7 $}& \multicolumn{4}{c|}{$\rho_{12}=0.9$ }\\ \hline
$\delta_{12}$ &$0.05$&$0.1$&$0.15$&$0.2$ &$0.05$&$0.1$&$0.15$&$0.2$ &$0.05$&$0.1$&$0.15$&$0.2$ \\
\hline\hline
$n=2^{10}$& $  16.5$& $  29.9$& $  40.6$& $  50.0$& $  25.4$& $  43.5$& $  66.1$& $  84.5$& $  74.8$& $  98.6$& $ 100.0$& $  99.9$\\
\hline
$n=2^{12}$& $  28.8$& $  60.8$& $  81.9$& $  93.3$& $  63.6$& $  96.8$& $  99.8$& $ 100.0$& $ 100.0$& $ 100.0$& $ 100.0$& $ 100.0$\\
\hline
$n=2^{14}$& $  67.1$& $  97.9$& $  99.9$& $ 100.0$& $  99.2$& $ 100.0$& $ 100.0$& $ 100.0$& $ 100.0$& $ 100.0$& $ 100.0$& $ 100.0$\\
\hline
$n=2^{16}$& $  99.4$& $ 100.0$& $ 100.0$& $ 100.0$& $ 100.0$& $ 100.0$& $ 100.0$& $ 100.0$& $ 100.0$& $ 100.0$& $ 100.0$& $ 100.0$\\
\hline
\hline \multicolumn{13}{|c|}{WCF -- $H_1:\;\delta_{12}\neq0$, $\hat d_\signif=0.1$}\\
\hline
 & \multicolumn{4}{c||}{ $\rho_{12}=0.5$} & \multicolumn{4}{c||}{$\rho_{12}=0.7 $}& \multicolumn{4}{c|}{$\rho_{12}=0.9$ }\\ \hline
$\delta_{12}$ &$0.05$&$0.1$&$0.15$&$0.2$ &$0.05$&$0.1$&$0.15$&$0.2$ &$0.05$&$0.1$&$0.15$&$0.2$ \\
\hline\hline
$n=2^{10}$& $  10.5$& $  15.4$& $  16.6$& $  22.3$& $  14.1$& $  21.7$& $  36.9$& $  52.9$& $  43.7$& $  81.0$& $  94.2$& $  99.2$\\
\hline
$n=2^{12}$& $  17.8$& $  35.9$& $  52.5$& $  69.4$& $  34.7$& $  78.9$& $  95.8$& $  99.4$& $  98.2$& $ 100.0$& $ 100.0$& $ 100.0$\\
\hline
$n=2^{14}$& $  43.4$& $  86.2$& $  99.0$& $ 100.0$& $  89.9$& $ 100.0$& $ 100.0$& $ 100.0$& $ 100.0$& $ 100.0$& $ 100.0$& $ 100.0$\\
\hline
$n=2^{16}$& $  93.0$& $ 100.0$& $ 100.0$& $ 100.0$& $ 100.0$& $ 100.0$& $ 100.0$& $ 100.0$& $ 100.0$& $ 100.0$& $ 100.0$& $ 100.0$\\
\hline
\end{tabular}
}
\caption{\label{tab:test:beta}\textbf{Test power for adjusted significance $\hat d_\signif=0.1$.}
Mean test decisions and $p$-values for different values of $n$, $\rho_{12}$ and $\alpha_{12}$ / alternative hypotheses $H_1:\;\delta_{12}=(\alpha_{11}+\alpha_{22})/2-\alpha_{12}\neq 0$ ($[\alpha_{11},\alpha_{22}]=[0.2,0.6]$, $j_1=2$ and $j_2=\{5,7,9,11\}$, $n=\{2^{10},2^{12},2^{14},2^{16}\}$) for the proposed test (top) and for the test in \cite{Wendt2009icassp} (bottom).}
\end{table}

\section{Conclusion}

The present contribution introduces a versatile class of multivariate stochastic processes, named Hadamard fractional Brownian motion (HfBm), for scale invariance modeling. HfBm provides a stochastic framework within which cross-component scaling laws are not directly controlled by the scaling along the main diagonal, namely, HfBm is not necessarily fractally connected.

Interestingly, the theoretical study of HfBm reveals that exact entry-wise scaling on both auto- and cross-components and departures from fractal connectivity are mathematically incompatible. In other words, there is a dichotomy in multivariate scaling modeling: either there is exact entry-wise scaling in every component combined with fractal connectivity, or departures from fractal connectivity are allowed at the price of approximate (i.e., asymptotic) scaling on the cross-components.

Our main mathematical results consist of an asymptotically normal, wavelet-based linear regression estimator for the scaling exponents, as well as asymptotically valid confidence intervals with convenient mathematical expressions. Furthermore, the Taylor expansions used in the development of the asymptotic confidence intervals lead to the construction of practical procedures for the numerical calculation of the variance of the estimates. These approximate calculations enable the study of the ubiquitous issue of the impact of neglecting intra- or inter-scale correlations amongst wavelet coefficients in the computations of variances and covariances for the estimates. We also devised an asymptotically normal hypothesis test for fractal connectivity. Again, a major feature of the designed test procedure is the fact that it can be applied to a single observed HfBm data path.

For both fractally and non-fractally connected instances, simulations demonstrate the satisfactory performance of the estimators of the scaling and fractal connectivity parameters, even for small sample size data. The estimation bias is shown to be negligible, and the variance decreases according to the inverse of the sample size. In addition, the practical computations of approximated variances and covariances of the estimates are shown to be of excellent quality, irrespective of sample size, and the Monte Carlo significance levels and powers are very close to their theoretical counterparts.

The tools developed in the present contribution pave the way for novel analysis and modeling perspectives on multivariate scaling in real-world data, in the spirit of the accomplishments in \cite{CIUCIU:2014:A}. Routines for the synthesis of HfBm, as well as for estimation, computation of confidence intervals and fractal connectivity testing will be made publicly available at time of publication.

\bibliographystyle{plain}
\bibliography{biblio}

\appendix

\vspace{5mm}
\noindent \textbf{Appendix: proofs}\\

This appendix comprises three parts, i.e., \ref{s:4th_moments_wavecoef}, \ref{s:aci_appendix} and \ref{s:CI_appendix}, which contain the proofs for Sections \ref{sec:wavsp}, \ref{sec:aci} and \ref{sec:CI}, respectively. In \ref{s:aci_appendix} and \ref{s:CI_appendix}, we assume throughout that the assumptions of Theorems \ref{t:asymptotic_dist_estimators} and \ref{t:Cov(log,log)_cross}, respectively, hold.

In the proofs, whenever convenient we will use the shorthand
\begin{equation}\label{e:a=a(n)}
a = a(n).
\end{equation}
For notational simplicity, we will assume throughout that $\sigma_q = 1$, $q= 1,\ldots,m$. Since the main diagonal entries of an HfBm behave like a perturbed (univariate) fBm, throughout the appendix we only provide proofs for cross-entry components, i.e., when the indices $q_l $ are pairwise distinct, $l = 1,2,3,4$. Whenever convenient we will use $l = 1,2,3,4$ in place of $q_l$, respectively, and also write
$$
h_{q_l q_l} = h_l, \quad h_{q_l q_p} = h_{lp},
$$
In addition, without loss of generality it will be assumed that
\begin{equation}\label{e:mu_j>0}
\varrho^{(12)}(a2^j) > 0, \quad n \in \bbN, \quad j \in \bbN
\end{equation}
(see \ref{e:rho^(q1q2)}), since otherwise we can consider $-\varrho^{(12)}(a2^j)$ instead. We will also write $\log$ instead of $\log_2$, for visual clarity. In the proofs, $C$ represents a generic constant whose value may change from one line to the next.

\section{Section \ref{sec:wavsp}}\label{s:4th_moments_wavecoef}

\noindent {\sc Proof of Proposition~\ref{p:4th_moments_wavecoef}}: For simplicity, we will write $\Xi^{jj'}(\cdot)$ instead of $\Xi^{jj',a}(\cdot)$ throughout the proof.

To show ($i$), the change of variable $s = 2^{-j}t - k$ in \eqref{e:D(j,k)} and the harmonizable representation of HfBm yield
\begin{equation}\label{e:ED(2j,k)D(2j',k')=based_on_harmonizable}
\Xi^{jj'}(a(2^jk - 2^{j'}k'))= \int_{\bbR} e^{ia(2^jk-2^{j'}k')x}f(x) \widehat{\psi}(a2^{j}x)\overline{\widehat{\psi}(a2^{j'}x)}dx,
\end{equation}
where $f(x) := \Big( |x|^{-2(h_{q_1q_2} + 1/2)}g_{q_1q_2}(x)\Big)_{q_1,q_2 = 1,\hdots,m}$. Statement ($ii$) is a consequence of the formula \eqref{e:ED(2j,k)D(2j',k')=based_on_harmonizable} with
\begin{equation}\label{e:k=k'=0,j=j'}
k=k'=0, \quad j = j'.
\end{equation}
Statement $(iii)$ follows from Lemma \ref{l:cov_and_eigenstructure_wavelet_transf_and_var}, $(ii)$, below under condition \eqref{e:assumption_a(n)_n}.

Turning to $(iv)$, establishing \eqref{e:limiting_kron} is equivalent to showing that
$$
\frac{2^{-(j+j')/2}}{n_*} \sum^{2^{j'}n_*}_{k=1}\sum^{2^{j}n_*}_{k'=1}\Big(\frac{\Xi^{jj'}_{12}(a(2^j k - 2^{j'}k'))}{a^{2h_{12}}} - \Phi^{jj'}_{12}(2^j k - 2^{j'}k')\Big)
$$
\begin{equation}\label{e:gcd(2j,2j')sum_Phijj'(z)}
+ \frac{2^{-(j+j')/2}}{n_*} \sum^{2^{j'}n_*}_{k=1}\sum^{2^{j}n_*}_{k'=1}\Phi^{jj'}_{12}(2^j k - 2^{j'}k')
\rightarrow 2^{-(j+j')/2}\textnormal{gcd}(2^j,2^{j'})\sum^{\infty}_{z = - \infty}\Phi^{jj'}_{12}(z\textnormal{gcd}(2^j,2^{j'})),
\end{equation}
$n \rightarrow \infty$.

Consider the entry $q_1=1$, $q_2 =2$ of the matrix-valued function $\Xi^{jj'}$ as in \eqref{e:ED(2j,k)D(2j',k')=based_on_harmonizable}. Based on a change of variable $y = ax$, recast the first term of the sum on the left-hand side of \eqref{e:gcd(2j,2j')sum_Phijj'(z)} as
\begin{equation}\label{e:int_exp(iyz)_ha(y)dy}
\int_{\bbR} e^{iyz} |y|^{-(2 h_{12}+1)}\rho_{12}\Big\{ g_{12}\Big( \frac{y}{a}\Big)-1\Big\} \widehat{\psi}(2^jy) \overline{\widehat{\psi}(2^{j'}y)}dy =:\int_{\bbR} e^{iyz} h_{a}(y) dy.
\end{equation}
Fix $A > 0$ and denote $\widehat{\psi}^{(l)}(x) =  \frac{d^{l}}{d x^{l}} (\Re \widehat{\psi}(x) + i \Im \widehat{\psi}(x))$, $l \geq 0$. By \eqref{e:N_psi}, \eqref{e:psihat_deriv=0} and a Taylor expansion with Lagrange residual of the real and imaginary parts of $\widehat{\psi}$, there exist functions $\lambda_1$, $\lambda_2$ on $[-A,A]$ such that
$$
\widehat{\psi}(x) = \Big( \frac{d^{N_{\psi}}}{dx^{N_{\psi}}} \Re \widehat{\psi}(x)\Big|_{\lambda_1(x)} + i \frac{d^{N_{\psi}}}{dx^{N_{\psi}}} \Im \widehat{\psi}(x)\Big|_{\lambda_2(x)} \Big)\frac{x^{N_{\psi}}}{N_{\psi}!}.
$$
Therefore, and extending this reasoning to $\widehat{\psi}'(x)$, $\widehat{\psi}''(x)$,
\begin{equation}\label{e:psi^_around_zero_u^_around_pi}
|\widehat{\psi}^{(l)}(x)| = O(|x|^{N_{\psi}-l}), \quad x \rightarrow 0, \quad l = 0,1,2.
\end{equation}
For $h_a$ as in \eqref{e:int_exp(iyz)_ha(y)dy}, we now show that
\begin{equation}\label{e:ha=h'a=h''a=0}
\textnormal{$h_{a}$, $h'_{a}$, $h''_{a}$ are differentiable and $h_{a}(0) = h'_{a}(0) = h''_a(0) = 0$}.
\end{equation}
We will only develop expressions for $y > 0$, since analogous developments hold for $y < 0$. For mathematical convenience, rewrite $h_a$ as in \eqref{e:int_exp(iyz)_ha(y)dy} as
\begin{equation}\label{e:ha(y)}
h_a(y) = y^{-(2h_{12}+1)}\rho_{12}\hspace{0.5mm} \vartheta_a(y).
\end{equation}
Hence,
\begin{equation}\label{e:h'a(y)}
h'_a(y) = \rho_{12}\{- (2 h_{12}+1) y^{-(2h_{12}+2)} \vartheta_a(y) + y^{-(2h_{12}+1)} \vartheta'_a(y)\},
\end{equation}
$$
h''_a(y) = \rho_{12}\{(2 h_{12}+1) (2 h_{12}+2) y^{-(2h_{12}+3)} \vartheta_a(y) -2 (2 h_{12}+1) y^{-(2h_{12}+2)} \vartheta'_a(y)
$$
\begin{equation}\label{e:h''a(y)}
+ y^{-(2h_{12}+1)} \vartheta''_a(y)\}.
\end{equation}
Note that, by conditions \eqref{e:N_psi} and \eqref{e:hmax<delta0=<2(1+hmin)},
\begin{equation}\label{e:2N+delta0-(2*h12+1+l)-1>0}
2N_{\psi}+ \varpi_0 - (2h_{12}+1 + l)-1 \geq 2 + \varpi_0 - 2 h_{\max} -l  > 0, \quad l = 0,1,2.
\end{equation}
Then, $h_a$ is also smooth around zero and $h_a(0) = 0$. Moreover, by \eqref{e:psi^_around_zero_u^_around_pi} and \eqref{e:2N+delta0-(2*h12+1+l)-1>0} with $l = 0$, for fixed $n$ and $|y| \leq a \varepsilon_0$,
\begin{equation}\label{e:|h_a(y)|/y->0}
\frac{|h_a(y)|}{y} \leq  \frac{C}{y}y^{-(2h_{12}+ 1)} \Big(\frac{y}{a}\Big)^{\varpi_0} y^{2N_{\psi}} = \frac{C}{a^{\varpi_0}} y^{2N_{\psi}+ \varpi_0 - (2h_{12}+2)} \rightarrow 0,
\end{equation}
as $y \rightarrow 0^{+}$, where the constant $C > 0$ does not depend on $n$. Similarly, by \eqref{e:g_is_Schwartz}, \eqref{e:psi^_around_zero_u^_around_pi} and \eqref{e:2N+delta0-(2*h12+1+l)-1>0} with $l = 1$,
$$
\frac{|h'_a(y)|}{y} \leq \frac{C'}{y} \Big\{y^{-(2h_{12}+ 2)}\Big(\frac{y}{a}\Big)^{\varpi_0} y^{2N_{\psi}}
+ y^{-(2h_{12}+ 1)}  \Big[  \Big( \frac{y}{a}\Big)^{\varpi_0 - 1}\frac{1}{a} y^{2N_{\psi}} + \Big( \frac{y}{a}\Big)^{\varpi_0} y^{2N_{\psi}-1}\Big] \Big\}
$$
\begin{equation}\label{e:|h'a(y)|/y->0}
\leq C' \frac{y^{2 N_{\psi} - (2h_{12}+2) + \varpi_0 -1}}{a^{\varpi_0}} \rightarrow 0.
\end{equation}
This proves \eqref{e:ha=h'a=h''a=0}, as desired. Next, note that
$$
\frac{\partial}{\partial x}e^{itx} \psi(t) = it e^{itx}\psi(t), \quad \frac{\partial^2}{\partial x^2}e^{itx} \psi(t) = -t^2 e^{itx}\psi(t)
$$
and $it \psi(t)$, $t^2 \psi(t) \in L^1(\bbR)$ by the continuity of $\psi$ and condition \eqref{e:W2}. Therefore, by the dominated convergence theorem, $\widehat{\psi}'(x) = C \int_{\bbR} e^{itx} it \psi(t) dt$, $\widehat{\psi}''(x) = C' \int_{\bbR} e^{itx} (-t^2) \psi(t) dt$ for appropriate constants $C, C' \in \bbR$. Consequently, by \eqref{e:W2},
\begin{equation}\label{e:psi^,psi^',psi^''_are_bounded}
\max_{l=0,1,2} \sup_{x \in \bbR}|\widehat{\psi}^{(l)}(x)| \leq C \int_{\bbR} |t|^l |\psi(t)| dt < \infty.
\end{equation}
So, fix $z \neq 0$. By \eqref{e:g_is_bounded} and \eqref{e:psi^,psi^',psi^''_are_bounded},
\begin{equation}\label{e:integ_by_parts_limits_are_zero}
\lim_{|y|\rightarrow \infty}\Big|h_{a}(y)\Big| = 0 = \lim_{|y|\rightarrow \infty}\Big|h'_{a}(y)\Big|.
\end{equation}
Thus, in view of \eqref{e:int_exp(iyz)_ha(y)dy}, \eqref{e:ha=h'a=h''a=0}, \eqref{e:h'a(y)}, \eqref{e:h''a(y)}, \eqref{e:2N+delta0-(2*h12+1+l)-1>0} (with $l =2$) and \eqref{e:integ_by_parts_limits_are_zero}, by integrating by parts twice we obtain
$$
\Big| \frac{\Xi^{jj'}_{12}(az)}{a^{2h_{12}}} - \Phi^{jj'}_{12}(z)\Big| = \Big|\frac{1}{z^2} \int_{\bbR} e^{izy} h''_a(y)dy \Big|
$$
\begin{equation}\label{e:|int_eiyz_ha(y)dy|=<C/z2}
\leq \frac{C}{z^2}  \int_{\bbR}  \{|y|^{-(2h_{12}+3)} |\vartheta_a(y)| + |y|^{-(2h_{12}+2)} |\vartheta'_a(y)| + |y|^{-(2h_{12}+1)} |\vartheta''_a(y)|\} dy \leq \frac{C'}{z^2},
\end{equation}
where the last inequality is a consequence of \eqref{e:g_is_bounded}, \eqref{e:psi^_around_zero_u^_around_pi} and \eqref{e:psi^,psi^',psi^''_are_bounded}.

Now consider the first summation term in \eqref{e:gcd(2j,2j')sum_Phijj'(z)}. We proceed as in the proof of Proposition 3.3, $(i)$, in \cite{abry:didier:2016:supplementary} to establish that
\begin{equation}\label{e:summation_deviation_HfBm_idealHfBm}
\frac{1}{n_*} \sum^{2^{j'}n_*}_{k=1}\sum^{2^{j}n_*}_{k'=1}\Big(\frac{\Xi^{jj'}_{12}(a(2^j k - 2^{j'}k'))}{a^{2h_{12}}} - \Phi^{jj'}_{12}(2^j k - 2^{j'}k')\Big)
\rightarrow 0, \quad n \rightarrow \infty.
\end{equation}
We outline the main steps for the reader's convenience. By Theorem 1.8 in Jones and Jones \cite{jones:jones:1998}, p.\ 10, the set of values $r \in \bbZ$ to the equation $a2^j k - a2^{j'}k'= r$, $k, k' \in \bbZ$, is given by $\textnormal{gcd}(a2^j,a2^{j'})\bbZ =: {\mathcal R}$. Therefore, a pair $(k,k') \in \bbZ^2$ is a solution to
\begin{equation}\label{e:ajk-aj'k'=wgcd(aj,aj')}
a2^j k - a2^{j'}k'= \textnormal{gcd}(a2^j,a2^{j'})w
\end{equation}
for some $w \in \bbZ$ if and only if it is a solution to
\begin{equation}\label{e:2jk-2j'k'=wgcd(2j,2j')}
2^j k - 2^{j'}k'= \textnormal{gcd}(2^j,2^{j'})w
\end{equation}
for the same $w$. Therefore, we can replace $n$ with $n_*$ in Lemmas B.2 and B.3, \cite{abry:didier:2016:supplementary}, and reexpress the first summation term on the left-hand side of \eqref{e:summation_deviation_HfBm_idealHfBm} as
\begin{equation}\label{e:sum_xi(n/a)/(n/a)*(Xi-Phi/a2h12)}
\sum_{r \in {\mathcal R \cap B_{jj'}(n_*)}} \frac{\xi_r(n_*)}{n_*} \hspace{1mm} \Big(\frac{\Xi^{jj'}_{12}(ar) - \Phi^{jj'}_{12}(ar)}{a^{2h_{12}}}\Big).
\end{equation}
In \eqref{e:sum_xi(n/a)/(n/a)*(Xi-Phi/a2h12)}, $B_{jj'}(n_*)$ is the range for $r$ such that the pairs $(k,k')$ satisfying \eqref{e:2jk-2j'k'=wgcd(2j,2j')} lie in the region
\begin{equation}\label{e:1=<k=<2j'n,1=<k'=<2jn}
1 \leq k \leq 2^{j'}n_*, \quad 1 \leq k' \leq 2^{j}n_*,
\end{equation}
and $\xi_r(n_*)$ is the number of such solution pairs $(k,k')$ given some $r$. Moreover, for any sufficiently large $n$, let $k_0 \in \{1,\hdots,2^{j'}n_*\}$ be the smallest number such that $(k_0,k'(k_0)) \in \bbN^2$ solves \eqref{e:ajk-aj'k'=wgcd(aj,aj')} (for some $w \in \bbZ$), where
\begin{equation}\label{e:k'(k)}
k'(k) := \frac{2^j}{2^{j'}}k - \frac{\textnormal{gcd}(2^j,2^{j'})w}{2^{j'}}.
\end{equation}
From the proof of Lemma B.2, \cite{abry:didier:2016:supplementary}, the set ${\mathcal A}$ of such solutions to \eqref{e:2jk-2j'k'=wgcd(2j,2j')} has the form
\begin{equation}\label{e:k**=k*+aj'N_set_of_solutions}
{\mathcal A} = \Big\{(k,k') \in \bbZ^2: k = k_0 + \frac{2^{j'}}{\gcd(2^j,2^{j'})} \bbZ ,\hspace{2mm} k' \textnormal{ is given by \eqref{e:k'(k)}}\Big\}.
\end{equation}
In light of \eqref{e:k**=k*+aj'N_set_of_solutions}, define the function $k(z) = k_0 + \frac{2^{j'}}{\textnormal{gcd}(2^j,2^{j'})}z$, $z \in \bbZ$. In particular, $(k(0), k'(k(0)))$ is a solution pair for \eqref{e:2jk-2j'k'=wgcd(2j,2j')}. Let $\bbR \ni x = \gcd(2^j,2^{j'}) (n_* - k_0/2^{j'})$. Then, by \eqref{e:k**=k*+aj'N_set_of_solutions}, $(k(\lfloor x \rfloor),k'(k(\lfloor x \rfloor)) )$ is the rightmost solution for \eqref{e:ajk-aj'k'=wgcd(aj,aj')} within the first-entry range $k = 1,\hdots,2^{j'}n_*$. Moreover, given $r$, the number of solution pairs in the region \eqref{e:1=<k=<2j'n,1=<k'=<2jn} is $\xi_r(n_*) = \lfloor x \rfloor + 1$, where
\begin{equation}\label{e:integer_part_x*(a/n)->gcd}
\lfloor x \rfloor n^{-1}_{*} \rightarrow \gcd(2^j,2^{j'}), \quad n \rightarrow \infty.
\end{equation}
In addition, by \eqref{e:|int_eiyz_ha(y)dy|=<C/z2}, and \eqref{e:Phi(j,j')q1q2(z)},
\begin{equation}\label{e:|Xi-Phi|=<1/r2}
\Big| \frac{\Xi^{jj'}_{12}(ar) - \Phi^{jj'}_{12}(ar)}{a^{2h_{12}}}\Big| \leq \frac{C}{r^2}, \hspace{2mm}r \neq 0,
\quad \lim_{n \rightarrow \infty}\Big| \frac{\Xi^{jj'}_{12}(ar) - \Phi^{jj'}_{12}(ar)}{a^{2h_{12}}}\Big| =0, \hspace{2mm}r \in \bbZ.
\end{equation}
By expression \eqref{e:sum_xi(n/a)/(n/a)*(Xi-Phi/a2h12)}, the dominated convergence theorem and \eqref{e:integer_part_x*(a/n)->gcd}, the limit \eqref{e:summation_deviation_HfBm_idealHfBm} holds.

Next recall that, up to a change of sign, $\Phi^{jj'}_{12}(z)$ corresponds to the cross moment of the wavelet transform of a fBm. Thus, by an analogous procedure, we also obtain
\begin{equation}\label{e:(1/n*)sum_sum_Phi(2jk-2j'k')->gcd(2j,2j')sum_Phi(zgcd(2j,2j'))}
\frac{1}{n_*} \sum^{2^{j'}n_*}_{k=1}\sum^{2^{j}n_*}_{k'=1}\Phi^{jj'}_{12}(2^j k - 2^{j'}k')
\rightarrow \textnormal{gcd}(2^j,2^{j'})\sum^{\infty}_{z = - \infty}\Phi^{jj'}_{12}(z\textnormal{gcd}(2^j,2^{j'})), \quad n \rightarrow \infty.
\end{equation}
This establishes \eqref{e:gcd(2j,2j')sum_Phijj'(z)}.

To show statement ($v$), first note that
$$
o\Big(\frac{a^{2\max\{h_{13}+h_{24},h_{14}+h_{23}\} - 2(h_{12}+h_{34})}}{n_*}\Big)\Big/ \frac{a^{(h_1 + h_2 + h_3 + h_4)- 2(h_{12}+h_{34})}}{n_*}=o(1),
$$
which follows from \eqref{e:assumption_a(n)_n} and the fact that 
\begin{equation}\label{e:2max[h13+h24,h14+h23]=<h1+h2+h3+h4}
2\max\{h_{13}+h_{24},h_{14}+h_{23}\}\leq h_1 + h_2 + h_3 + h_4. 
\end{equation}
Thus, by expression \eqref{e:Cov(W12,W34)} for $\cov[W^{(12)}_n(a2^j),W^{(34)}_n(a2^{j'})]$ (established in the proof of Lemma \ref{l:E(Wnj-1)(Wnj'-1)} below),
$$
\frac{\sqrt{n_{a,j}}}{a^{\delta_{12}}}\frac{\sqrt{n_{a,j'}}}{a^{\delta_{34}}}\Cov\Big[W^{(12)}_{n}(a2^j), W^{(34)}_{n}(a2^{j'}) \Big]
= \frac{2^{\frac{j+j'}{2}}n_*}{a^{\delta_{12}+\delta_{34}}}\Cov\Big[W^{(12)}_{n}(a2^j), W^{(34)}_{n}(a2^{j'}) \Big]
$$
$$
= \frac{2^{-\frac{(j+j')}{2}}}{[\Phi^{jj}_{12}(0)\Phi^{j'j'}_{34}(0) (1+ O(a^{-\varpi_0}))^2 ]}
$$
$$
\cdot \Big\{\frac{a^{2(h_{13}+h_{24}) }}{a^{h_1 + h_2 + h_3 + h_4}}  \frac{1}{n_*} \sum^{2^{j'}n_*}_{k = 1}\sum^{2^{j}n_*}_{k' = 1}\Phi^{jj'}_{13}(2^{j}k-2^{j'}k')\Phi^{jj'}_{24}(2^{j}k-2^{j'}k')
$$
\begin{equation}\label{e:sqrt(n)sqrt(n)/aa_Cov(W,W)}
+ \frac{a^{2(h_{14}+h_{23})}}{a^{h_1 + h_2 + h_3 + h_4}} \frac{1}{n_*}\sum^{2^{j'}n_*}_{k = 1}\sum^{2^{j}n_*}_{k' = 1}\Phi^{jj'}_{14}(2^{j}k-2^{j'}k')\Phi^{jj'}_{23}(2^{j}k- 2^{j'}k') \Big\} + o(1).
\end{equation}
By \eqref{e:2max[h13+h24,h14+h23]=<h1+h2+h3+h4} and \eqref{e:sqrt(n)sqrt(n)/aa_Cov(W,W)}, statement $(v)$ holds.

The argument for showing ($vi$) is an adaptation of the proof of Theorem 3.1 in \cite{abry:didier:2016:supplementary} (see also \cite{bardet:2000}, pp.510-513; \cite{bardet:2002}, p.997; and \cite{istas:lang:1997}, Lemma 2). For notational simplicity, we only write the proof for $m = 2$, where the entries are indexed 1 and 2, and $q_1, q_2 = 1,2$ denote generic entries. The proof is by means of the Cram\'{e}r-Wold device. Under \eqref{e:mu_j>0}, form the vector of wavelet coefficients
$$
{\mathbf Y}_n := \Big( d_1(a2^{j_1},1),d_2(a2^{j_1},1), \hdots, d_1(a2^{j_1},n_{a,j_1}),d_2(a2^{j_1},n_{a,j_1});\hdots;
$$
\begin{equation}\label{e:Yn}
d_1(a2^{j_2},1),d_2(a2^{j_2},1), \hdots, d_1(a2^{j_2},n_{a,j_2}),d_2(a2^{j_2},n_{a,j_2})\Big) \in \bbR^{\Upsilon(n)},
\end{equation}
where
\begin{equation}\label{e:Upsilon(n)}
\Upsilon(n) := 2 \hspace{1mm}\sum^{j_2}_{j=j_1}n_{a,j}.
\end{equation}
Let
\begin{equation}\label{e:alphavec}
{\boldsymbol \theta} = ({\boldsymbol \theta}_{j_1},\hdots,{\boldsymbol \theta}_{j_2}) \in \bbR^{3J},
\end{equation}
where $J = j_2 - j_1 + 1$ and ${\boldsymbol \theta}_{j} = (\theta_{j,1},\theta_{j,12},\theta_{j,2})^* \in \bbR^3$, $j = j_1,\hdots,j_2$. Now form the block-diagonal matrix $D_n$ defined by
\begin{equation}\label{e:block-diagonal_matrix_D}
\textnormal{diag}\Big( \underbrace{\frac{1}{n_{a,j_1}}\sqrt{\frac{1}{2^{j_1}}}\Omega_{n,j_1}, \hdots, \frac{1}{n_{a,j_1}}\sqrt{\frac{1}{2^{j_1}}}\Omega_{n,j_1}}_\text{$n_{a,j_1}$};\hdots;
\underbrace{\frac{1}{n_{a,j_2}}\sqrt{\frac{1}{2^{j_2}}}\Omega_{n,j_2}, \hdots, \frac{1}{n_{a,j_2}}\sqrt{\frac{1}{2^{j_2}}}\Omega_{n,j_2}}_\text{$n_{a,j_2}$}\Big),
\end{equation}
where
\begin{equation}\label{e:Omega-nj}
\Omega_{n,j}
= \left(\begin{array}{cc}
\frac{\theta_{j,1}}{\EE{d^2_{1}(a2^j,0)}}  & \frac{1}{a^{\delta_{12}}}\frac{\theta_{j,12}}{2 \EE{d_{1}(a2^j,0)d_{2}(a2^j,0)}} \\
\frac{1}{a^{\delta_{12}}}\frac{\theta_{j,12}}{2 \EE{d_{1}(a2^j,0)d_{2}(a2^j,0)}}  & \frac{\theta_{j,2}}{\EE{d^2_{2}(a2^j,0)}}
\end{array}\right), \quad j = j_1,\hdots, j_2
\end{equation}
In \eqref{e:Omega-nj}, it can be understood that the slow growth factors for the main diagonal terms are $\frac{1}{a^{2\delta_{11}}} = \frac{1}{a^{2\delta_{22}}} \equiv 1$. We would like to establish the limiting distribution of the statistic
$$
T_{n} = \sum^{j_2}_{j=j_1} \frac{{\boldsymbol \theta}^*_{j}}{\sqrt{2^j}} \vecoper_{{\mathcal S}}\Big[ \Big(\frac{1}{a^{ \delta_{q_1 q_2}}}\Big)_{q_1,q_2 = 1,2} \circ W_n(a2^j) \Big]
$$
$$
= \sum^{j_2}_{j=j_1} \frac{{\boldsymbol \theta}^*_{j}}{\sqrt{2^j}} \frac{1}{n_{a,j}}\Big(\sum^{n_{a,j}}_{k=1}\frac{d^2_1(a2^j,k)}{\EE{d^2_{1}(a2^j,0)}},
\frac{1}{a^{\delta_{12}}}\sum^{n_{a,j}}_{k=1}\frac{d_1(a2^j,k)d_2(a2^j,k)}{\EE{d_{1}(a2^j,0)d_{2}(a2^j,0)}},\sum^{n_{a,j}}_{k=1}\frac{d^2_2(a2^j,k)}{\EE{d^2_{2}(a2^j,0)}}\Big)^*
$$
$$
= {\mathbf Y}^*_n D_n {\mathbf Y}_n,
$$
where it suffices to consider ${\boldsymbol \theta}$ in \eqref{e:alphavec} such that
\begin{equation}\label{e:only_alphas_for_which_the_limit_is_nontrivial}
{\boldsymbol \theta}^* \Sigma(H) {\boldsymbol \theta} > 0
\end{equation}
(see \cite{brockwell:davis:1991}, pp.211 and 214). The matrix $\Sigma(H)$ in \eqref{e:only_alphas_for_which_the_limit_is_nontrivial} is obtained from \eqref{e:Gjj'} and can be written in block form as $\Sigma(H) =: (G^{jj'})_{j,j'=j_1,\hdots,j_2}$, corresponding to block entries of the vector ${\boldsymbol \theta} = (\theta_{j_1},\hdots,\theta_{j_2})^*$. Let $\Gamma_{{\textbf Y}_n} = \COV{{\mathbf Y}_n}{{\mathbf Y}_n}$, and consider the spectral decomposition $\Gamma^{1/2}_{\textbf{Y}_n} D_n \Gamma^{1/2}_{\textbf{Y}_n} = O \Lambda O^*$, where $\Lambda$ is diagonal with real, and not necessarily positive, eigenvalues
\begin{equation}\label{e:xi_i,n}
\xi_{i}(a2^j) , \quad i = 1,\hdots,\Upsilon(n),
\end{equation}
and $O$ is an orthogonal matrix. Now let ${\mathbf Z} \sim N(0,I_{\Upsilon(n)})$. Then,
$$
T_{n} \stackrel{d}= {\mathbf Z}^* \Gamma^{1/2}_{\textbf{Y}_n} D_n \Gamma^{1/2}_{\textbf{Y}_n} {\mathbf Z} = {\mathbf Z}^* O \Lambda  O^* {\mathbf Z} \stackrel{d}={\mathbf Z}^* \Lambda {\mathbf Z} =: \sum^{\Upsilon(n)}_{i=1}\xi_{i}(a2^j) Z^{2}_i.
$$
Assume for the moment that
\begin{equation}\label{e:max_eig=o(1/(2^J/2))}
\max_{i=1,\hdots,\Upsilon(n)}|\xi_{i}(a2^j)| = o\Big(\Big(\frac{a}{n} \Big)^{1/2}\Big).
\end{equation}
By \eqref{e:only_alphas_for_which_the_limit_is_nontrivial} and \eqref{e:Gjj'},
$$
\frac{n}{a} \VAR{T_{n}} = \sum^{j_2}_{j=j_1}\sum^{j_2}_{j'=j_1}{\boldsymbol \theta}^*_{j} \Big\{ \sqrt{\frac{n}{a2^j}} \sqrt{\frac{n}{a2^{j'}}}\cov \Big[\hspace{1mm}\vecoper_{{\mathcal S}}\Big[\Big(\frac{1}{a^{\delta_{q_1 q_2}}}\Big)_{q_1,q_2 = 1,2} \circ W_n(a2^j)\Big],
$$
$$
{\vecoper_{{\mathcal S}}\Big[\Big(\frac{1}{a^{ \delta_{q_1 q_2}}}\Big)_{q_1,q_2 = 1,2} \circ W_n(a2^{j'})\Big]} \hspace{1mm}\Big] \Big\} {\boldsymbol \theta}_{j'}
\rightarrow \sum^{j_2}_{j=j_1}\sum^{j_2}_{j'=j_1}{\boldsymbol \theta}^*_{j} \hspace{1mm}G^{jj'} \hspace{1mm}  {\boldsymbol \theta}_{j'} > 0.
$$
Therefore, there exists a constant $C > 0$ such that, for large enough $n$, $\frac{n}{a} \VAR{T_{n}} \geq C > 0$. In view of condition \eqref{e:max_eig=o(1/(2^J/2))},
$$
\frac{\max_{i=1,\hdots,\Upsilon(n)}|\xi_{i}(a2^j)|}{\sqrt{\VAR{T_{n}}}} \leq C' \Big(\frac{n}{a}\Big)^{1/2} \max_{i=1,\hdots,\Upsilon(n)}|\xi_{i}(a2^j)| \rightarrow 0, \quad n \rightarrow \infty.
$$
The claim \eqref{e:clt_wavelet_coef} is now a consequence of Lemma B.4 in \cite{abry:didier:2016:supplementary}.

So, we need to show \eqref{e:max_eig=o(1/(2^J/2))}. The first step is to establish the bound
\begin{equation}\label{e:maxeig_Gamma1/2DGamma1/2=<maxblockD*maxeigGamma}
\sup_{\textbf{u} \in S^{\Upsilon(n) - 1}} |\textbf{u}^* \Gamma^{1/2}_{\textbf{Y}_n} D_n \Gamma^{1/2}_{\textbf{Y}_n} \textbf{u} | \leq C
\max_{j=j_1,\hdots,j_2}\frac{1}{n_{a,j}} \|\Omega_{n,j}\| \hspace{1mm} \sup_{\textbf{u} \in S^{\Upsilon(n) - 1}} \textbf{u}^* \Gamma_{{\textbf Y}_n} \textbf{u}.
\end{equation}
Let $\textbf{u} \in S^{\Upsilon(n) - 1}$ and let $\textbf{v} =  \Gamma^{1/2}\textbf{u} $. We can break up the vector $\textbf{v}$ into two-dimensional subvectors $v_{\cdot,\cdot}$ to reexpress $\textbf{v} = (v_{j_1,1}, \hdots, v_{j_1,n_{a,j_1}}; \hdots; v_{j_2,1}, \hdots, v_{j_2,n_{a,j_2}})^*$. Based on the block-diagonal structure of $D_n$ expressed in \eqref{e:block-diagonal_matrix_D},
$$
|\textbf{u}^* \Gamma^{1/2}_{\textbf{Y}_n} D_n \Gamma^{1/2}_{\textbf{Y}_n} \textbf{u}| =| \textbf{v}^* D_n \textbf{v} |
=  \Big|\sum^{j_2}_{j=j_1}\sum^{n_j}_{l=1}v^*_{j,l} \hspace{1mm}\frac{\Omega_{n,j}}{n_{a,j} \sqrt{2^j}} \hspace{1mm} v_{j,l}  \Big|
\leq C \sum^{j_2}_{j=j_1}\sum^{n_{a,j}}_{l=1} \frac{1}{n_{a,j} \sqrt{2^j}}\|\Omega_{n,j}\|  \hspace{1mm} \|v_{j,l}\|^2
$$
\begin{equation}\label{e:maxeig_Gamma1/2DGamma1/2=<maxblockD*maxeigGamma_proof}
\leq C \Big( \max_{j=j_1,\hdots,j_2}\frac{1}{n_{a,j}\sqrt{2^j}}\|\Omega_{n,j}\|  \Big)\hspace{1mm}  \sum^{j_2}_{j=j_1}\sum^{n_{a,j}}_{l=1} \|v_{j,l}\|^2
= C \Big( \max_{j=j_1,\hdots,j_2}\frac{1}{n_{a,j} \sqrt{2^j}}\|\Omega_{n,j}\|  \Big) \textbf{u}^* \Gamma_{{\textbf Y}_n} \textbf{u},
\end{equation}
where the constant $C$ comes from a change of matrix norms and only depends on the fixed dimension $m=2$. By taking $\sup_{\textbf{u} \in S^{\Upsilon(n) - 1}}$ on both sides of \eqref{e:maxeig_Gamma1/2DGamma1/2=<maxblockD*maxeigGamma_proof}, we arrive at \eqref{e:maxeig_Gamma1/2DGamma1/2=<maxblockD*maxeigGamma}.

The second step towards showing \eqref{e:max_eig=o(1/(2^J/2))} consists of analyzing the asymptotic behavior of the right-hand side of \eqref{e:maxeig_Gamma1/2DGamma1/2=<maxblockD*maxeigGamma}, as $n \rightarrow \infty$. For this, we will assume the result of Lemma \ref{l:cov_and_eigenstructure_wavelet_transf_and_var} below. So, note that
\begin{equation}\label{e:max1/K_asympt_1/2J}
\max_{j=j_1,\hdots,j_2}\frac{1}{n_{a,j}} \|\Omega_{n,j}\|  \sim C \hspace{1mm}\frac{a}{n}\frac{1}{a^{2 \min\{h_1,h_2,h_{12}\}}}\frac{1}{a^{\delta_{12}}}, \quad n \rightarrow \infty.
\end{equation}
Moreover, by \eqref{e:max_ki=<decay} below, the maximum eigenvalue of $\Gamma_{{\textbf Y}_n}$ is bounded by $\|\Gamma_{{\textbf Y}_n}\|\leq C a^{2 \max\{h_1,h_2\}}$, where $\|\cdot\|$ is the matrix Euclidean norm. Therefore, in view of \eqref{e:max1/K_asympt_1/2J} and \eqref{e:max_ki=<decay}, the right-hand side of \eqref{e:maxeig_Gamma1/2DGamma1/2=<maxblockD*maxeigGamma} is bounded by $C \frac{a^{2 (h_{\max} - h_{\min})+1}}{n}\frac{1}{a^{ \delta_{12}}}$. In turn, the latter expression divided by $\sqrt{\frac{a}{n}}$ is bounded by $C \Big( \frac{a^{4 (h_{\max} - h_{\min}) + 1}}{n}\Big)^{1/2}\frac{1}{a^{\delta_{12}}}$. By condition \eqref{e:assumption_a(n)_n}, this implies \eqref{e:max_eig=o(1/(2^J/2))}, and as a result, also \eqref{e:clt_wavelet_coef}. $\Box$\\

\begin{remark}
As a consequence of expression \eqref{e:sqrt(n)sqrt(n)/aa_Cov(W,W)} and the fact that $\frac{a^{\delta_{12}}}{\sqrt{n_{a,j}}} \rightarrow 0$ under \eqref{e:assumption_a(n)_n},
\begin{equation}\label{e:W12->1_in_L2}
W^{(12)}_{n}(a2^j) \stackrel{L^2(P)}\longrightarrow 1, \quad n \rightarrow \infty.
\end{equation}
\end{remark}

\section{Section \ref{sec:aci}}\label{s:aci_appendix}

\noindent {\sc Proof of Theorem~\ref{t:asymptotic_dist_estimators}}: Fix $q_1,q_2 = 1,\hdots,m$. Based on \eqref{e:sum_wj=0,sum_jwj=1}, rewrite
$$
\frac{1}{a^{\delta_{q_1 q_2}}}\sqrt{\frac{n}{a}}(\widehat{\alpha}_{q_1 q_2} - \alpha_{q_1 q_2})
$$
$$
= \sum^{j_2}_{j=j_1} \frac{w_j}{a^{\delta_{q_1 q_2}}} \sqrt{\frac{n}{a}} \hspace{0.5mm}\log |W^{(q_1 q_2)}_n(a2^j)|
+  \sum^{j_2}_{j=j_1} \frac{w_j}{a^{\delta_{q_1 q_2}}}\sqrt{\frac{n}{a}} \Big\{ \log \frac{\Big|\EE{S^{(q_1 q_2)}_n(a2^j)}\Big|}{a^{\alpha_{q_1 q_2}}} - \log |\Phi^{jj}_{q_1 q_2}(0)| \Big\}
$$
\begin{equation}\label{e:sqrt(n/a)(alphahat-alpha)}
+ \frac{1}{a^{\delta_{q_1 q_2}}}\sqrt{\frac{n}{a}} \Big(\sum^{j_2}_{j=j_1} w_j \hspace{0.5mm} \log |\Phi^{jj}_{q_1 q_2}(0)|  - \alpha_{q_1 q_2}\Big).
\end{equation}
By Lemma \ref{l:cov_and_eigenstructure_wavelet_transf_and_var}, $(ii)$ (below), and an application of the mean value theorem, for some $\theta(n) > 0$ between $\Big|\EE{S^{(q_1 q_2)}_n(a2^j)}\Big|/a^{\alpha_{q_1 q_2}}$ and $|\Phi^{jj}_{q_1 q_2}(0)|$ the second term in the sum \eqref{e:sqrt(n/a)(alphahat-alpha)} can be bounded in absolute value by
$$
\sum^{j_2}_{j=j_1} \frac{w_j}{a^{\delta_{q_1 q_2}}} \hspace{0.5mm}\sqrt{\frac{n}{a}} \Big\{\frac{1}{\theta(n) }\Big|\frac{\Big|\EE{S^{(q_1 q_2)}_n(a2^j)}\Big|}{a^{\alpha_{q_1 q_2}}} -  |\Phi^{jj}_{q_1 q_2}(0)|\Big|\Big\}
$$
$$
= \sum^{j_2}_{j=j_1} \frac{w_j}{a^{\delta_{q_1 q_2}}}  \hspace{0.5mm}\sqrt{\frac{n}{a}}\frac{1}{|\Phi^{jj}_{q_1 q_2}(0)|+ o(1)} \hspace{0.5mm}\frac{C_j}{a^{\varpi_0}} \leq C' \sqrt{\frac{n}{a^{2 \delta_{q_1 q_2}+1 + 2 \varpi_0}}} \rightarrow 0,
$$
as $n \rightarrow \infty$, where the limit is a consequence of condition \eqref{e:assumption_a(n)_n}. Also note that, after a change of variable $2^j x = y$ in the expression for $\Phi^{jj}_{q_1 q_2}(0)$ (see \eqref{e:Phi(j,j')q1q2(z)}),
$$
\Phi^{jj}_{q_1 q_2}(0) = 2^{j \alpha_{q_1 q_2}} \rho_{q_1 q_2} \int_{\bbR} |y|^{-(\alpha_{q_1 q_2}+1)}|\widehat{\psi}(y)|^2 dy
=: 2^{j \alpha_{q_1 q_2}} c_{q_1 q_2} \in \bbR.
$$
Therefore, by \eqref{e:sum_wj=0,sum_jwj=1}, the third term in the sum \eqref{e:sqrt(n/a)(alphahat-alpha)} can be written as
$$
\frac{1}{a^{\delta_{q_1 q_2}}}\sqrt{\frac{n}{a}}\Big(\sum^{j_2}_{j=j_1} w_j \hspace{0.5mm} (j \alpha_{q_1 q_2} + \log |c_{q_1 q_2}|)  - \alpha_{q_1 q_2} \Big) = 0.
$$
So, in regard to the first term in the sum \eqref{e:sqrt(n/a)(alphahat-alpha)}, consider the weight matrix $M \in M(\frac{m(m+1)}{2},\frac{m(m+1)}{2}J,\bbR)$ defined by
$$
\Big(2^{j_1/2}w_{j_1}I_{\frac{m(m+1)}{2}} \hspace{1mm}; 2^{(j_1+1)/2}w_{j_1+1}I_{\frac{m(m+1)}{2}}\hspace{1mm}; \hspace{2mm}\hdots \hspace{2mm}; 2^{j_2/2}w_{j_2} I_{\frac{m(m+1)}{2}}\Big),
$$
where $I_{\frac{m(m+1)}{2}}$ is a $\frac{m(m+1)}{2} \times  \frac{m(m+1)}{2}$ identity matrix and $J$ is given by \eqref{e:J=j2-j1+1}. We would like to show that
\begin{equation}\label{e:sqrt(n/a)(MWn-1)->Normal}
M\Big(  \textnormal{vec}_{{\mathcal S}}\Big[\Big( \frac{\sqrt{n_{a,j}}}{a^{ \delta_{q_1 q_2}}} \Big)_{q_1,q_2 = 1,\hdots,m} \circ \hspace{2mm}\log \circ | W_{n}(a2^j)|\Big]  \Big)_{j=j_1,\hdots,j_2}\stackrel{d}\rightarrow \NOR_{\frac{m(m+1)}{2}  }(0,MGM^*),
\end{equation}
where $\log \circ | A| := \Big(\log |A_{q_1 q_2}| \Big)_{q_1,q_2 = 1,\hdots,m}$ for any $m \times m$ real matrix $A$, and the term post-multiplying the matrix $M$ in \eqref{e:sqrt(n/a)(MWn-1)->Normal} is a $\frac{m(m+1)}{2} J$-dimensional random vector.

For any $0 < r < 1$, fix a pair $q_1$, $q_2$ and an octave $j$, which specifies one of the entries of the random vector on the left-hand side of \eqref{e:sqrt(n/a)(MWn-1)->Normal}. Define the set $A_n = \{\omega: \min_{q_1,q_2} W^{(q_1 q_2)}_n(a2^j) > r\}$. Under \eqref{e:mu_j>0}, Proposition \ref{p:4th_moments_wavecoef}, $(vi)$, implies that $P(A_n) \rightarrow 1$ as $n \rightarrow \infty$. Thus, for large enough $n$, in the set $A_n$ the mean value theorem gives the almost sure expression
\begin{equation}\label{e:logWq1q2_MVT-based_expansion}
\bbR \ni 2^{j/2}w_j \frac{\sqrt{n_{a,j}}}{a^{ \delta_{q_1 q_2}}}\log | W^{(q_1 q_2)}_n(a2^j)|
= 2^{j/2}w_j\frac{\sqrt{n_{a,j}}}{a^{ \delta_{q_1 q_2}}}\frac{(W^{(q_1 q_2)}_n(a2^j)-1)}{{\theta_{+}(W^{(q_1 q_2)}_n(a2^j))}}
\end{equation}
for a random variable $\theta_{+}(W^{(q_1 q_2)}_n(a2^j))$ between $W^{(q_1 q_2)}_n(a2^j)$ and 1. Since $W^{(q_1 q_2)}_n(a2^j) \stackrel{P}\rightarrow 1$, then $\theta_+(W^{(q_1 q_2)}_n(a2^j)) \stackrel{P}\rightarrow 1$. By considering \eqref{e:logWq1q2_MVT-based_expansion} for all $1 \leq q_1 \leq q_2 \leq m$ and $j  = j_1,\hdots,j_2$, relation \eqref{e:sqrt(n/a)(MWn-1)->Normal} is now a consequence of Proposition \ref{p:4th_moments_wavecoef}, ($vi$), and Slutsky's theorem. $\Box$\\

\section{Section \ref{sec:CI}}\label{s:CI_appendix}

\subsection{General results}

For $j = j_1,\hdots,j_2$ and $n \in \bbN$, consider the jointly Gaussian vector
$$
{\boldsymbol {\mathcal Y}}_{n} = \Big(\frac{d_{1}(a2^{j_1},1)}{\sqrt{\varrho^{(12)}(a2^{j_1})}},\hdots,\frac{d_{1}(a2^{j_1},n_{a,j_1})}{\sqrt{\varrho^{(12)}(a2^{j_1})}},
\hdots, \frac{d_{1}(a2^{j_2},1)}{\sqrt{\varrho^{(12)}(a2^{j_2})}},\hdots, \frac{d_{1}(a2^{j_2},n_{a,j_2})}{\sqrt{\varrho^{(12)}(a2^{j_2})}},
$$
\begin{equation}\label{e:Wn}
\frac{d_{2}(a2^{j_1},1)}{\sqrt{\varrho^{(12)}(a2^{j_1})}}, \hdots ,\frac{d_{2}(a2^{j_1},n_{a,j_1})}{\sqrt{\varrho^{(12)}(a2^{j_1})}},\hdots,
\frac{d_{2}(a2^{j_2},1)}{\sqrt{\varrho^{(12)}(a2^{j_2})}}, \hdots, \frac{d_{2}(a2^{j_2},n_{a,j_2})}{\sqrt{\varrho^{(12)}(a2^{j_2})}} \Big)^* \in \bbR^{\Upsilon(n)}
\end{equation}
(see \eqref{e:Upsilon(n)} for the definition of $\Upsilon(n)$). Let
\begin{equation}\label{e:Gamma_n_spectral}
\Gamma_{{\boldsymbol {\mathcal Y}}_n} = \EE{{\boldsymbol {\mathcal Y}}_{n}{\boldsymbol {\mathcal Y}}^*_{n}} = O \Lambda O^* = O \hspace{1mm}\diag(\lambda_{1,{\boldsymbol {\mathcal Y}}},\hdots, \lambda_{\Upsilon(n),{\boldsymbol {\mathcal Y}}})\hspace{1mm}O^*, \quad O \in O(\Upsilon(n)),
\end{equation}
be the associated covariance matrix and its matrix spectral decomposition. The following lemma provides the finite-sample and asymptotic properties of the covariance structure of wavelet coefficients, both from $\Xi^{jj'}$ and ${\boldsymbol {\mathcal Y}}_n$. Note that such covariance structure does not in general correspond to a multivariate stationary stochastic process when multiple octaves $j$ are considered.
\begin{lemma}\label{l:cov_and_eigenstructure_wavelet_transf_and_var}
For $j = j_1,\hdots,j_2$, $q_1,q_2 = 1,\hdots,m$, and $n \in \bbN$, the following statements hold.
\begin{itemize}
\item [$(i)$] Consider $\Xi^{jj',a}(a(n)(2^{j}k-2^{j'}k')) = \EE{D(a(n)2^j,k)D(a(n)2^{j'},k')^*}$ (see \eqref{e:Phij,j'}). For $j = j'$ and $n \in \bbN$, there is a continuous spectral density $f_{\psi,n}(x)_{q_1 q_2}$ such that we can write
\begin{equation}\label{e:Xijj/standard}
\frac{\Xi^{jj,a}_{q_1 q_2}(a(n)2^{j}(k-k'))}{(a(n)2^j)^{2 h_{q_1 q_2}}} = \int^{\pi}_{-\pi} e^{i(k-k')x}f_{\psi,n}(x)_{q_1 q_2}\hspace{1mm}dx, \quad k,k'\in \bbZ.
\end{equation}
Moreover,
\begin{equation}\label{e:f_psi,n(x)=<C}
|f_{\psi,n}(x)_{q_1 q_2}| \leq C, \quad x \in (-\pi,\pi],
\end{equation}
for a constant $C > 0$ that does not depend on $n$.
\item [$(ii)$] Let $\Phi^{jj'}_{\cdot \cdot}(z)$ be as in \eqref{e:Phi(j,j')q1q2(z)}, $z \in \bbZ$, and fix any $0 < \xi < 1$. Then,
\begin{equation}\label{e:bound_for_deviation_integral_from_that_for_fBm}
\Big|\frac{\Xi^{jj',a}_{q_1 q_2}(za(n))}{a(n)^{2h_{q_1 q_2}}} - \Phi^{jj'}_{q_1 q_2}(z) \Big| \leq  \frac{C}{a(n)^{\varpi_0}}, \quad n \rightarrow \infty,
\end{equation}
where $\varpi_0$ and $\beta$ are defined in expressions \eqref{e:g_is_Schwartz} and \eqref{e:psihat_is_slower_than_a_power_function}, respectively.
\item [$(iii)$] Let ${\boldsymbol {\mathcal Y}}_n$ be as in \eqref{e:Wn}, and let $\lambda_{i,{\boldsymbol {\mathcal Y}}}$, $i = 1,\hdots,\Upsilon(n)$, be the eigenvalues of the covariance matrix $\Gamma_{{\boldsymbol {\mathcal Y}}_n}$ (see \eqref{e:Upsilon(n)} and \eqref{e:Gamma_n_spectral}). Then, for some $C > 0$,
\begin{equation}\label{e:max_eig=<a(n)^(2maxh1,h2-2h12)}
\max_{i=1,\hdots,\Upsilon(n)}\lambda_{i,{\boldsymbol {\mathcal Y}}}\leq C a(n)^{2 (\max\{h_{1},h_{2} \} - h_{12})}.
\end{equation}
\end{itemize}
\end{lemma}
\begin{proof}
We will use the shorthand $q_1 = 1$ and $q_2 = 2$ throughout the proof.

We first show $(i)$. By making the change of variable $y = a2^jx$ in \eqref{e:ED(2j,k)D(2j',k')=based_on_harmonizable}, we obtain \eqref{e:Xijj/standard} with
\begin{equation}\label{e:f_psi(x)_discrete}
f_{\psi,n}(x)_{12} := \sum^{\infty}_{l = - \infty} \frac{|\widehat{\psi}(x + 2 \pi l)|^2}{|x + 2 \pi l|^{2 h_{12} + 1}} \hspace{1mm}\rho_{12}g_{12}\Big(\frac{x + 2 \pi l}{a2^j}\Big).
\end{equation}
Moreover, by \eqref{e:N_psi}--\eqref{e:psihat_is_slower_than_a_power_function}, $\widehat{\psi}$ is continuous, and hence, by \eqref{e:g_is_bounded}, \eqref{e:psihat_is_slower_than_a_power_function}, \eqref{e:psihat_deriv=0} and the dominated convergence theorem, the periodic function $f_{\psi,n}(x)_{12}$ in \eqref{e:f_psi(x)_discrete} is also continuous on $[-\pi,\pi]$, and hence, bounded. This shows \eqref{e:Xijj/standard} and \eqref{e:f_psi,n(x)=<C}.

We now turn to $(ii)$. It suffices to show that, for any $0 < \xi < 1$,
\begin{equation}\label{e:bound_for_deviation_integral_from_that_for_fBm-0<xi<1}
\Big|\frac{\Xi^{jj',a}_{1 2}(za)}{a^{2h_{1 2}}} - \Phi^{jj'}_{1 2}(z) \Big| \leq  \frac{C}{a^{\min\{\varpi_0,(1 - \xi)(2h_{1 2}+ 2 \beta)\}}}, \quad n \rightarrow \infty,
\end{equation}
since for small enough $\xi$, conditions \eqref{e:hmax<delta0=<2(1+hmin)} and \eqref{e:psihat_is_slower_than_a_power_function} imply that
$$
(1 - \xi)(2h_{1 2}+ 2 \beta )> 2 h_{\min} + 2 \geq \varpi_0.
$$
In fact, by \eqref{e:ED(2j,k)D(2j',k')=based_on_harmonizable}, \eqref{e:Phi(j,j')q1q2(z)} and a change of variable $y = ax$,
$$
\frac{\Xi^{jj',a}_{12}(za)}{a^{2h_{12}}} - \Phi^{jj'}_{12}(z)
$$
$$
= \frac{1}{a^{2 h_{12}}} \int_{\bbR} e^{i z ax}|x|^{-(2h_{12}+1)} \rho_{12}\{g_{12}(x) - 1\} \widehat{\psi}(a2^jx)\overline{\widehat{\psi}(a2^{j'}x)}dx
$$
\begin{equation}\label{e:deviation_integral_from_that_for_fBm}
= \Big\{\int_{|y| \leq a^{1- \xi}}+\int_{|y| > a^{1- \xi}}\Big\}
e^{i z y}|y|^{-(2h_{12}+1)} \rho_{12}\Big\{g_{12}\Big(\frac{y}{a} \Big) - 1\Big\} \widehat{\psi}(2^j y)\overline{\widehat{\psi}(2^{j'} y)}dy
\end{equation}
for any $0 < \xi < 1$. For large enough $n$, by \eqref{e:g_is_Schwartz}, \eqref{e:hmax<delta0=<2(1+hmin)} and \eqref{e:N_psi}--\eqref{e:psihat_is_slower_than_a_power_function}, the absolute value of the first integral in the sum \eqref{e:deviation_integral_from_that_for_fBm} can be bounded by
\begin{equation}\label{e:deviation_integral_from_that_for_fBm_1st_term}
\frac{C}{a^{\varpi_0}}\int_{|y| \leq a^{1 - \xi}} |y|^{-(2 h_{12}+1) + \varpi_0} |\widehat{\psi}(2^j y)\overline{\widehat{\psi}(2^{j'} y)}|dy \leq \frac{C'}{a^{\varpi_0}}.
\end{equation}
On the other hand, by \eqref{e:W2} the absolute value of the second term in the sum \eqref{e:deviation_integral_from_that_for_fBm} can be bounded by
\begin{equation}\label{e:deviation_integral_from_that_for_fBm_2nd_term}
C \int_{|y|> a^{1 - \xi}} |y|^{-(2 h_{12}+1) - 2 \beta}dy \leq \frac{C'}{a^{(1 - \xi)(2 h_{12} +2 \beta)}}.
\end{equation}
Expressions \eqref{e:deviation_integral_from_that_for_fBm_1st_term} and \eqref{e:deviation_integral_from_that_for_fBm_2nd_term} yield \eqref{e:bound_for_deviation_integral_from_that_for_fBm}.\\

We turn to $(iii)$. For notational simplicity, consider only two octaves $j,j'$, whence $J = j_2 - j_1 + 1 = 2$. Then, from \eqref{e:Upsilon(n)},
$$
\Upsilon(n) = 2(n_{a,j_1}+n_{a,j_2}).
$$
Fix ${\mathbf v} \in \bbC^{\Upsilon(n)}$. For notational simplicity, divide the summation range $k = 1,\hdots, \Upsilon(n)$ into the subranges
$$
K_1 = \{1,\hdots,n_{a,j}\}, \quad K_2 = \{n_{a,j}+1,\hdots,n_{a,j}+n_{a,j'}\},
$$
$$
K_3 = \{n_{a,j}+n_{a,j'}+1,\hdots,2n_{a,j}+n_{a,j'}\}, \quad K_4 = \{2n_{a,j}+n_{a,j'}+1,\hdots,2(n_{a,j}+n_{a,j'})\}.
$$
Define the octave and index functions
$$
j(k) = j 1_{\{ K_1 \cup K_3 \}}(k) + j' 1_{\{K_2 \cup K_4 \}}(k), \quad q(k) = 1 \hspace{1mm}1_{\{K_1 \cup K_2 \}}(k) + 2 \hspace{1mm} 1_{\{K_3 \cup K_4 \}}(k),
$$
respectively, which reflects the order of appearance of different octave and index values in the vector \eqref{e:Wn}. By \eqref{e:ED(2j,k)D(2j',k')=based_on_harmonizable},
$$
{\mathbf v}^* \Gamma_{{\boldsymbol {\mathcal Y}}_n}{\mathbf v} = \sum^{\Upsilon(n)}_{k_1 =1}\sum^{\Upsilon(n)}_{k_2 =1}v_{k_1}\overline{v}_{k_2}\Big( \Gamma_{{\boldsymbol {\mathcal Y}}_n}\Big)_{k_1,k_2}
$$
\begin{equation}\label{e:v*Gammav_bound}
= \int_{\bbR} \sum^{\Upsilon(n)}_{k_1 =1} \sum^{\Upsilon(n)}_{k_2 =1} \frac{v_{k_1}\overline{v}_{k_2}e^{ia(2^{j(k_1)}k_1 -2^{j(k_2)}k_2)x}\widehat{\psi}(a2^{j(k_1)}x)\overline{\widehat{\psi}(a2^{j(k_2)}x)}}{\sqrt{\varrho^{(12)}(a2^{j(k_1)})\varrho^{(12)}(a2^{j(k_2)})}}f_{q(k_1)q(k_2)}(x) dx
\end{equation}
By expanding the double summation in \eqref{e:v*Gammav_bound} into each pair in the Cartesian product $\{K_1, K_2, K_3, K_4\}^2$, we end up with 16 double summation terms under the sign of the integral, i.e., 8 pairs of conjugates. To the cross terms, namely, those involving distinct summation ranges in the index $k$, we can apply the elementary inequality $ x\overline{y} + \overline{x}y\leq |x|^2 + |y|^2$. One such pair, associated with the summation regions $K_2 \times K_3$ and $K_3 \times K_2$, is
$$
\int_{\bbR}\Big(\sum_{k_1 \in K_3} \frac{v_{k_1}e^{ia2^{j}k_1 x}\widehat{\psi}(a2^{j}x)}{\sqrt{\varrho^{(12)}(a2^{j})}}\sum_{k_2 \in K_2}\frac{\overline{v}_{k_2}e^{-ia2^{j'}k_2 x}\overline{\widehat{\psi}(a2^{j'}x)}}{\sqrt{\varrho^{(12)}(a2^{j'})}}
$$
$$
+\sum_{k_1 \in K_2}\frac{v_{k_1}e^{ia2^{j'}k_1 x}\widehat{\psi}(a2^{j'}x)}{\sqrt{\varrho^{(12)}(a2^{j'})}}
\sum_{k_2 \in K_3} \frac{\overline{v}_{k_2}e^{-ia2^{j}k_2 x}\overline{\widehat{\psi}(a2^{j}x)}}{\sqrt{\varrho^{(12)}(a2^{j})}}\Big)f_{12}(x)dx
$$
$$
\leq \int_{\bbR}\Big\{\Big|\sum_{k \in K_3} \frac{v_{k}e^{ia2^{j}k x}\widehat{\psi}(a2^{j}x)}{\sqrt{\varrho^{(12)}(a2^{j})}}
 \Big|^2
 + \Big|\sum_{k \in K_2 }\frac{\overline{v}_{k}e^{-ia2^{j'}k x}\overline{\widehat{\psi}(a2^{j'}x)}}{\sqrt{\varrho^{(12)}(a2^{j'})}}\Big|\Big\}^2 f_{12}(x)dx,
$$
since $f_{12}(x) = f_{21}(x)$, and analogous bounds hold for the remaining terms. Therefore, \eqref{e:v*Gammav_bound} is bounded by
$$
\int_{\bbR} \Big|\sum_{k \in K_1} v_{k}e^{ia2^{j}kx}\Big|^2 \frac{|\widehat{\psi}(a2^{j}x)|^2}{\varrho^{(12)}(a2^{j})} (f_{11}(x) + 3 f_{12}(x))dx
$$
$$
+ \int_{\bbR} \Big|\sum_{k \in K_2} v_{k}e^{ia2^{j'}kx}\Big|^2 \frac{|\widehat{\psi}(a2^{j'}x)|^2}{\varrho^{(12)}(a2^{j'})} (f_{11}(x) + 3 f_{12}(x))dx
$$
$$
+ \int_{\bbR} \Big|\sum_{k \in K_3} v_{k}e^{ia2^{j}kx}\Big|^2 \frac{|\widehat{\psi}(a2^{j}x)|^2}{\varrho^{(12)}(a2^{j})} (f_{22}(x)+ 3 f_{12}(x))dx
$$
$$
+ \int_{\bbR}  \Big|\sum_{k \in K_4} v_{k}e^{ia2^{j'}kx}\Big|^2 \frac{|\widehat{\psi}(a2^{j'}x)|^2}{\varrho^{(12)}(a2^{j'})} (f_{22}(x)+ 3 f_{12}(x))dx.
$$
By the change of variable $y = a2^{j(k)}x$ in each integral, breaking them up (in $\bbR$) into subregions of length $2\pi$ and using the periodicity of Fourier sums, we obtain
$$
\int^{\pi}_{-\pi} \Big|\sum_{k \in K_1} v_{k}e^{iky}\Big|^2 \sum^{\infty}_{l= -\infty}\Big\{(a2^j)^{2h_1}|y+ 2 \pi l|^{-(2h_1 +1)}\rho_1 g_1\Big(\frac{y+ 2\pi l}{a2^j}\Big)
$$
$$
+ 3 (a2^j)^{2h_{12}}|y+ 2 \pi l|^{-(2h_{12} +1)}\rho_{12} g_{12}\Big(\frac{y+ 2 \pi l}{a2^j}\Big)\Big\}\frac{|\widehat{\psi}(y + 2 \pi l )|^2}{\varrho^{(12)}(a2^{j})} dy
$$
$$
+  \int^{\pi}_{- \pi}  \Big|\sum_{k \in K_2} v_{k}e^{iky}\Big|^2 \sum^{\infty}_{l = -\infty}\Big\{(a2^j)^{2h_1}|y+2 \pi l|^{-(2h_1 +1)}\rho_1 g_1\Big(\frac{y+ 2 \pi l}{a2^{j'}}\Big)
$$
$$
+3(a2^j)^{2h_{12}}|y+ 2 \pi l|^{-(2h_{12} +1)}\rho_{12} g_{12}\Big(\frac{y+ 2 \pi l}{a2^{j'}}\Big) \Big\}\frac{|\widehat{\psi}(y+ 2 \pi l)|^2}{\varrho^{(12)}(a2^{j'})} dy
$$
$$
+ \int^{\pi}_{- \pi}  \Big|\sum_{k \in K_3} v_{k}e^{iky}\Big|^2 \sum^{\infty}_{l = - \infty}\Big\{(a2^j)^{2h_2}|y+ 2 \pi l|^{-(2h_2 +1)}\rho_2 g_2\Big(\frac{y+ 2 \pi l}{a2^{j}}\Big) $$
$$
+ 3 (a2^j)^{2h_{12}}|y+ 2 \pi l|^{-(2h_{12} +1)}\rho_{12} g_{12}\Big(\frac{y+ 2 \pi l}{a2^{j}}\Big)\Big\}\frac{|\widehat{\psi}(y+ 2 \pi l)|^2}{\varrho^{(12)}(a2^{j})} dy
$$
$$
+ \int^{\pi}_{-\pi}  \Big|\sum_{k \in K_4} v_{k}e^{iky}\Big|^2 \sum^{\infty}_{l= - \infty}\Big\{(a2^j)^{2h_2}|y+ 2 \pi l|^{-(2h_2 +1)}\rho_2 g_2\Big(\frac{y+ 2 \pi l}{a2^{j'}}\Big)
$$
$$
+ 3(a2^j)^{2h_{12}}|y+ 2 \pi l|^{-(2h_{12} +1)}\rho_{12} g_{12}\Big(\frac{y+ 2 \pi l}{a2^{j'}}\Big)\Big\}\frac{|\widehat{\psi}(y+ 2 \pi l)|^2}{\varrho^{(12)}(a2^{j'})}  dy
$$
\begin{equation}\label{e:bound_cov_matrix_Wn}
\leq C a^{2(\max\{h_1,h_2\}-h_{12})} \int^{\pi}_{-\pi} \Big\{\Big|\sum_{k \in K_1} v_{k}e^{iky}\Big|^2 + \Big|\sum_{k \in K_2} v_{k}e^{iky}\Big|^2
\end{equation}
$$
+ \Big|\sum_{k \in K_3} v_{k}e^{iky}\Big|^2
+\Big|\sum_{k \in K_4} v_{k}e^{iky}\Big|^2 \Big\} dy
$$
$$
= C a^{2(\max\{h_1,h_2\}-h_{12})} {\mathbf v}^* {\mathbf v}
$$
for some $C > 0$. The inequality \eqref{e:bound_cov_matrix_Wn} is a consequence of \eqref{e:Xijj/standard} (from Lemma \ref{l:cov_and_eigenstructure_wavelet_transf_and_var}, $(i)$) and Proposition \ref{p:4th_moments_wavecoef}, $(iii)$, as applied to $\varrho^{(12)}(a2^{j(k)})$. Thus, the claim \eqref{e:max_eig=<a(n)^(2maxh1,h2-2h12)} holds. $\Box$\\
\end{proof}

\begin{remark}For $q_1 = 1$ and $q_2 = 2$, let ${\mathbf Y}_n$ and ${\boldsymbol {\mathcal Y}}_n$ be as in \eqref{e:Yn} and \eqref{e:Wn}, respectively. Then,
$$
{\mathbf Y}_n = {\mathcal P}_n \textnormal{diag}\Big(\underbrace{\sqrt{\varrho^{(12)}(a2^{j_1})},\hdots,\sqrt{\varrho^{(12)}(a2^{j_1})}}_{n_{a,j_1}},
\hdots, \underbrace{\sqrt{\varrho^{(12)}(a2^{j_2})},\hdots, \sqrt{\varrho^{(12)}(a2^{j_2})}}_{n_{a,j_2}},
$$
$$
\underbrace{\sqrt{\varrho^{(12)}(a2^{j_1})}, \hdots ,\sqrt{\varrho^{(12)}(a2^{j_1})}}_{n_{a,j_1}},\hdots,
\underbrace{\sqrt{\varrho^{(12)}(a2^{j_2})}, \hdots, \sqrt{\varrho^{(12)}(a2^{j_2})}}_{n_{a,j_2}} \Big){\boldsymbol {\mathcal Y}}_n
$$
$$
=: {\mathcal P}_n  N_n {\boldsymbol {\mathcal Y}}_n
$$
for some permutation matrix ${\mathcal P}_n $. Moreover, since $\Gamma_{{\textbf Y}_n} = {\mathcal P}_n  N_n \Gamma_{{\boldsymbol {\mathcal Y}}_n}N_n {\mathcal P}^*_n $ is a real symmetric matrix, by Lemma \ref{l:cov_and_eigenstructure_wavelet_transf_and_var}, $(ii)$ and $(v)$, we obtain the bound
\begin{equation}\label{e:max_ki=<decay}
\|\Gamma_{{\textbf Y}_n}\| = \|{\mathcal P}_n  N_n \Gamma_{{\boldsymbol {\mathcal Y}}_n}N_n {\mathcal P}^*_n \| \leq C a(n)^{2\max\{h_{1},h_{2}\}}
\end{equation}
for the maximum eigenvalue of $\Gamma_{{\textbf Y}_n}$, where $\|\cdot\|$ is the matrix Euclidean norm.
\end{remark}

For any $n \in \bbN$, consider the Gaussian vector ${\boldsymbol {\mathcal Y}}_{n}$ as in \eqref{e:Wn} but with only one octave $j$. It will be convenient to reexpress the sum $W^{(12)}_{n}(a2^j)$ as in \eqref{e:Sq1q2} (with $q_1 = 1$ and $q_2 =2$) based on a quadratic form. Define the permutation matrix
\begin{equation}\label{e:Pi}
\Pi_n = \left(\begin{array}{cc}
{\mathbf 0} & I_{n_{a,j}} \\
I_{n_{a,j}} & {\mathbf 0}
\end{array}\right) \in O(2 n_{a,j}).
\end{equation}
Consider the spectral decomposition \eqref{e:Gamma_n_spectral}. Then,
$$
W^{(12)}_{n}(a2^j) = \frac{1}{2n_{a,j}} {\boldsymbol {\mathcal Y}}^*_n \hspace{0.5mm} \Pi_n \hspace{0.5mm} {\boldsymbol {\mathcal Y}}_n \quad \textnormal{(for one octave $j$)}
$$
\begin{equation}\label{e:sum_q_1,q_2}
\stackrel{d}=
\frac{1}{n_{a,j}}\hspace{0.5mm} {\mathbf Z}^{*} \hspace{0.5mm}\frac{( O \Lambda^{1/2} O^*) \Pi_n (O \Lambda^{1/2}O^*)}{2}\hspace{0.5mm}{\mathbf Z}
\stackrel{d}= \frac{1}{ n_{a,j}}\hspace{0.5mm} {\mathbf Z}^* \hspace{0.5mm}\frac{\Lambda^{1/2} O^* \Pi_n O \Lambda^{1/2}}{2} \hspace{0.5mm}{\mathbf Z},
\end{equation}
where ${\mathbf Z} \sim N(0,I_{2 n_{a,j}})$. Let
\begin{equation}\label{e:Sn}
S_n = \frac{\Lambda^{1/2}O^* \Pi_n O \Lambda^{1/2}}{2},
\end{equation}
which is a real symmetric matrix. Write out its spectral decomposition
\begin{equation}\label{e:Sn=OS_LambdaS_O*S}
S_n = O_S \Lambda_S O^*_{S}, \quad O \in O(2n_{a,j}), \quad \Lambda_{S} = \textnormal{diag}(\lambda_{1,S}, \hdots,\lambda_{2n_{a,j},S}).
\end{equation}
Expression \eqref{e:sum_q_1,q_2} becomes
$$
W^{(12)}_{n}(a2^j) = \frac{1}{n_{a,j}}{\mathbf Z}^* O_S \Lambda_S O^*_{S} {\mathbf Z} \stackrel{d}= \frac{1}{n_{a,j}}{\mathbf Z}^* \Lambda_S {\mathbf Z}
$$
\begin{equation}\label{e:sum_W=sum_lambda*Z2-lambda*Z2}
= \sum_{i \in i_{+}(n)} \frac{\lambda_{i,S}}{n_{a,j}} Z^2_{i} + \sum_{i \in i_{-}(n)} \frac{\lambda_{i,S}}{n_{a,j}} Z^2_{i}
=: \sum_{i \in i_+(n)} \eta_{i,n}Z^{2}_{i}  - \sum_{i \in i_{-}(n)} \eta_{i,n} Z^{2}_{i},
\end{equation}
where $i_{+}(n)$ and $i_{-}(n)$ are the indices for which $\lambda_{i,S}$ is nonnegative or negative, respectively. In particular, note that
$$
\VAR{ \hspace{0.5mm}W^{(12)}_n(a(n)2^j)} = \|{\boldsymbol \eta}_n\|^2_2,
$$
where
\begin{equation}\label{e:eta-n_vec}
{\boldsymbol \eta}_{n} := (\eta_{1,n},\hdots,\eta_{2n_{a,j},n})^*
\end{equation}
and $\eta_{i,n}$, $i=1,\hdots,2 n_{a,j}$, are as in expression \eqref{e:sum_W=sum_lambda*Z2-lambda*Z2}.

The following lemma is an exponential inequality for $\chi^2$-like distributions, and corresponds to Lemma 1 in \cite{laurent:massart:2000}. It will be used in the ensuing Lemma \ref{l:P(WN<e)->0} to establish a bound on the rate of convergence to zero of the probability $P(W^{(12)}_n(a2^j) \leq r)$, where $-\infty < r < 1/2$ (under \eqref{e:mu_j>0}).

\begin{lemma}\label{l:lemma_laurent_massart}
\cite{laurent:massart:2000} Let $Z_1,\hdots,Z_n \stackrel{\textnormal{i.i.d.}}\sim N(0,1)$ and $\eta_1,\hdots,\eta_n \geq 0$, not all zero. Let $\|{\boldsymbol \eta}\|_2$ and $\|{\boldsymbol \eta}\|_{\infty}$ be the Euclidean square and sup norms of the vector ${\boldsymbol \eta} = (\eta_1,\hdots,\eta_n)^*$. Also, define the random variable
$X = \sum^{n}_{i=1} \eta_{i,n} (Z^2_i - 1)$.
Then, for every $x > 0$,
\begin{equation}\label{e:laurent_massart_bound1}
P(X \geq 2 \|{\boldsymbol \eta}\|_2 \hspace{1mm}\sqrt{x}+ 2 \|{\boldsymbol \eta}\|_{\infty} \hspace{1mm}x) \leq \exp(-x),
\end{equation}
\begin{equation}\label{e:laurent_massart_bound2}
P(X \leq - 2 \|{\boldsymbol \eta}\|_2 \hspace{1mm}\sqrt{x} ) \leq \exp(-x).
\end{equation}
\end{lemma}

\begin{lemma}\label{l:P(WN<e)->0}
Let $W^{(12)}_{n}(a2^j)$ be as in \eqref{e:sum_W=sum_lambda*Z2-lambda*Z2}, and fix $-\infty < r < 1/2$. Then, for any $0 < \xi < 1$,
\begin{equation}\label{e:P(Wbar=<bn0)=<exp(-nu(1+xi))}
P(W^{(12)}_{n}(a(n)2^j) \leq r) \leq \exp\Big\{ -\Big(\frac{n}{a^{2(h_{1}+h_{2}) - 4 h_{12}+1}} \Big)^{1 - \xi} \Big\}
\end{equation}
for large enough $n$.
\end{lemma}
\begin{proof}
Expression \eqref{e:sum_W=sum_lambda*Z2-lambda*Z2} implies that
$$
P(W^{(12)}_{n}(a2^j) \leq r)
= P\Big( \sum_{i \in i_{+}(n)} \eta_{i,n}(Z^{2}_{i}-1)  \leq -1 + r + \sum_{i \in i_{-}(n)} \eta_{i,n}(Z^{2}_{i}-1) \Big).
$$
For notational simplicity, let $X_{n_{a,j}} = \sum_{i \in i_{+}(n)} \eta_{i,n}(Z^{2}_{i}-1)$ and $Y_{n_{a,j}} = \sum_{i \in i_{-}(n)} \eta_{i,n}(Z^{2}_{i}-1) $, which are zero mean random variables. Now fix $r < \delta < 1$. By the independence of $X_{n_{a,j}}$ and $Y_{n_{a,j}}$,
\begin{equation}\label{e:P(X=<-1+Y)=two_integrals}
P( X_{n_{a,j}} \leq -1 + r  + Y_{n_{a,j}}) = \Big\{ \int^{1-\delta }_{-\infty} + \int^{\infty}_{1-\delta} \Big\} P(X_{n_{a,j}} \leq -1 + r  + y) f_{Y_{n_{a,j}}}(y)dy,
\end{equation}
where $f_{Y_{n_{a,j}}}$ is the density function of $Y_{n_{a,j}}$. The first integral in \eqref{e:P(X=<-1+Y)=two_integrals} is bounded from above by $P(X_{n_{a,j}} \leq -\delta + r) P( Y_{n_{a,j}} \leq 1-\delta)$. Let
$$
{\boldsymbol \eta}_{1,n} = (\eta_{i,n})_{i \in i_{+}(n)}, \quad {\boldsymbol \eta}_{2,n} = (\eta_{i,n})_{i \in i_{-}(n)}.
$$
Then, by \eqref{e:Cov(W12,W34)} below (under \eqref{e:2max[h1,h2,2h12]=2(h1+h2)}),
\begin{equation}\label{e:bounds_norm_eta-n}
\max\{\|{\boldsymbol \eta}_{1,n}\|_2 ,\|{\boldsymbol \eta}_{2,n}\|_2 \}\leq \|{\boldsymbol \eta}_{n}\|_2 \sim C \frac{a^{2(h_{1}+h_{2}) - 4 h_{12}+1}}{n}, \quad \|{\boldsymbol \eta}_{2,n}\|_\infty \leq \|{\boldsymbol \eta}_{n}\|_\infty \leq \|{\boldsymbol \eta}_{n}\|_2,
\end{equation}
for a constant $C \in \bbR$. Moreover, since $- \delta + r < 0$, then
$$
P(X_{n_{a,j}} \leq -\delta + r) = P \Big( \sum_{i \in i_{+}(n)} \eta_{i,n} (Z^2_i-1) \leq - \delta + r \Big)
$$
$$
= P \Big( \sum_{i \in i_{+}(n)} \eta_{i,n} (Z^2_i-1) \leq - 2 \|{\boldsymbol \eta}_{1,n}\|_2 \sqrt{ \frac{(\delta - r)^2}{4\|{\boldsymbol \eta}_{1,n}\|^2_2}}\Big) \leq \exp\Big\{- \frac{(\delta-1/2)^2}{4 \|{\boldsymbol \eta}_{1,n}\|^2_2}\Big\}
$$
\begin{equation}\label{e:P(X=<-delta)_bound_LandM}
\leq \exp\Big\{ - C \frac{n}{a^{2(h_{1}+h_{2}) - 4 h_{12}+1}}  \Big\}.
\end{equation}
In \eqref{e:P(X=<-delta)_bound_LandM}, $C$ does not depend on $r$ and the last two inequalities are a consequence of \eqref{e:laurent_massart_bound2} and \eqref{e:bounds_norm_eta-n} below, respectively, where
\begin{equation}\label{e:2max[h1,h2,2h12]=2(h1+h2)}
2\max\{h_{q_1}+h_{q_2},2h_{q_1 q_2}\} = 2(h_{q_1}+h_{q_2}), \quad q_1,q_2 = 1,\hdots,m,
\end{equation}
stems from \eqref{e:hq1q2}. On the other hand, for any $0 < \xi < 1$ and large enough $n$ the second integral in \eqref{e:P(X=<-1+Y)=two_integrals} is bounded from above by
\begin{equation}\label{e:int_fYn_dy}
\int^{\infty}_{1-\delta} f_{Y_{n_{a,j}}}(y) dy = P(Y_{n_{a,j}} > 1 - \delta) = P \Big( \sum_{i \in i_{-}(n)} \eta_{i,n} (Z^2_i - 1) > 1 - \delta \Big).
\end{equation}
Suppose, without loss of generality, that
$$
\|{\boldsymbol \eta_{2,n}}\|_{2} > 0, \quad n \in \bbN
$$
(otherwise, $P ( \sum_{i \in i_{-}(n)} \eta_{i,n} (Z^2_i - 1) > 1 - \delta )=0$ for $n$ such that $\|{\boldsymbol \eta_{2,n}}\|_{2} = 0$). Therefore, by \eqref{e:bounds_norm_eta-n},
$$
\|{\boldsymbol \eta}_{2,n}\|_2 \Big(\frac{n}{a^{2(h_{1}+h_{2}) - 4 h_{12}+1}} \Big)^{\frac{1 - \xi}{2}}+ \|{\boldsymbol \eta}_{2,n}\|_\infty \Big(\frac{n}{a^{2(h_{1}+h_{2}) - 4 h_{12}+1}} \Big)^{1 - \xi} \rightarrow 0,
$$
for any $0 < \xi < 1$, as $n \rightarrow \infty$. Consequently, for large enough $n$, \eqref{e:int_fYn_dy} is bounded by
$$
P \Big( \sum_{i \in i_{-}(n)}\eta_{i,n} (Z^2_i - 1) > 2\|{\boldsymbol \eta}_{2,n}\|_2 \Big(\frac{n}{a^{2(h_{1}+h_{2}) - 4 h_{12}+1}} \Big)^{\frac{1 - \xi}{2}}
+ 2\|{\boldsymbol \eta}_{2,n}\|_\infty \Big(\frac{n}{a^{2(h_{1}+h_{2}) - 4 h_{12}+1}} \Big)^{1 - \xi}\Big)
$$
\begin{equation}\label{e:P(Y>1-delta)_bound_LandM}
\leq \exp\Big\{-\Big(\frac{n}{a^{2(h_{1}+h_{2}) - 4 h_{12}+1}} \Big)^{1 - \xi}\Big\},
\end{equation}
where the last inequality follows from \eqref{e:laurent_massart_bound1}. 
From \eqref{e:P(X=<-delta)_bound_LandM} and \eqref{e:P(Y>1-delta)_bound_LandM}, since $C$ does not depend on $r$, we obtain \eqref{e:P(Wbar=<bn0)=<exp(-nu(1+xi))}. $\Box$\\
\end{proof}

\subsection{Asymptotic covariances}

We will establish Theorem \ref{t:Cov(log,log)_cross} at the end of this section, after proving Lemmas \ref{l:third_order_cross_term=O(1/n^2)}--\ref{l:Cov(log,1)_Cov(1,1)}. The latter establish the asymptotic behavior of the first four moments of the $\bbR$-valued random variables $W^{(q_1 q_2)}_{n}(a(n)2^j)$, $W^{(q_3 q_4)}_{n}(a(n)2^j)$. So, consider the function
\begin{equation}\label{e:cov_truncated}
\log\hspace{0.5mm}|S^{(q_1 q_2)}_{n}(a(n)2^j)|\hspace{0.5mm}1_{\{|W_n^{(q_1 q_2)}(a(n)2^j)| > r\}},\qquad 1 \leq q_1 \leq q_2 \leq m, \quad 0 < r < \frac{1}{2},
\end{equation}
where the truncation works as a regularization of the log function around the origin. In the event $|W^{(q_1 q_2)}_{n}(a(n)2^j)| \leq r$, we interpret that
$$
\log \hspace{0.5mm} |S^{(q_1 q_2)}_{n}(a(n)2^j)|\hspace{0.5mm}  1_{\{|W^{(q_1q_2)}_{n}(a(n)2^j)| > r\}} = 0 \quad \textnormal{a.s.}
$$
Throughout this section, we will make use of the Isserlis theorem (e.g., \cite{vignat:2012}). For a zero mean, Gaussian random vector ${\mathbf Z} = (Z_1,\hdots,Z_m)^* \in \bbR^m$, the theorem states that
\begin{equation}\label{e:Isserlis}
\EE{Z_1 \hdots Z_{2k}} = \sum \prod \EE{Z_i Z_j}, \quad \EE{Z_1 \hdots Z_{2k+1}} = 0, \quad k = 1,\hdots, \lfloor m/2 \rfloor.
\end{equation}
The notation $\sum \prod$ stands for adding over all possible $k$-fold products of pairs $\EE{Z_i Z_j}$, where the indices partition the set $1,\hdots,2k$.

The following lemma shows that the high order centered cross moments of $W^{(12)}_n(a(n)2^j)$, $W^{(34)}_n(a(n)2^j)$ are negligible with respect to their low order counterparts.
\begin{lemma}\label{l:third_order_cross_term=O(1/n^2)}
Let $\kappa_{1},\kappa_{2} \in \bbN \cup \{0\}$, $\kappa_{1} + \kappa_{2} \geq 3$, and fix $0 < r < 1/2$. Then, as $n \rightarrow \infty$,
$$
\EE{ \Big( \frac{W^{(12)}_{n}(a(n)2^j)-1}{a(n)^{\delta_{12}}} \Big)^{\kappa_1} \Big( \frac{W^{(34)}_{n}(a(n)2^{j'})-1}{{a(n)^{\delta_{34}}}} \Big)^{\kappa_2} 1_{\{\min\{|W^{(12)}_{n}(a(n)2^{j'})|,|W^{(34)}_{n}(a(n)2^{j'})|\} > r \}}}
$$
\begin{equation}\label{e:E(Wnj-1)(Wnj'-1)^2_with_indicators}
= O\Big[ \Big(\frac{a(n)}{n}\Big)^2 \Big].
\end{equation}
\end{lemma}
\begin{proof}
We first show that
\begin{equation}\label{e:E(Wnj-1)(Wnj'-1)^2_no_indicators}
\EE{ \Big( \frac{W^{(12)}_{n}(a2^j)-1}{a^{\delta_{12}}} \Big)^{\kappa_1} \Big( \frac{W^{(34)}_{n}(a2^{j'})-1}{{a^{\delta_{34}}}} \Big)^{\kappa_2} }
= O\Big[ \Big(\frac{a}{n}\Big)^2\Big],
\end{equation}
$\kappa_{1},\kappa_{2} \in \bbN$, $\kappa_{1}  + \kappa_{2} \geq 3$. We will only establish \eqref{e:E(Wnj-1)(Wnj'-1)^2_no_indicators} for $\kappa_1 = 1$ and $\kappa_2 = 2$, since the remaining cases can be tackled by a similar argument.

The left-hand side of \eqref{e:E(Wnj-1)(Wnj'-1)^2_no_indicators} can be rewritten as
$$
\frac{1}{a^{\delta_{12}+2 \delta_{34}}}\frac{1}{n_{a,j} n^2_{a,j'}}{\Bbb E}\Big[\sum^{n_{a,j}}_{k=1}\sum^{n_{a,j'}}_{k_1=1}\sum^{n_{a,j'}}_{k_2=1}\Big(\frac{d_{1}(a2^j,k)d_{2}(a2^j,k)}{\varrho^{(12)}(a2^j) } - 1\Big)
$$
\begin{equation}\label{e:E(Wnj-1)(Wnj'-1)^2_no_indicators_expanded}
\Big(\frac{d_{3}(a2^{j'},k'_1)d_{4}(a2^{j'},k'_1)}{\varrho^{(34)}(a2^{j'})} - 1\Big)\Big(\frac{d_{3}(a2^{j'},k'_2)d_{4}(a2^{j'},k'_2)}{\varrho^{(34)}(a2^{j'})} - 1\Big)\Big],
\end{equation}
Starting from the assumption \eqref{e:mu_j>0}, for notational simplicity denote the generic terms under the summation sign in \eqref{e:E(Wnj-1)(Wnj'-1)^2_no_indicators_expanded} as
$X_1 = d_{1}(a2^j,k)/\sqrt{\varrho^{(12)}(a2^j) }$, $X_2 = d_{2}(a2^j,k)/\sqrt{\varrho^{(12)}(a2^j) }$, $X_3 = d_{3}(a2^{j'},k'_{1})/\sqrt{\varrho^{(34)}(a2^{j'})}$, $X_4 = d_{4}(a2^{j'},k'_{1})/\sqrt{\varrho^{(34)}(a2^{j'})}$, $X_5 = d_{3}(a2^{j'},k'_2)/\sqrt{\varrho^{(34)}(a2^{j'})}$, $X_6 = d_{4}(a2^{j'},k'_2)/\sqrt{\varrho^{(34)}(a2^{j'})}$. Then,
$$
\EE{(X_1X_2 - 1)(X_3X_4 - 1)(X_5X_6 - 1)}
$$
\begin{equation}\label{e:E(X1X2-1)(X3X4-1)(X5X6-1)}
= \EE{X_1 X_2X_3X_4X_5X_6} - \{\EE{X_1X_2X_3X_4} + \EE{X_1X_2X_5X_6} + \EE{X_3X_4X_5X_6}\}+2.
\end{equation}
By applying the Isserlis relation \eqref{e:Isserlis} to the six-fold and four-fold products in \eqref{e:E(X1X2-1)(X3X4-1)(X5X6-1)}, the latter expression becomes
$$
\EE{X_1X_3}\EE{X_2X_5}\EE{X_4 X_6}+ \EE{X_1X_3}\EE{X_2X_6}\EE{X_4 X_5}
$$
$$
+ \EE{X_1X_4}\EE{X_2X_5}\EE{X_3 X_6}+ \EE{X_1X_4}\EE{X_2X_6}\EE{X_3 X_5}
$$
$$
+ \EE{X_1X_5}\EE{X_2X_3}\EE{X_4 X_6} + \EE{X_1X_5}\EE{X_2X_4}\EE{X_3 X_6}
$$
\begin{equation}\label{e:E(X1X2-1)(X3X4-1)(X5X6-1)_after_Isserlis}
+ \EE{X_1X_6}\EE{X_2X_3}\EE{X_4 X_5} + \EE{X_1X_6}\EE{X_2X_4}\EE{X_3 X_5}.
\end{equation}
The asymptotic behavior of the summation of each term in the seven-fold sum \eqref{e:E(X1X2-1)(X3X4-1)(X5X6-1)_after_Isserlis} can be established in the same way, so we only study the first one. Up to a constant, the latter is asymptotically equivalent to
$$
\frac{a^{(2h_{13}+2h_{23}+2h_4)- (2h_{12}+4 h_{34})}}{a^{\delta_{12}+2 \delta_{34}}\hspace{1mm}n_{a,j}}\frac{1}{n^2_{a,j'}}\sum^{n_{a,j}}_{k=1}\sum^{n_{a,j'}}_{k'_1=1}\sum^{n_{a,j'}}_{k'_2=1}\frac{\EE{d_{1}(a2^j,k)d_{3}(a2^{j'},k'_1)}}{a^{2h_{13}}}
$$
$$
\hspace{1mm}\frac{\EE{d_{2}(a2^j,k)d_{3}(a2^{j'},k'_2)}}{a^{2h_{23}}}\hspace{1mm}\frac{\EE{d_{4}(a2^{j'},k'_1)d_{4}(a2^{j'},k'_2)}}{a^{2h_{4}}}
$$
$$
= \frac{a^{2(h_{13}+h_{23}+ h_4)}}{a^{h_1 + h_2 + 2(h_3 + h_4)} \hspace{1mm}n_{a,j}}\frac{1}{n^2_{a,j'}}\sum^{n_{a,j'}}_{k'_1=1}\sum^{n_{a,j'}}_{k'_2=1} \hspace{1mm}\frac{\EE{d_{4}(a2^{j'},k'_1)d_{4}(a2^{j'},k'_2)}}{a^{2h_4}}
$$
\begin{equation}\label{e:triple_sum}
\cdot\Big\{\sum^{n_{a,j}}_{k=1} \frac{\EE{d_{1}(a2^j,k)d_{3}(a2^{j'},k'_1)}}{a^{2 h_{13}}}\hspace{1mm}\frac{\EE{d_{2}(a2^j,k)d_{3}(a2^{j'},k'_2)}}{a^{2h_{23}}}\Big\}.
\end{equation}
However, the summation in $k$ in expression \eqref{e:triple_sum} is bounded by
$$
\Big|\sum^{n_{a,j}}_{k=1} \frac{\EE{d_{1}(a2^j,k)d_{3}(a2^{j'},k'_1)}}{a^{2 h_{13}}}\hspace{1mm}\frac{\EE{d_{2}(a2^j,k)d_{3}(a2^{j'},k'_2)}}{a^{2h_{23}}}\Big|
$$
$$
\leq \sum^{n_{a,j}}_{k=1} \Big(\Big|\frac{\EE{d_{1}(a2^j,k)d_{3}(a2^{j'},k'_1)} - \Phi^{jj'}_{13}(a(2^j k-2^{j'}k'_1))}{a^{2 h_{13}}}\Big| + \Big|\frac{\Phi^{jj'}_{13}(a(2^j k-2^{j'}k'_1))}{a^{2 h_{13}}}\Big| \Big)
$$
$$
\Big( \Big|\frac{\EE{d_{2}(a2^j,k)d_{3}(a2^{j'},k'_2)} - \Phi^{jj'}_{23}(a(2^j k-2^{j'}k'_2))}{a^{2h_{23}}}\Big|  + \Big|\frac{\Phi^{jj'}_{23}(a(2^j k-2^{j'}k'_2))}{a^{2h_{23}}}\Big|\Big)
$$
$$
\leq \sum^{n_{a,j}}_{k=1} \Big\{\Big|\frac{\EE{d_{1}(a2^j,k)d_{3}(a2^{j'},k'_1)} - \Phi^{jj'}_{13}(a(2^j k-2^{j'}k'_1))}{a^{2 h_{13}}}\Big|
$$
$$
\cdot \Big|\frac{\EE{d_{2}(a2^j,k)d_{3}(a2^{j'},k'_2)} - \Phi^{jj'}_{23}(a(2^j k-2^{j'}k'_2))}{a^{2h_{23}}}\Big|
$$
$$
+ \Big|\frac{\Phi^{jj'}_{13}(a(2^j k-2^{j'}k'_1))}{a^{2 h_{13}}}\Big| \Big|\frac{\EE{d_{2}(a2^j,k)d_{3}(a2^{j'},k'_2)} - \Phi^{jj'}_{23}(a(2^j k-2^{j'}k'_2))}{a^{2h_{23}}}\Big|
$$
$$
+ \Big|\frac{\EE{d_{1}(a2^j,k)d_{3}(a2^{j'},k'_1)} - \Phi^{jj'}_{13}(a(2^j k-2^{j'}k'_1))}{a^{2 h_{13}}}\Big| \Big|\frac{\Phi^{jj'}_{23}(a(2^j k-2^{j'}k'_2))}{a^{2h_{23}}}\Big|
$$
\begin{equation}\label{e:bound_doublesum_product_wavelet_moments}
+ \Big|\frac{\Phi^{jj'}_{13}(a(2^j k-2^{j'}k'_1))}{a^{2 h_{13}}}\Big| \Big|\frac{\Phi^{jj'}_{23}(a(2^j k-2^{j'}k'_2))}{a^{2h_{23}}}\Big|\Big\} \leq C,
\end{equation}
where $C$ does not depend on $k$. To justify the last inequality, we only look at the second term in \eqref{e:bound_doublesum_product_wavelet_moments}, since the remaining terms can be analyzed in a similar way. It suffices to proceed as in \cite{abry:didier:2016:supplementary}, in particular around expression (B.31) in the latter reference. Indeed, starting from \eqref{e:W2}, suppose without loss of generality that
$$
\textnormal{supp}(\psi) = [0,1].
$$
For $0 < \varepsilon < 1/2$, decompose
$$
\sum^{n_{a,j}}_{k=1} \Big|\Phi^{jj'}_{13}(2^j k-2^{j'}k'_1)\Big| \Big|\frac{\EE{d_{2}(a2^j,k)d_{3}(a2^{j'},k'_2)} - \Phi^{jj'}_{23}(a(2^j k-2^{j'}k'_2))}{a^{2h_{23}}}\Big|
$$
$$
= \sum^{n_{a,j}}_{k=1} \Big(1_{\Big\{\frac{\max\{2^j,2^{j'}\}}{|2^jk-2^{j'}k'_1|} > \varepsilon\Big\}} + 1_{\Big\{\frac{\max\{2^j,2^{j'}\}}{|2^jk-2^{j'}k'_1|} \leq \varepsilon \Big\}}\Big)
$$
\begin{equation}\label{e:sum_small_2jk-2j'k'_vs_large_2jk-2j'k'}
\Big|\Phi^{jj'}_{13}(2^j k-2^{j'}k'_1)\Big| \Big|\frac{\EE{d_{2}(a2^j,k)d_{3}(a2^{j'},k'_2)} - \Phi^{jj'}_{23}(a(2^j k-2^{j'}k'_2))}{a^{2h_{23}}}\Big|.
\end{equation}
The first sum term in \eqref{e:sum_small_2jk-2j'k'_vs_large_2jk-2j'k'} only contains a finite number of terms, where such number does not depend on $k$ or $k'_1$. Moreover,
$$
\max \Big\{|\Phi^{jj'}_{13}(2^j k-2^{j'}k'_1)|,\Big|\frac{\EE{d_{2}(a2^j,k)d_{3}(a2^{j'},k'_2)} - \Phi^{jj'}_{23}(a(2^j k-2^{j'}k'_2))}{a^{2h_{23}}}\Big|\Big\} \leq C,
$$
where $C$ does not depend $k$, $k'_1$ or $k''_2$. The latter statement follows from the fact that $\Phi^{jj'}_{13} $ is, up to a change of sign, the wavelet variance of a fBm and from \eqref{e:|Xi-Phi|=<1/r2}.
Moreover, by Proposition II.2 in \cite{bardet:2002} and again by \eqref{e:|Xi-Phi|=<1/r2}, the second sum term in \eqref{e:sum_small_2jk-2j'k'_vs_large_2jk-2j'k'} is bounded by $C \sum_{z \in \bbZ \backslash\{0\}} z^{-4}$. This shows that \eqref{e:sum_small_2jk-2j'k'_vs_large_2jk-2j'k'} is bounded by a constant not depending on $k$, $k'_1$ or $k'_2$. Hence, \eqref{e:bound_doublesum_product_wavelet_moments} holds.

Consequently, up to a constant the absolute value of \eqref{e:triple_sum} is bounded from above by
$$
C' \frac{a^{2(h_{13}+h_{23}+ h_4)+2}}{a^{h_1 + h_2 + 2(h_3 + h_4)}\hspace{1mm}n^2}\hspace{1mm}\frac{1}{n_{a,j'}}\sum^{n_{a,j'}}_{k'_1=1} \sum^{n_{a,j'}}_{k'_2=1} \Big|\frac{\EE{d_{4}(a2^{j'},k'_1) d_{4}(a2^{j'},k'_2)}}{a^{2h_4}}\Big|
\leq C'' \Big(\frac{a}{n}\Big)^2,
$$
where we used the fact that $\frac{1}{a^{\delta_{13}+\delta_{23}}} \leq 1$. This establishes \eqref{e:E(Wnj-1)(Wnj'-1)^2_no_indicators}.

So, rewrite
$$
\EE{\Big( \frac{W^{(12)}_{n}(a2^j)-1}{a^{\delta_{12}}} \Big)\Big( \frac{W^{(34)}_{n}(a2^{j'})-1}{a^{\delta_{34}}} \Big)^2}
$$
$$
- \EE{\Big( \frac{W^{(12)}_{n}(a2^j)-1}{a^{\delta_{12}}} \Big)1_{\{W^{(12)}_{n}(a2^j) > r\}}\Big( \frac{W^{(34)}_{n}(a2^{j'})-1}{a^{\delta_{34}}} \Big)^2 1_{\{W^{(34)}_{n}(a2^{j'}) > r\}}}
$$
$$
= {\Bbb E}\Big[\Big( \frac{W^{(12)}_{n}(a2^j)-1}{a^{\delta_{12}}} \Big)\Big(\frac{W^{(34)}_{n}(a2^{j'})-1}{a^{\delta_{34}}}  \Big)^2 \Big\{1_{\{W^{(12)}_{n}(a2^j) \leq r\}}1_{\{W^{(34)}_{n}(a2^{j'}) > r\}}
$$
\begin{equation}\label{e:difference_E(Wnj-1)(Wnj'-1)2_with_without_truncation}
+ 1_{\{W^{(12)}_{n}(a2^j) > r\}}1_{\{W^{(34)}_{n}(a2^{j'}) \leq r\}} + 1_{\{W^{(12)}_{n}(a2^j) \leq r\}} 1_{\{W^{(34)}_{n}(a2^{j'}) \leq r\}}\Big\}\Big].
\end{equation}
The asymptotic behavior of every term on the right-hand side of \eqref{e:difference_E(Wnj-1)(Wnj'-1)2_with_without_truncation} can be established in the same way, so we only look at the first one. For any $0 < \xi < 1$, the Cauchy-Schwarz inequality, Lemma \ref{l:P(WN<e)->0} and expression \eqref{e:E(Wnj-1)(Wnj'-1)^2_no_indicators} yield
$$
\Big|\EE{\Big(\frac{W^{(12)}_{n}(a2^j)-1}{a^{\delta_{12}}} \Big)1_{\{W^{(12)}_{n}(a2^j) \leq r\}}\Big(\frac{W^{(34)}_{n}(a2^{j'})-1}{a^{\delta_{34}}}  \Big)^2 1_{\{W^{(34)}_{n}(a2^{j'}) > r\}}} \Big|
$$
$$
\leq \sqrt{\EE{\Big(\frac{W^{(12)}_{n}(a2^j)-1}{a^{\delta_{12}}} \Big)^2 \Big(\frac{W^{(34)}_{n}(a2^{j'})-1}{a^{\delta_{34}}}  \Big)^4}} \sqrt{P(W^{(12)}_{n}(a2^j) \leq r)}
$$
$$
\leq C \Big(\frac{a}{n}\Big) \exp\Big\{-\frac{1}{2}\Big( \frac{n}{a^{2(h_{3}+h_{4}) - 4 h_{34}+1}} \Big)^{1 -\xi}\Big\},
$$
where $C$ does not depend on $r$. Moreover, under \eqref{e:assumption_a(n)_n},
$$
\exp\Big\{-\frac{1}{2}\Big( \frac{n}{a^{2(h_{3}+h_{4}) - 4 h_{34}+1}} \Big)^{1 -\xi}\Big\} = o\Big(\frac{a}{n}\Big).
$$
Therefore, \eqref{e:E(Wnj-1)(Wnj'-1)^2_with_indicators} holds. $\Box$\\
\end{proof}

The following lemma expresses, up to a residual term, the cross-covariance (first cross-moment) between sample wavelet variances in terms of the functions \eqref{e:Phi(j,j')q1q2(z)}.
\begin{lemma}\label{l:E(Wnj-1)(Wnj'-1)}
Let $\Phi^{jj'}_{\cdot \cdot}(z)$ and $n_*$ be as in \eqref{e:Phi(j,j')q1q2(z)} and \eqref{e:n*}, respectively. For $0 < r < 1/2$,
$$
\EE{ (W^{(12)}_{n}(a(n)2^j) - 1) 1_{\{|W^{(12)}_{n}(a(n)2^j)| > r \}} (W^{(34)}_{n}(a(n)2^{j'}) - 1) 1_{\{|W^{(34)}_{n}(a(n)2^{j'})| > r \}} }
$$
$$
= \hspace{0.5mm} \frac{2^{-(j+j')}}{\Phi^{jj}_{12}(0) \Phi^{j'j'}_{34}(0)\{1 + O(a(n)^{-\varpi_0})\}^2 }
$$
$$
\Big\{\frac{a(n)^{2 (h_{13}+h_{24})- 2 (h_{12}+h_{34})}}{n_*}\Big[ \frac{1}{n_*}\sum^{2^{j'}n_*}_{k = 1}\sum^{2^{j}n_*}_{k' = 1}\Phi^{jj'}_{1 3}(2^{j}k-2^{j'}k')\Phi^{jj'}_{24}(2^{j}k-2^{j'}k')\Big]
$$
$$
+ \frac{a(n)^{2 (h_{14}+h_{23})-2 (h_{12}+h_{34})}}{n_*} \Big[\frac{1}{n_*}\sum^{2^{j'}n_*}_{k = 1}\sum^{2^{j}n_*}_{k' = 1}\Phi^{jj'}_{1 4}(2^{j}k-2^{j'}k')\Phi^{jj'}_{23}(2^{j}k-2^{j'}k')\Big]
$$\begin{equation}\label{e:E(Wnj-1)(Wnj'-1)^2}
+ o\Big( \frac{a^{2\max\{h_{13}+h_{24},h_{12}+h_{23}\} - 2 (h_{12}+h_{34})}}{n_*}\Big)\Big\},
\end{equation}
as $n \rightarrow \infty$.
\end{lemma}
\begin{proof}
As in the proof of Lemma \ref{l:third_order_cross_term=O(1/n^2)}, we first drop the indicator functions on the left-hand side of \eqref{e:E(Wnj-1)(Wnj'-1)^2} and investigate the limit. We will show that
$$
\Cov\Big[W^{(12)}_{n}(a2^j), W^{(34)}_{n}(a2^{j'}) \Big]
= \frac{2^{-(j+j')}}{[\Phi^{jj}_{12}(0)\Phi^{j'j'}_{34}(0) (1+ O(a^{-\varpi_0}))^2 ]}
$$
$$
\cdot \Big\{\frac{a^{2(h_{13}+h_{24}) - 2(h_{12}+h_{34})}}{n_*}  \frac{1}{n_*} \sum^{2^{j'}n_*}_{k = 1}\sum^{2^{j}n_*}_{k' = 1}\Phi^{jj'}_{13}(2^{j}k-2^{j'}k')\Phi^{jj'}_{24}(2^{j}k-2^{j'}k')
$$
$$
+ \frac{a^{2(h_{14}+h_{23}) - 2(h_{12}+h_{34})}}{n_*} \frac{1}{n_*}\sum^{2^{j'}n_*}_{k = 1}\sum^{2^{j}n_*}_{k' = 1}\Phi^{jj'}_{14}(2^{j}k-2^{j'}k')\Phi^{jj'}_{23}(2^{j}k- 2^{j'}k')
$$
\begin{equation}\label{e:Cov(W12,W34)}
+ o\Big(\frac{a^{2\max\{h_{13}+h_{24},h_{14}+h_{23}\} - 2(h_{12}+h_{34})}}{n_*}\Big)\Big\}.
\end{equation}
In fact, the left-hand side of \eqref{e:first_order_cross_expression_up_to_premultiplication_by_growth_rate} can be written as
$$
\EE{(W^{(12)}_{n}(a2^j) - 1) (W^{(34)}_{n}(a2^{j'}) - 1) }
$$
\begin{equation}\label{e:E(product_cross_sums)_reexpressed}
= \frac{1}{n_{a,j} n_{a,j'}} \sum^{n_{a,j}}_{k=1} \sum^{n_{a,j'}}_{k'=1} \Big\{ \EE{ \frac{d_{1}(a2^j,k)d_{2}(a2^j,k)}{\varrho^{(12)}(a2^j)}\frac{d_{3}(a2^{j'},k')d_{4}(a2^{j'},k')}{\varrho^{(34)}(a2^{j'})}}
- 1\Big\}.
\end{equation}
By the Isserlis relation \eqref{e:Isserlis}, the first term in the argument of the sum \eqref{e:E(product_cross_sums)_reexpressed} can be reexpressed as
$$
\frac{\EE{ d_{1}(a2^j,k)d_{2}(a2^j,k) d_{3}(a2^{j'},k')d_{4}(a2^{j'},k')}}{\varrho^{(12)}(a2^j)\varrho^{(34)}(a2^{j'})}
$$
$$
= \Big\{1 + \frac{\EE{d_{1}(a2^j,k)d_{3}(a2^{j'},k') }\hspace{1mm}\EE{d_{2}(a2^j,k)d_{4}(a2^{j'},k')}}{\varrho^{(12)}(a2^j)\varrho^{(34)}(a2^{j'})} $$
\begin{equation}\label{e:first_term_reexpressed}
+ \frac{\EE{d_{1}(a2^j,k)d_{4}(a2^{j'},k')} \hspace{1mm}\EE{d_{2}(a2^j,k)d_{3}(a2^{j'},k')}}{\varrho^{(12)}(a2^j)\varrho^{(34)}(a2^{j'})} \Big\}.
\end{equation}
By Lemma \ref{l:cov_and_eigenstructure_wavelet_transf_and_var}, $(ii)$, and \eqref{e:first_term_reexpressed}, we can rewrite \eqref{e:E(product_cross_sums)_reexpressed} as
$$
\frac{1}{n_{a,j} n_{a,j'}}\sum^{n_{a,j}}_{k=1} \sum^{n_{a,j'}}_{k'=1} \Big\{\frac{\EE{ d_{1}(a2^j,k)d_{3}(a2^{j'},k') }\hspace{1mm}\EE{d_{2}(a2^j,k)d_{4}(a2^{j'},k')}}{\varrho^{(12)}(a2^j) \varrho^{(34)}(a2^{j'}) }
$$
$$
+ \frac{\EE{d_{1}(a2^j,k)d_{4}(a2^{j'},k')} \hspace{1mm}\EE{d_{2}(a2^j,k)d_{3}(a2^{j'},k')}}{\varrho^{(12)}(a2^j) \varrho^{(34)}(a2^{j'})} \Big\}
$$
$$
= \frac{2^{-(j+j')}}{\Phi^{jj}_{12}(0)\Phi^{j'j'}_{34}(0) (1+ O(a^{-\varpi_0}))^2}
$$
$$
\Big\{\frac{a^{2(h_{13}+h_{24}) - 2(h_{12}+h_{34})}}{n_*} \Big[\frac{1}{n_*}\sum^{2^{j'}n_*}_{k=1}\sum^{2^{j}n_*}_{k'=1}\frac{\EE{d_{1}(a2^j,k)d_{3}(a2^{j'},k')}}{a^{2h_{13}}}\frac{\EE{d_{2}(a2^j,k)d_{4}(a2^{j'},k')}}{a^{2h_{24}}} \Big]
$$
\begin{equation}\label{e:first_order_cross_expression}
+\frac{a^{2(h_{14}+h_{23}) - 2(h_{12}+h_{34})}}{n_*} \Big[\frac{1}{n_*}\sum^{2^{j'}n_*}_{k=1}\sum^{2^{j}n_*}_{k'=1}\frac{\EE{d_{1}(a2^j,k)d_{4}(a2^{j'},k')}}{a^{2h_{14}}}\frac{\EE{d_{2}(a2^{j},k)d_{3}(a2^{j'},k')}}{a^{2h_{23}}}\Big]\Big\}.
\end{equation}
By adding and subtracting the counterparts $\Phi^{jj'}_{\cdot \cdot}(2^jk-2^{j'}k')$ for each term, up to the factor $2^{-(j+j')}/\{\Phi^{jj}_{12}(0)\Phi^{j'j'}_{34}(0) (1+ O(a^{-\varpi_0}))^2 \}$ the expression \eqref{e:first_order_cross_expression} can be written as
$$
\frac{a^{2(h_{13}+h_{24}) - 2(h_{12}+h_{34})}}{n_*} \Big[\frac{1}{n_*}\sum^{2^{j'}n_*}_{k=1}\sum^{2^{j}n_*}_{k'=1}\Big(\frac{\EE{d_{1}(a2^j,k)d_{3}(a2^{j'},k')} }{a^{2h_{13}}}- \Phi^{jj'}_{13}(2^jk-2^{j'}k')\Big)
$$
$$
\Big(\frac{\EE{d_{2}(a2^j,k)d_{4}(a2^{j'},k')} }{a^{2h_{24}}}- \Phi^{jj'}_{24}(2^jk-2^{j'}k')\Big)
$$
$$
+ \Phi^{jj'}_{13}(2^jk-2^{j'}k') \Big(\frac{\EE{d_{2}(a2^j,k)d_{4}(a2^{j'},k')}}{a^{2h_{24}}} - \Phi^{jj'}_{24}(2^jk-2^{j'}k')\Big)
$$
$$
+ \Phi^{jj'}_{24}(2^jk-2^{j'}k') \Big(\frac{\EE{d_{1}(a2^j,k)d_{3}(a2^{j'},k')}}{a^{2h_{13}}} - \Phi^{jj'}_{13}(2^jk-2^{j'}k')\Big)
$$
$$
+ \Phi^{jj'}_{13}(2^jk-2^{j'}k') \Phi^{jj'}_{24}(2^jk-2^{j'}k')\Big\}\Big]
$$
$$
+ \frac{a^{2(h_{14}+h_{23}) - 2(h_{12}+h_{34})}}{n_*} \Big[\frac{1}{n_*}\sum^{2^{j'}n_*}_{k=1}\sum^{2^{j}n_*}_{k'=1}\Big\{\Big( \frac{\EE{d_{1}(a2^j,k)d_{4}(a2^{j'},k')}}{a^{2h_{14}}} - \Phi^{jj'}_{14}(2^jk-2^{j'}k')\Big)
$$
$$
\cdot \Big( \frac{\EE{d_{2}(a2^{j},k)d_{3}(a2^{j'},k')}}{a^{2h_{23}}} - \Phi^{j'j}_{23}(2^{j}k-2^{j'}k')\Big)
$$
$$
+\Phi^{jj'}_{14}(2^{j}k-2^{j'}k') \Big( \frac{\EE{d_{1}(a2^j,k)d_{4}(a2^{j'},k')}}{a^{2h_{14}}} - \Phi^{jj'}_{23}(2^jk-2^{j'}k')\Big)
$$
$$
+\Phi^{jj'}_{23}(2^{j}k-2^{j'}k') \Big( \frac{\EE{d_{2}(a2^{j},k)d_{3}(a2^{j'},k') }}{a^{2h_{23}}} - \Phi^{jj'}_{14}(2^{j}k-2^{j'}k')\Big)
$$
$$
+ \Phi^{jj'}_{14}(2^{j}k-2^{j'}k') \Phi^{jj'}_{23}(2^{j}k- 2^{j'}k')\Big\} \Big]
$$
$$
= \frac{a^{2(h_{13}+h_{24}) - 2(h_{12}+h_{34})}}{n_*}  \frac{1}{n_*} \sum^{2^{j'}n_*}_{k = 1}\sum^{2^{j}n_*}_{k' = 1}\Phi^{jj'}_{13}(2^{j}k-2^{j'}k')\Phi^{jj'}_{24}(2^{j}k-2^{j'}k')
$$
$$
+ \frac{a^{2(h_{14}+h_{23}) - 2(h_{12}+h_{34})}}{n_*} \frac{1}{n_*}\sum^{2^{j'}n_*}_{k = 1}\sum^{2^{j}n_*}_{k' = 1}\Phi^{jj'}_{14}(2^{j}k-2^{j'}k')\Phi^{jj'}_{23}(2^{j}k- 2^{j'}k')
$$
\begin{equation}\label{e:first_order_cross_expression_up_to_premultiplication_by_growth_rate}
+ o\Big(\frac{a^{2\max\{h_{13}+h_{24},h_{14}+h_{23}\} - 2(h_{12}+h_{34})}}{n_*}\Big).
\end{equation}
It remains to justify the order of the error term in \eqref{e:first_order_cross_expression_up_to_premultiplication_by_growth_rate}. So, by adapting the proof of \eqref{e:summation_deviation_HfBm_idealHfBm}, we obtain
$$
\frac{a^{2(h_{13}+h_{24}) - 2(h_{12}+h_{34})}}{n_*} \frac{1}{n_*}\Big|\sum^{2^j n_*}_{k=1}\sum^{2^{j'} n_*}_{k'=1} \Phi^{jj'}_{13}(2^jk-2^{j'}k') \Big(\frac{\EE{d_{2}(a2^j,k)d_{4}(a2^{j'},k')}}{a^{2h_{24}}} - \Phi^{jj'}_{24}(2^jk-2^{j'}k')\Big)\Big|
$$
$$
= \frac{a^{2(h_{13}+h_{24}) - 2(h_{12}+h_{34})}}{n_*} \hspace{1mm}o(1) = o\Big(\frac{a^{2(h_{13}+h_{24}) - 2(h_{12}+h_{34})}}{n_*} \Big).
$$
The same bound holds for
$$
\frac{a^{2(h_{13}+h_{24}) - 2(h_{12}+h_{34})}}{n_*} \frac{1}{n_*}\Big|\sum^{2^j n_*}_{k=1}\sum^{2^{j'} n_*}_{k'=1} \Phi^{jj'}_{24}(2^jk-2^{j'}k') \Big(\frac{\EE{d_{1}(a2^j,k)d_{3}(a2^{j'},k')}}{a^{2h_{13}}} - \Phi^{jj'}_{13}(2^jk-2^{j'}k')\Big)\Big|.
$$
and
$$
\frac{a^{2(h_{13}+h_{24}) - 2(h_{12}+h_{34})}}{n_*} \frac{1}{n_*}\Big|\sum^{2^{j'}n_*}_{k=1}\sum^{2^{j}n_*}_{k'=1}\Big(\frac{\EE{d_{1}(a2^j,k)d_{3}(a2^{j'},k')}}{a^{2h_{13}}}- \Phi^{jj'}_{13}(2^jk-2^{j'}k')\Big)
$$
$$
\Big(\frac{\EE{d_{2}(a2^j,k)d_{4}(a2^{j'},k')}}{a^{2h_{24}}}- \Phi^{jj'}_{24}(2^jk-2^{j'}k')\Big)\Big|
$$
By extending this analysis to the remaining terms of \eqref{e:first_order_cross_expression_up_to_premultiplication_by_growth_rate}, we obtain analogous bounds and the error term $o\Big(\frac{a^{2\max\{h_{13}+h_{24},h_{14}+h_{23}\} - 2(h_{12}+h_{34})}}{n_*}\Big)$, as claimed.

In view of \eqref{e:first_order_cross_expression} and \eqref{e:first_order_cross_expression_up_to_premultiplication_by_growth_rate}, it suffices to show that the indicator functions on the left-hand side of \eqref{e:E(Wnj-1)(Wnj'-1)^2} do not affect the approximation order. In fact,
$$
\EE{ (W^{(12)}_{n}(a2^j) - 1) (W^{(34)}_{n}(a2^{j'}) - 1) }
 $$
 $$
 - \EE{ (W^{(12)}_{n}(a2^j) - 1) 1_{\{W^{(12)}_{n}(a2^j) > r \}} (W^{(34)}_{n}(a2^{j'}) - 1) 1_{\{W^{(34)}_{n}(a2^{j'}) > r \}} }
$$
$$
= {\Bbb E }\Big[(W^{(12)}_{n}(a2^j) - 1) (W^{(34)}_{n}(a2^{j'}) - 1)\cdot \Big(1_{\{W^{(12)}_{n}(a2^j) > r \}} 1_{\{W^{(34)}_{n}(a2^{j'}) \leq r \}}
$$
\begin{equation}\label{e:E(Wnj-1)(Wnj'-1)}
+ 1_{\{W^{(12)}_{n}(a2^j) \leq r\}} 1_{\{W^{(34)}_{n}(a2^{j'}) > r \}} + 1_{\{W^{(12)}_{n}(a2^j) \leq r \}} 1_{\{W^{(34)}_{n}(a2^{j'}) \leq r \}}\Big) \Big].
\end{equation}
Define
\begin{equation}\label{e:h*_hsub*}
h^* = \max_{q_1,q_2=1,2,3,4}h_{q_1q_2}, \quad h_* = \min_{q_1,q_2=1,2,3,4}h_{q_1q_2}.
\end{equation}
For $0 < \xi' < 1$, by the Cauchy-Schwarz inequality, expression \eqref{e:E(Wnj-1)(Wnj'-1)^2_no_indicators} (from Lemma \ref{l:third_order_cross_term=O(1/n^2)}) and Lemma \ref{l:P(WN<e)->0}, the first term on the right-hand side of \eqref{e:E(Wnj-1)(Wnj'-1)} is bounded by
$$
\Big|\EE{ (W^{(12)}_{n}(a2^j) - 1) 1_{\{W^{(12)}_{n}(a2^j) > r \}} (W^{(34)}_{n}(a2^{j'}) - 1) 1_{\{W^{(34)}_{n}(a2^{j'}) \leq r \}} }\Big|
$$
$$
\leq a^{\delta_{12}+\delta_{34}}\sqrt{\EE{ \Big( \frac{W^{(12)}_{n}(a2^j) - 1}{a^{\delta_{12}}} \Big)^2 \Big( \frac{W^{(34)}_{n}(a2^{j'}) - 1}{a^{\delta_{34}}} \Big)^2}} \sqrt{P(W^{(34)}_{n}(a2^j) \leq r)}
$$
$$
\leq C \Big(\frac{a^{4 h_{\max} - 4 h_{\min}+1}}{n}\Big) \exp\Big\{-\frac{1}{2}\Big(\frac{n}{a^{2(h_{3}+h_{4}) - 4 h_{34}+1}} \Big)^{1- \xi'} \Big\}
$$
$$
= o\Big( \frac{a^{2\max\{h_{13}+h_{24},h_{12}+h_{23}\} - 2 (h_{12}+h_{34})}}{n_*}\Big),
$$
where the constant $C$ does not depend on $r$ and the last equality is a consequence of \eqref{e:assumption_a(n)_n}. Similar bounds hold for the remaining terms on the right-hand side of \eqref{e:E(Wnj-1)(Wnj'-1)}. Therefore, the expression \eqref{e:E(Wnj-1)(Wnj'-1)^2} follows. $\Box$\\
\end{proof}

The following lemma describes the decay rate of the first individual truncated moment of the wavelet variance.
\begin{lemma}\label{l:order_first_moment_1(Wn>0)}
For any $0 < \xi < 1$ and $0 < r < 1/2$,
$$
\max\Big\{\Big| \EE{ \{W^{(12)}_{n}(a(n)2^j) - 1 \} 1_{\{W^{(12)}_{n}(a(n)2^j) \leq r\}} } \Big|; \Big|\EE{\{W^{(12)}_{n}(a(n)2^j) - 1 \} 1_{\{W^{(12)}_{n}(a(n)2^j) > r\}} }\Big|\Big\}
$$
\begin{equation}\label{e:order_first_moment_1(Wn>0)}
= O\Big(\frac{a(n)^{h_{1}+h_{2} - 2 h_{12}+1/2}}{\sqrt{n}}\exp\Big\{-\frac{1}{2}\Big(\frac{n}{a^{2(h_{1}+h_{2}) - 4 h_{12}+1}} \Big)^{1 - \xi}\Big\} \Big).
\end{equation}
\end{lemma}
\begin{proof}
Notice that
$$
0 = \EE{ W^{(12)}_{n}(a2^j)   - 1} = \EE{ (W^{(12)}_{n}(a2^j)  - 1) 1_{\{W^{(12)}_{n}(a2^j) > r\}}} + \EE{( W^{(12)}_{n}(a2^j) - 1 ) 1_{\{W^{(12)}_{n}(a2^j) \leq r\}}}.
$$
Hence,
$$
\EE{( W^{(12)}_{n}(a2^j) - 1)  1_{\{W^{(12)}_{n}(a2^j) > r\}}} = - \EE{( W^{(12)}_{n}(a2^j)  - 1 ) 1_{\{W^{(12)}_{n}(a2^j) \leq r\}}}.
$$
However, by the Cauchy-Schwarz inequality,
\begin{equation}\label{e:CS_bound_on_first_moment*1(Wn<0)}
\EE{( W^{(12)}_{n}(a2^j)  - 1) 1_{\{W^{(12)}_{n}(a2^j) \leq r\}}} \Big| \leq \sqrt{ \VAR{\hspace{0.5mm}W^{(12)}_{n}(a2^j)}} \sqrt{P(W^{(12)}_{n}(a2^j) \leq r)}.
\end{equation}
By expressions \eqref{e:CS_bound_on_first_moment*1(Wn<0)}, \eqref{e:Cov(W12,W34)} and Lemma \ref{l:P(WN<e)->0}, the expression \eqref{e:order_first_moment_1(Wn>0)} follows. $\Box$\\
\end{proof}

The following lemma establishes the decay of the covariances between truncated terms \eqref{e:cov_truncated} and indicators involving wavelet variances, or between indicators only.
\begin{lemma}\label{l:Cov(log,1)_Cov(1,1)}
For any $0 < \xi < 1$ and $0 < r < 1/2$,
\begin{itemize}
\item [$(i)$]
$$
\COV{1_{\{|W^{(12)}_{n}(a 2^{j'})| > r\}}}{
1_{\{|W^{(34)}_{n}(a2^{j'})| > r\}}}
$$
\begin{equation}\label{e:Cov(1,1)}
\leq \exp\Big\{ - \frac{1}{2} \Big[\Big( \frac{n}{a^{2(h_{1}+h_{2}) - 4 h_{12}+1}}  \Big)^{1 - \xi}
+ \Big( \frac{n}{a^{2(h_{3}+h_{4}) - 4 h_{34}+1}}  \Big)^{1 - \xi}\Big]\Big\};
\end{equation}
\item [$(ii)$]
$$
\var{\Big[\log |W^{(12)}_n(a(n)2^j)| 1_{\{W^{(12)}_n(a(n)2^j) < - r \}} \Big]}
$$
\begin{equation}\label{e:Var_logW_1(W<-r)}
\leq (C + \log^2(r)) \exp \Big\{- \frac{1}{2}\Big( \frac{n}{a(n)^{2(h_1 + h_2) - 4h_{12}+1}}\Big)^{1 - \xi}\Big\}
\end{equation}
and
$$
\var{ \Big[\log W^{(12)}_n(a(n)2^j) 1_{\{W^{(12)}_n(a(n)2^j) > r \}} \Big]}
$$
\begin{equation}\label{e:Var_logW_1(W>r)}
\leq \log^2(r) \exp \Big\{- \Big( \frac{n}{a(n)^{2(h_1 + h_2) - 4h_{12}+1}}\Big)^{1 - \xi}\Big\} + o(1);
\end{equation}
\item [$(iii)$]
$$
\COV{ \log \hspace{0.5mm} |W^{(12)}_{n}(a 2^j)| \hspace{0.5mm}1_{\{|W^{(1 2)}_{n}(a2^j)| > r\}}}{ 1_{\{|W^{(34)}_{n}(a2^{j'})| > r\}}}
$$
$$
\leq \sqrt{\log^2(r) \exp \Big\{- \Big( \frac{n}{a(n)^{2(h_1 + h_2) - 4h_{12}+1}}\Big)^{1 - \xi} + o(1)}
$$
\begin{equation}\label{e:Cov(log,1)}
\cdot \exp\Big\{ - \frac{1}{2}\Big( \frac{n}{a^{2(h_{3}+h_{4}) - 4 h_{34}+1}} \Big)^{1 - \xi}\Big\}.
\end{equation}
\end{itemize}
In \eqref{e:Var_logW_1(W<-r)}, $C >0$ does not depend on $r$.
\end{lemma}
\begin{proof}
We first show \eqref{e:Cov(1,1)}. By the Cauchy-Schwarz inequality, the left-hand side of \eqref{e:Cov(1,1)} is bounded from above by
\begin{equation}\label{e:sqrt(Var1*Var1)}
\sqrt{\VAR{1_{\{|W^{(1 2)}_{n}(a2^j)| > r\}}}} \sqrt{ \VAR{1_{\{|W^{(3 4)}_{n}(a2^{j'})| > r\}} }}.
\end{equation}
Moreover, for $0 < \xi < 1$ and the octave $j$, Lemma \ref{l:P(WN<e)->0} implies that
$$
\VAR{1_{\{|W^{(1 2)}_{n}(a2^j)| > r\}} } = P(|W^{(12)}_{n}(a2^j)| > r) P(|W^{(1 2)}_{n}(a2^j)| \leq r)
$$
\begin{equation}\label{e:bound_Var(1)}
\leq \exp\Big\{ - \Big( \frac{n}{a^{2(h_{1}+h_{2}) - 4 h_{12}+1}} \Big)^{1 - \xi}\Big\}.
\end{equation}
The same bound holds for the octave $j'$ in \eqref{e:sqrt(Var1*Var1)}. Thus, \eqref{e:Cov(1,1)} holds.

To prove \eqref{e:Var_logW_1(W<-r)}, note that $\VAR{\log |W^{(12)}_n(a2^j)| 1_{\{W^{(12)}_n(a2^j) < - r \}}}$ is bounded by
$$
\EE{ \log^2 |W^{(12)}_n(a2^j)| 1_{\{W^{(1 2)}_{n}(a2^j) < - r\}}}
$$
$$
= \EE{\log^2 |W^{(12)}_n(a2^j)|\Big(1_{\{W^{(1 2)}_{n}(a2^j)\leq - 1/2\}} + 1_{\{-1/2 < W^{(1 2)}_{n}(a2^j) < - r\}}\Big)}
$$
$$
\leq C \EE{|W^{(12)}_n(a2^j)| 1_{\{W^{(1 2)}_{n}(a2^j)\leq - 1/2\}}} + \log^2(r) P(-1/2 < W^{(1 2)}_{n}(a2^j) < - r)
$$
$$
\leq C \sqrt{\EE{W^{(12)}_n(a2^j)^2} P(W^{(1 2)}_{n}(a2^j)\leq - 1/2)} + \log^2(r) P(-1/2 < W^{(1 2)}_{n}(a2^j) < - r)
$$
$$
\leq C' \exp \Big\{ -\frac{1}{2}\Big( \frac{n}{a^{2(h_1 + h_2) - 4h_{12}+1}}\Big)^{1 - \xi}\Big\}
+ \log^2(r) \exp \Big\{ -\Big( \frac{n}{a^{2(h_1 + h_2) - 4h_{12}+1}}\Big)^{1 - \xi}\Big\}.
$$
This establishes \eqref{e:Var_logW_1(W<-r)}.

To prove \eqref{e:Var_logW_1(W>r)}, note that $\VAR{\log W^{(12)}_n(a2^j) 1_{\{W^{(12)}_n(a2^j) > r \}}}$ is bounded by
$$
\EE{ \log^2 |W^{(12)}_n(a2^j)| 1_{\{W^{(1 2)}_{n}(a2^j) > r\}}}
$$
$$
= \EE{\log^2 |W^{(12)}_n(a2^j)|\Big(1_{\{r < W^{(1 2)}_{n}(a2^j) < 1/2\}}+ 1_{\{W^{(1 2)}_{n}(a2^j)\geq 1/2\}}\Big)}
$$
$$
\leq \log^2(r) P(r < W^{(1 2)}_{n}(a2^j) < 1/2) + \EE{\log^2 |W^{(12)}_n(a2^j)| 1_{\{W^{(1 2)}_{n}(a2^j)\geq 1/2\}}}.
$$
However, for some $C > 0$,
$$
\log^2 |W^{(12)}_n(a2^j)| 1_{\{W^{(1 2)}_{n}(a2^j)\geq  1/2\}} \leq C W^{(12)}_n(a2^j)^2,
$$
where, by \eqref{e:W12->1_in_L2}, $W^{(12)}_n(a2^j) \stackrel{P}\rightarrow 1$ and
$$
\EE{W^{(12)}_n(a2^j)^2} = \var{W^{(12)}_n(a2^j)} + 1 \rightarrow 1, \quad n \rightarrow \infty.
$$
Therefore, by the dominated convergence theorem for convergence in probability,
$$
\lim_{n \rightarrow \infty}\EE{\log^2 |W^{(12)}_n(a2^j)| 1_{\{W^{(1 2)}_{n}(a2^j)\geq  1/2\}}} = 0.
$$
This establishes \eqref{e:Var_logW_1(W>r)}.

To show \eqref{e:Cov(log,1)}, again by applying the Cauchy-Schwarz inequality, the left-hand side of \eqref{e:Cov(log,1)} is bounded from above by
\begin{equation}\label{e:sqrt(Var)*sqrt(Var_1)}
\sqrt{\VAR{\log|W^{(12)}_n(a2^j)|1_{\{|W^{(1 2)}_{n}(a2^j)| > r\}} }}\sqrt{\VAR{1_{\{|W^{(34)}_{n}(a2^{j'})| > r\}} }}.
\end{equation}
Hence, the bound follows from \eqref{e:Cov(1,1)} and \eqref{e:Var_logW_1(W>r)}. $\Box$\\
\end{proof}

We are now in a position to establish Theorem \ref{t:Cov(log,log)_cross}.\\

\noindent {\sc Proof of Theorem~\ref{t:Cov(log,log)_cross}}: Fix $0 < \xi < 1$ and recall that $n_* = \frac{n}{a2^{j+j'}}$. Then,
$$
\COV{ \log |S^{(12)}_{n}(a2^j)| \hspace{0.5mm}1_{\{|W^{(12)}_{n}(a2^j)| > r_n\}} }{ \log |S^{(34)}_{n}(a2^{j'})| \hspace{0.5mm} 1_{\{|W^{(34)}_{n}(a2^{j'})| > r_n\}}}
$$
$$
= \COV{\log |W^{(12)}_{n}(a2^j)| \hspace{0.5mm} 1_{\{|W^{(12)}_{n}(a2^j)| > r_n\}} }{ \log |W^{(34)}_{n}(a2^{j'})|\hspace{0.5mm} 1_{\{|W^{(34)}_{n}(a2^{j'})| > r_n\}}}
$$
$$
+ \log \big|\EE{d_{3}(a2^{j},0)d_{4}(a2^{j'},0)}\big|\hspace{0.5mm} \COV{ \log |W^{(12)}_{n}(a 2^j)|\hspace{0.5mm} 1_{\{|W^{(12)}_{n}(a2^j)| > r_n\}} }{ 1_{\{|W^{(34)}_{n}(a2^{j'})| > r_n\}}}
$$
$$
+ \log \big|\EE{d_{1}(a2^{j},0)d_{2}(a2^{j},0)}\big| \hspace{0.5mm} \COV{ 1_{\{|W^{(12)}_{n}(a 2^{j})| > r_n\}} }{ \log |W^{(34)}_{n}(a2^{j'})| \hspace{0.5mm} 1_{\{|W^{(34)}_{n}(a2^{j'})| > r_n\}}}
$$
$$
+ \log \big|\EE{d_{1}(a2^{j},0)d_{2}(a2^{j},0)}\big| \log \big|\EE{d_{3}(a2^{j'},0)d_{4}(a2^{j'},0)}\big|
$$
$$
\COV{ 1_{\{|W^{(12)}_{n}(a 2^{j})| > r_n\}} }{ 1_{\{|W^{(34)}_{n}(a2^{j'})| > r_n\}}}
$$
$$
 = \COV{ \log |W^{(12)}_{n}(a2^j)| 1_{\{|W^{(12)}_{n}(a2^j)| > r_n\}} }{ \log |W^{(34)}_{n}(a(n)2^{j'})| 1_{\{|W^{(34)}_{n}(a(n)2^{j'})| > r_n\}} }
 $$
\begin{equation}\label{e:Cov(logW,logW)}
+ o\Big( \Big(\frac{a(n)^{4 h_{\max} - 4 h_{\min}}}{n_*}\Big)^2\Big).
\end{equation}
The last two equalities in \eqref{e:Cov(logW,logW)} are a consequence of Lemma \ref{l:cov_and_eigenstructure_wavelet_transf_and_var}, $(ii)$, as applied to $\EE{d_{1}(a2^{j},0)d_{2}(a2^{j},0)}$ and $\EE{d_{3}(a2^{j'},0)d_{4}(a2^{j'},0)}$, and of the bound \eqref{e:Cov(log,1)} (from Lemma \ref{l:Cov(log,1)_Cov(1,1)}, $(iii)$) under the condition \eqref{e:rn->0}.

Therefore, it suffices to show that the main term on the right-hand side of \eqref{e:Cov(logW,logW)} is equal to the main term on the right-hand side of \eqref{e:Cov(log,log)_cross}. By accounting for absolute values, the covariance term in the former can be broken up into a sum of four terms, namely,
\begin{equation}\label{e:Cov(logW12,logW34)}
\COV{ \log W^{(12)}_{n}(a2^j)1_{\{W^{(12)}_{n}(a2^j) > r_n\}} }{ \log W^{(34)}_{n}(a2^{j'})1_{\{W^{(34)}_{n}(a2^{j'}) > r_n\}} }
\end{equation}
plus the remainder
$$
\COV{ \log W^{(12)}_{n}(a2^j) \hspace{0.5mm} 1_{\{W^{(12)}_{n}(a2^j) > r_n\}} }{ \log |W^{(34)}_{n}(a2^j)| \hspace{0.5mm} 1_{\{W^{(34)}_{n}(a2^j) < - r_n\}} }
$$
$$
+ \COV{ \log |W^{(12)}_{n}(a2^j)| \hspace{0.5mm} 1_{\{W^{(12)}_{n}(a2^j) < -r_n\}} }{ \log W^{(34)}_{n}(a2^j) \hspace{0.5mm} 1_{\{W^{(34)}_{n}(a2^j) > r_n\}} }
$$
\begin{equation}\label{e:proof_theo_CI_residual_sum_cov}
+ \COV{ \log |W^{(12)}_{n}(a2^j)| \hspace{0.5mm} 1_{\{W^{(12)}_{n}(a2^j) < -r_n\}} }{ \log |W^{(34)}_{n}(a2^j)| \hspace{0.5mm} 1_{\{W^{(34)}_{n}(a2^j) < - r_n\}} }
\end{equation}
By the Cauchy-Schwarz inequality, the bounds \eqref{e:Var_logW_1(W<-r)} and \eqref{e:Var_logW_1(W>r)} (from Lemma \ref{l:Cov(log,1)_Cov(1,1)}, $(ii)$) and condition \eqref{e:rn->0}, the absolute value of the second term in the sum \eqref{e:proof_theo_CI_residual_sum_cov} is bounded by
$$
\sqrt{\VAR{\log |W^{(12)}_n(a2^j)|1_{\{W^{(12)}_n(a2^j)< - r_n\}}}\VAR{ \log W^{(34)}_n(a2^j) 1_{\{W^{(34)}_n(a2^j)>  r_n\}}}}
$$
$$
= o\Big( \Big(\frac{a(n)^{4 h_{\max} - 4 h_{\min}}}{n_*}\Big)^2\Big).
$$
By a similar argument, the same bound holds for the remaining terms in the sum \eqref{e:proof_theo_CI_residual_sum_cov}.
Thus, it suffices to focus on \eqref{e:Cov(logW12,logW34)}. In the following developments, expressions involving individual sample wavelet variance terms will be expressed in terms of $W^{(12)}_n(a2^j)$, but analogous expressions hold when substituting $W^{(34)}_n(a2^{j'})$ for $W^{(12)}_n(a2^j)$.

Fix $0< \xi< 1$. For any given $j $, write out the almost sure Taylor expansion
$$
\log W^{(12)}_{n}(a2^j) \hspace{0.5mm} 1_{\{W^{(12)}_{n}(a2^j) > r_n\}}
$$
\begin{equation}\label{e:Taylor}
= \Big\{ (W^{(12)}_{n}(a2^j) - 1 ) - \frac{1}{2} \Big(\frac{W^{(12)}_{n}(a2^j)- 1 }{\theta^2_+(W^{(12)}_{n}(a2^j))}\Big)^2\Big\} 1_{\{W^{(12)}_{n}(a2^j) > r_n\}},
\end{equation}
where $\theta_+(W^{(12)}_{n}(a2^j)) \in [\min\{W^{(12)}_{n}(a2^j),1\},\max\{W^{(12)}_{n}(a2^j),1\}]$. Then,
$$
\EE{\log W^{(12)}_{n}(a2^j) \hspace{1mm} \log W^{(34)}_{n}(a2^{j'}) \hspace{1mm} 1_{\min\{\{W^{(12)}_{n}(a2^{j}),W^{(34)}_{n}(a2^{j'})\} > r_n\}}}
$$
$$
=  \EE{(W^{(12)}_{n}(a2^j) - 1 )(W^{(34)}_{n}(a2^{j'}) - 1 ) 1_{\{\min\{W^{(12)}_{n}(a2^{j}),W^{(34)}_{n}(a2^{j'})\} > r_n\}}}
$$
$$
- \frac{1}{2}\EE{(W^{(12)}_{n}(a2^j) - 1 )\Big(\frac{W^{(34)}_{n}(a2^{j'})- 1 }{\theta_+(W^{(34)}_{n}(a2^{j'}))}\Big)^2 1_{\{\min\{W^{(12)}_{n}(a2^{j}),W^{(34)}_{n}(a2^{j'})\} > r_n\}}}
$$
$$
- \frac{1}{2} \EE{\Big(\frac{W^{(12)}_{n}(a2^j)- 1}{\theta_+(W^{(12)}_{n}(a2^j))}  \Big)^2 (W^{(34)}_{n}(a2^{j'}) - 1 )\hspace{1mm}1_{\{\min\{W^{(12)}_{n}(a2^{j}),W^{(34)}_{n}(a2^{j'})\} > r_n\}} }
$$
\begin{equation}\label{e:main_proof_basic_expression_for_the_cross_moment}
+ \frac{1}{4} \EE{ \Big(\frac{W^{(12)}_{n}(a2^j)- 1 }{\theta_+(W^{(12)}_{n}(a2^{j}))}\Big)^2 \Big(\frac{W^{(34)}_{n}(a2^{j'})- 1 }{\theta_+(W^{(34)}_{n}(a2^{j'}))}\Big)^2 1_{\{\min\{W^{(12)}_{n}(a2^{j}),W^{(34)}_{n}(a2^{j'})\} > r_n\}} }.
\end{equation}
For $0 < r_n < 1/2$, recast
$$
\Big(\frac{W^{(12)}_{n}(a2^j) - 1}{\widehat{\theta}_+(W^{(12)}_{n}(a2^j))} \Big)^2  1_{\{W^{(12)}_{n}(a2^j) > r_n \}}
$$
$$
= \Big(\frac{W^{(12)}_{n}(a2^j) - 1}{\widehat{\theta}_+(W^{(12)}_{n}(a2^j))} \Big)^2  \Big(1_{\{r_n < W^{(12)}_{n}(a2^j) < 1/2\}}
+ 1_{\{W^{(12)}_{n}(a2^j) \geq 1/2\}}\Big)
$$
\begin{equation}\label{e:bound_remainder_+}
\leq \Big(\frac{W^{(12)}_{n}(a2^j) - 1}{r_n} \Big)^2   1_{\{r_n < W^{(12)}_{n}(a2^j)  < 1/2\}}
+ \Big(\frac{W^{(12)}_{n}(a2^j) - 1}{1/2} \Big)^2 1_{\{W^{(12)}_{n}(a2^j) \geq 1/2\}}.
\end{equation}
Therefore, we can rewrite the fourth term in \eqref{e:main_proof_basic_expression_for_the_cross_moment} as
$$
\EE{\Big|\Big(\frac{W^{(12)}_{n}(a2^{j})- 1}{\theta_+(W^{(12)}_{n}(a2^{j}))} \Big)^{2}\Big(\frac{W^{(34)}_{n}(a2^{j})- 1}{\theta_+(W^{(34)}_{n}(a2^{j}))} \Big)^{2}1_{\{\min\{W^{(12)}_{n}(a2^j),W^{(34)}_{n}(a2^{j'})\} > r_n\}}\Big|}
$$
$$
\leq  \frac{1}{r^4_n}  \EE{(W^{(12)}_{n}(a2^j) - 1 )^{2}1_{\{r_n < W^{(12)}_{n}(a2^j) < 1/2\}} (W^{(34)}_{n}(a2^{j'})- 1 )^{2}1_{\{r_n < W^{(34)}_{n}(a2^j) < 1/2\}}}
$$
$$
+ \frac{1}{(r_n/2)^2} \EE{(W^{(12)}_{n}(a2^j) - 1 )^{2} 1_{\{W^{(12)}_{n}(a2^j) \geq 1/2\}}(W^{(34)}_{n}(a2^{j'})- 1 )^{2}1_{\{r_n < W^{(34)}_{n}(a2^j) < 1/2\}}}
$$
$$
+ \frac{1}{(r_n/2)^2} \EE{(W^{(12)}_{n}(a2^j) - 1 )^{2} 1_{\{r_n < W^{(12)}_{n}(a2^j) < 1/2\}}(W^{(34)}_{n}(a2^{j'})- 1 )^{2}1_{\{W^{(34)}_{n}(a2^j) \geq 1/2\}}}
$$
\begin{equation}\label{e:main_proof_product_fourth_order_terms}
+ \frac{1}{(1/2)^4} \EE{(W^{(12)}_{n}(a2^j) - 1 )^{2} 1_{\{W^{(12)}_{n}(a2^j) \geq 1/2\}}(W^{(34)}_{n}(a2^{j'})- 1 )^{2}1_{\{W^{(34)}_{n}(a2^{j'}) \geq 1/2\}}}.
\end{equation}
By Lemma \ref{l:third_order_cross_term=O(1/n^2)}, the fourth term term in the sum is bounded by
\begin{equation}\label{e:main_proof_4th_term}
O\Big( \Big( \frac{a^{4 h_{\max} - 4 h_{\min}+1}}{n}\Big)^2\Big).
\end{equation}
By the Cauchy-Schwarz inequality, Lemma \ref{l:third_order_cross_term=O(1/n^2)} and condition \eqref{e:rn->0}, the first term in the sum \eqref{e:main_proof_product_fourth_order_terms} is bounded by
$$
\frac{a^{2 (\delta_{12}+\delta_{34})}}{r^4_n}  \sqrt{\EE{\Big( \frac{W^{(12)}_{n}(a2^j) - 1}{a^{\delta_{12}}} \Big)^{4} \Big( \frac{W^{(34)}_{n}(a2^{j'})- 1}{a^{\delta_{34}}} \Big)^{4}}}
$$ 
$$
\cdot \sqrt{ \EE{1_{\{r_n < W^{(12)}_{n}(a2^j) < 1/2\}}1_{\{r_n < W^{(34)}_{n}(a2^j) < 1/2\}} }  }
$$
$$
\leq \frac{a^{2 (\delta_{12}+\delta_{34})}}{r^4_n}  O\Big( \frac{a}{n}\Big) \sqrt{P(r_n < W^{(12)}_{n}(a2^j) < 1/2)P(r_n < W^{(34)}_{n}(a2^j) < 1/2)}
$$
$$
\leq \frac{a^{2 (\delta_{12}+\delta_{34})}}{r^4_n}  O\Big( \frac{a}{n}\Big) \exp\Big\{ - \frac{1}{2}\Big[\Big( \frac{n}{a^{2(h_1 + h_2) - 4 h_{12}+1}}\Big)^{1 - \xi}+\Big( \frac{n}{a^{2(h_3 + h_4) - 4 h_{34}+1}}\Big)^{1 - \xi}\Big]\Big\}
$$
\begin{equation}\label{e:main_proof_1st_term}
= O\Big( \Big( \frac{a^{4 h_{\max} - 4 h_{\min}+1}}{n}\Big)^2\Big),
\end{equation}
since $2 (\delta_{12}+\delta_{34}) \leq 4 h_{\max} - 4 h_{\min}$. The second term in the sum \eqref{e:main_proof_product_fourth_order_terms} is bounded by
$$
\frac{C a^{2 (\delta_{12}+\delta_{34})}}{r^2_n} \sqrt{\EE{\Big( \frac{W^{(12)}_{n}(a2^j) - 1}{a^{\delta_{12}}} \Big)^{4} \Big(\frac{W^{(34)}_{n}(a2^{j'})- 1 }{a^{\delta_{34}}}\Big)^{4}}}
$$
$$ 
\cdot \sqrt{ \EE{1_{\{W^{(12)}_{n}(a2^j) \geq 1/2\}}1_{\{r_n < W^{(34)}_{n}(a2^j) < 1/2\}}}}
$$
$$
\leq \frac{Ca^{2 (\delta_{12}+\delta_{34})}}{r^2_n} O\Big( \frac{a}{n}\Big) \sqrt{P(W^{(12)}_{n}(a2^j) \geq 1/2)P(r_n < W^{(34)}_{n}(a2^j) < 1/2)}
$$
\begin{equation}\label{e:main_proof_2nd_term}
\leq \frac{Ca^{2 (\delta_{12}+\delta_{34})}}{r^2_n} O\Big( \frac{a}{n}\Big) \exp\Big\{ - \frac{1}{2}\Big( \frac{n}{a^{2(h_3 + h_4) - 4 h_{34}+1}}\Big)^{1 - \xi}\Big\}
= O\Big( \Big( \frac{a^{4 h_{\max} - 4 h_{\min}+1}}{n}\Big)^2\Big).
\end{equation}
An analogous bound holds for the third term in the sum \eqref{e:main_proof_product_fourth_order_terms}. Therefore, by \eqref{e:main_proof_basic_expression_for_the_cross_moment}, \eqref{e:main_proof_4th_term}, \eqref{e:main_proof_1st_term}, \eqref{e:main_proof_2nd_term} and Lemma \ref{l:E(Wnj-1)(Wnj'-1)}, we conclude that \eqref{e:Cov(logW,logW)} is equal to the right-hand side of \eqref{e:Cov(log,log)_cross}, as claimed. $\Box$\\

\begin{remark}\label{r:q=q1=q2=>trunc_is_unnecessary}
For $q = q_1 = q_2$, $W^{(qq)}_{n,-}(a2^j) = 0$ a.s. (see \eqref{e:sum_W=sum_lambda*Z2-lambda*Z2}). Then, the existence of the moment $\EE{\log^l |W^{(qq)}_n(a2^j)|}$, $l \in \bbN$, can be directly established by applying relation (96) in \cite{moulines:roueff:taqqu:2007:spectral}. Moreover, the analysis of moments in this section can be extended without the truncation based on the sequence \eqref{e:rn->0}.
\end{remark}

\end{document}